\documentclass[11pt]{article}
\usepackage{}
\usepackage{enumerate}
\usepackage{mathbbold}
\usepackage{bbm}
\usepackage{amsfonts}
\usepackage[T1]{fontenc}
\usepackage{latexsym,amssymb,amsmath,amsfonts,amsthm}
\usepackage{graphicx}
\usepackage{subfigure}
\usepackage{epsf}
\usepackage{epstopdf}
\usepackage{amssymb}
\usepackage{mathrsfs}
\usepackage{xcolor}

\renewcommand{\theequation}{\arabic{section}.\arabic{equation}}

\newcommand{\R}{{\mathbb R}}

\newcommand{\C}{{\mathbb C}}

\newcommand{\be}{\begin{eqnarray}}
\newcommand{\ben}{\begin{eqnarray*}}
\newcommand{\en}{\end{eqnarray}}
\newcommand{\enn}{\end{eqnarray*}}
\newcommand{\ba}{\backslash}
\newcommand{\pa}{\partial}

\newcommand{\ov}{\overline}

\newcommand{\g}{\gamma}

\newcommand{\eps}{\epsilon}

\newcommand{\hx}{\hat{x}}
\newtheorem{theorem}{Theorem}[section]
\newtheorem{lemma}[theorem]{Lemma}

\newtheorem{remark}[theorem]{Remark}

\newtheorem{algorithm}{Data Completion Algorithm}

\newtheorem{notation}[theorem]{Notation}

\begin{document}
\title{\bf Data completion algorithms and their applications in inverse acoustic scattering with limited-aperture backscattering data}
\author{Fangfang Dou\thanks{School of Mathematical Sciences, University of Electronic Science and Technology of China, Chengdu 611731, China. Email: fangfdou@uestc.edu.cn }
\and
Xiaodong Liu\thanks{Academy of Mathematics and Systems Science,
Chinese Academy of Sciences, Beijing 100190, China. Email: xdliu@amt.ac.cn }
\and
Shixu Meng\thanks{Academy of Mathematics and Systems Science, Chinese Academy of Sciences,
Beijing 100190, China. Email: shixumeng@amss.ac.cn }
\and
Bo Zhang\thanks{Academy of Mathematics and Systems Science, Chinese Academy of Sciences,
Beijing 100190, China and School of Mathematical Sciences, University of Chinese Academy of Sciences,
Beijing 100049, China. Email: b.zhang@amt.ac.cn }}
\date{}
\maketitle

\begin{abstract}
We introduce two data completion algorithms for the limited-aperture problems in inverse acoustic scattering. Both completion algorithms are independent of the topological and physical properties of the unknown scatterers. The main idea is to relate the limited-aperture data to the full-aperture data via the prolate matrix. The data completion algorithms are simple and fast since only the approximate inversion of the prolate matrix is involved. We then combine the data completion algorithms with imaging methods such as factorization method and direct sampling method for the object reconstructions.
A variety of numerical examples are presented to illustrate the effectiveness and robustness of the proposed algorithms.

\vspace{.2in}

\vspace{.2in} {\bf AMS subject classifications:}
35P25, 45Q05, 78A46, 74B05

\end{abstract}

\section{Introduction}
The inverse scattering problems aim to determine the unknown objects from the measured wave fields. In many cases of practical interest, it is difficult or even impossible to obtain the {\em full-aperture} measurements all around the unknown objects, this motivates us to consider the so-called {\em limited-aperture} problems. The {  limited-aperture} problems are fundamental in applications such as radar, sonar, geophysical exploration, medical imaging and nondestructive testing. Various reconstruction algorithms have been developed
\cite{AhnJeonMaPark, BaoLiu,ChengPengYamamoto2005IP, ColtonMonk06, IkehataNiemiSiltanen, KirschGrinberg,L4,LuXuXu2012AA,MagerBleistein,MagerBleistein1978,Robert1987, Zinn1989} using the {limited-aperture} data directly.
{Although uniqueness of the inverse scattering problems can be proved in some cases \cite{CK}}, the performance of the reconstruction algorithms is not entirely satisfactory.
A typical feature is that the ``shadow region'' is elongated in down range \cite{ColtonMonk06,L4}.
Physically, the information from the ``shadow region'' is very weak, especially for high frequency waves \cite{MagerBleistein}.
For the two-dimensional problems, the numerical
experiments of the decomposition methods  in \cite{Robert1987,Zinn1989} indicate that satisfactory reconstructions need an aperture not smaller than $180$ degrees.

Other than directly using the {limited-aperture} data, one may first recover the {full-aperture} data  or process the {limited-aperture} data in appropriate ways.
From the perspective of recovering the data, it is numerically difficult to recover the {full-aperture} data since analytic continuation problems are severely ill-posed in general \cite{Atkinson}. It is possible to design stable regularization methods for some specific problems.
We refer to \cite{FuDouFengQian,FuDengFengDou} for  stable regularization methods on analytic continuation to a strip domain with data given only on a line.
We also refer to  \cite{ChengYamamoto1998IP, LuXuXu2012AA} for a conditional stability estimate on a line or an analytic curve.
By considering integral equation methods for solving the inverse scattering problems, one may also look for density functions of layer potentials that generate the measured data approximately by regularization, and then apply the density to obtain the {full-aperture} data \cite{LiuSun19, LuXuXu2012AA}.
From the perspective of processing the data in appropriate ways, matched filters and related filters have been shown to be practical and powerful when dealing with {limited-aperture} data \cite{BORCEA2019556,Borcea_2009,CB2009}. These filters are usually applied to the data before applying imaging methods. It is possible to see that processing the data in appropriate ways  is  the same as  recovering the {full-aperture} data.
Based on propagating modal formulation in a waveguide, \cite{BORCEA2019556} showed how to recover full-aperture data from limited-aperture data and applied such data in the implementation of the linear sampling method, where the idea is to relate the limited-aperture data to the full-aperture data via an ill-conditioned matrix.

In this paper we introduce two data completion algorithms in two dimensions.
In the first algorithm, we represent the full-aperture data in the form of double Fourier series, and find that the corresponding Fourier coefficients are related to the limited-aperture data via two prolate matrices \cite{Slepian78}. In the second algorithm, for each incident direction, we represent the far field data by the single layer potential, and find that the  Fourier coefficients of the density are related to the  limited-aperture data via a prolate matrix (which turns out to be same for every incident direction).
 The data completion algorithms are simple and fast since only the approximate inversion of the prolate matrix is involved. As an application, we then combine the data completion algorithms with imaging methods such as factorization method and direct sampling methods for the object reconstructions.
We remark that both the data completion techniques and the sampling methods are independent of the physical and geometrical properties of the unknown objects. As a final remark, we restrict ourselves to the two dimensional case though extension to three dimensional case is similar yet to be done.

The paper is further organized as follows. In the next section, we introduce the mathematical model for the  inverse acoustic scattering with limited-aperture backscattering data. In Section \ref{section prolate MSR}, we collect properties of the prolate matrix and discuss the backscattering multi-static response matrix to facilitate the data completion algorithms later on.
Section \ref{Data completion} is devoted  to the data completion algorithms.
In Section \ref{DSMs}, we introduce the object reconstruction algorithm by combining the proposed data completion algorithms and the well known sampling methods.
Finally, a variety of numerical examples are presented to illustrate the effectiveness and robustness of the proposed algorithms in Section \ref{NumExamples}.



\section{Mathematical model}
Let $D\subset\R^2$ be an open and bounded domain with Lipschitz boundary $\pa D$ such that the exterior $\R^2\ba\ov{D}$ of $\ov{D}$ is connected. Here and throughout the paper we denote by $\ov{D}$ the closure of the set $D$ of points in $\R^2$. A confusion with the complex conjugate $\ov{z}$ of $z\in\C$ is not expected. Furthermore, let $k>0$ be the wavenumber and $\mathbb{S}:=\{(\cos\theta, \sin\theta)^{\rm T}:\,\theta\in [-\pi, \pi]\}$ denote the unit circle in $\R^2$. The  incident field of our particular interest is the plane wave
\ben
u^{in}(x,\theta_d):=e^{ikx\cdot d},\quad x\in\R^2
\enn
with incident direction $d=(\cos\theta_d, \sin\theta_d)^{\rm T}\in\mathbb{S}$.
The scatterer $D$ gives rise to a scattered field $u^{sc}(x)$ satisfying
\be
\label{HemEquobstacle}\Delta u^{sc} + k^2 u^{sc} = 0\quad \mbox{in }\R^2\ba\ov{D},\\
\label{Dbc} u^{sc} = -u^{in}\quad\mbox{on }\pa D,\\
\label{Srcobstacle}\lim_{r:=|x|\rightarrow\infty}r^{\frac{1}{2}}\left(\frac{\pa u^{sc}}{\pa r}-iku^{sc}\right) =\,0.
\en
The well-posedness of the direct scattering problems \eqref{HemEquobstacle}-\eqref{Srcobstacle} can be found in the standard monograph \cite{CK}.

Note that we have restrict ourselves to the simplest case with the Dirichlet boundary condition \eqref{Dbc} corresponding to a sound-soft obstacle. Boundary conditions other than \eqref{Dbc} can also be considered, for example the Neumann boundary condition
or the impedance boundary condition. The scatterer can also be an inhomogeneous medium. These boundary conditions represent the properties of the underlying scatterer, which are often not known in advance in practical situations. Our analyses and methods in the subsequent sections will be seen to be independent of these physical  properties.

Every radiating solution of the Helmholtz equation has the following asymptotic
behavior at infinity \cite{KirschGrinberg}
\be\label{0asyrep}
u^{sc}(x)
=\frac{e^{i\frac{\pi}{4}}}{\sqrt{8k\pi}}\frac{e^{ikr}}{\sqrt{r}}\left\{u^{\infty}(\theta_{\hat{x}})+\mathcal{O}\left(\frac{1}{r}\right)\right\}\quad\mbox{as }\,r:=|x|\rightarrow\infty,
\en
uniformly with respect to all directions $\hx:=x/|x|\in\mathbb{S}$. Here, we identify $\hx=(\cos\theta_{\hx}, \sin\theta_{\hx})^{\rm T}$ with $\theta_{\hx}\in [-\pi, \pi]$.
The complex valued function $u^{\infty}(\hx)$ defined on $\mathbb{S}$ is known as the scattering amplitude or the far field pattern with $\hx\in\mathbb{S}$ denoting the observation direction.

For the incident plane waves $u^{in}(x,\theta_d)$ we denote the scattered field by $u^{sc}(x,\theta_d)$ and its far field pattern by $u^{\infty}(\theta_{\hx},\theta_d)$. Then the inverse scattering problem we consider in this paper is to identify the obstacle $D$ from the following limited-aperture "backscattering" data

$$
\{u^{\infty}(\theta_{\hx},\theta_d): \quad \theta_{\hx}\in [-\alpha, \alpha],\, \theta_d\in [\pi-\alpha, \pi+\alpha],\, \alpha\in (0, \pi]\}.
$$

\section{Prolate matrix and backscattering multi-static response matrix}\label{section prolate MSR}

In this section, we collect properties of the prolate matrix  which  play an important role in the data completion algorithms. We also discuss the backscattering multi-static response matrix to facilitate the data completion algorithms later on.

\begin{notation}
We are going to work with matrices and vectors with negative indices for notational convenience. For $-N_1\le n\le N_2$, we denote by
$$
b:=\Big(b_n\Big), \quad -N_1\le n\le N_2
$$
as a $N_1+N_2+1$ dimensional vector.

For $-M_1\le m\le M_2$ and $-N_1\le n\le N_2$, we denote by
$$
A:=\Big(A_{mn}\Big), \quad -M_1\le m\le M_2, -N_1\le n\le N_2
$$
as a $(M_1+M_2+1) \times (N_1+N_2+1)$ dimensional matrix.

\end{notation}

\subsection{Prolate Matrix}
\setcounter{equation}{0}
Let us denote
\ben
\phi_m(\theta) := \frac{1}{\sqrt{2\pi}}e^{im\theta}, \quad m=0, \pm1, \pm2,\cdots
\enn
which form a complete orthonormal basis in $L^2(-\pi, \pi)$.

For fixed $\alpha\in (0, \pi]$, we first introduce a $N\times N$ matrix $\mathbb{P}=\mathbb{P}(\alpha)$ whose $mn$-th entry is given by
\be \label{def Prolate Matrix}
\mathbbm{p}_{mn}
&:=& \int_{-\alpha}^\alpha \phi_m(\theta) \overline{\phi_n(\theta)} d \theta \cr
&=& \frac{1}{2\pi}\int_{-\alpha}^\alpha e^{i(m-n)\theta} d \theta \cr
&=&
\bigg\{
\begin{array}{cc}
\frac{\alpha}{\pi},  &  m=n    \\
\frac{\sin((m-n)\alpha)}{\pi (m-n)}, &     m\not=n
\end{array}
.
\en
{In particular, $\mathbb{P}(\pi)$ is the identity matrix.}
Following  \cite{Varah1993}, we conveniently set $\mathbb{P}^\infty$ as the infinite matrix when $N=\infty$. This is the well known prolate matrix which also appears in signal processing \cite{Slepian78,Varah1993}.
We collect some  properties \cite{Slepian78,Varah1993} of the prolate matrix $\mathbb{P}(\alpha)$ in the following lemma.
\begin{lemma}\label{Mproperties}
For $\alpha\in(0,\pi)$, let $\mathbb{P}(\alpha)$ be the  $(2J+1)\times(2J+1)$ prolate matrix, and $\{\sigma_j=\sigma_j(\alpha)\}_{j=-J}^J$ be the eigenvalues of $\mathbb{P}(\alpha)$ numbered in decreasing order with corresponding normalized eigenfunctions $\{\xi_j\}_{j=-J}^J$.
\begin{itemize}
\item $1>\sigma_{-J}>\sigma_{-J+1}>\cdots>\sigma_{J}>0$;
\item Let $\psi(a, b)$ be the number of eigenvalues of $\mathbb{P}(\alpha)$ contained in $[a,b]\subset[0,1]$. Then
     \ben
     \lim_{J\rightarrow \infty}\frac{\psi(a,b)}{2J+1} = \frac{\varphi(a,b)}{2\pi},
     \enn
     where $\varphi(a,b)$ is the measure of the set of $\theta's$ with $a\leq g_{\alpha}(\theta)\leq b$, and
\ben
g_{\alpha}(\theta) := \left\{
                        \begin{array}{ll}
                          1, & \hbox{$0\leq|\theta|\leq \alpha$;} \\
                          0, & \hbox{$\pi-\alpha\leq|\theta|\leq \pi$.}
                        \end{array}
                      \right.
\enn
 In particular, for any $0<a<b<1$, $\varphi(a,b)=0$, i.e., the eigenvalues $\{\sigma_j\}$ cluster near $1$ and $0$;
\item The eigenvalues satisfy the following symmetric property:
            \ben
      \sigma_j(\alpha) = 1-\sigma_{-j}(\pi-\alpha),\quad j=-J,-J+1,\cdots, J;
      \enn
\item Let $N=2J+1$, then as $J$ becomes large,
     \ben
     \sigma_{N}(\alpha) \cong C(N,\alpha)e^{-\g N},
     \enn
     where
     \ben
     \g := \log\left(\frac{\sqrt{2}+\sqrt{1-\cos\alpha}}{\sqrt{2}-\sqrt{1-\cos\alpha}}\right), \quad
     C(N,\alpha) := \sqrt{\pi}2^{9/4}(1-\cos\alpha)^{1/4}(1+\cos\alpha)^{-1/2}N^{1/2}.
     \enn
      Correspondingly, the spectral condition number is
      \ben
      \kappa(\mathbb{P}(\alpha)) \cong \frac{e^{\g N}}{C(N,\alpha)}.
      \enn
\end{itemize}
\end{lemma}
The proof can be found in the classical references, e.g., \cite{GrenanderSzego, Slepian78,Varah1993}. As can be seen from Lemma \ref{Mproperties}, the prolate matrix $\mathbb{P}$ is very ill-conditioned for $\alpha\in (0, \pi)$. This brings a great challenge to compute its inverse, which will be a key step of the subsequent data completion algorithms.

To compute the inverse of $\mathbb{P}$, regularization techniques have to be used. Let the singular value decomposition of the symmetric real-valued matrix $\mathbb{P}$ be
\begin{equation}
\mathbb{P} = \mathbb{U} \Sigma \mathbb{U}^T,
\end{equation}
with $\mathbb{U}\mathbb{U}^T=I$ (where $I$ is the identity matrix) and $\Sigma=\mbox{diag}(\sigma_{-J},\sigma_{-J+1},\cdots,\sigma_{J})$.

We seek to find an approximate inverse of $\mathbb{P}$, denoted by $\mathbb{P}^\dagger:=\mathbb{U} \Sigma^\dagger \mathbb{U}^*$, where $\Sigma^\dagger$ depends on the regularization technique.

\textbf{Regularization I}: we may first propose an approximation of $\mathbb{P}^{-1}$ based on truncated SVD as follows. Let $\sigma_{-J+J_\alpha},\cdots,\sigma_{J}$ be the eigenvalues that are clustered near $0$ (where the positive integer $J_\alpha$ is determined by $J$ and the aperture $(-\alpha,\alpha)$), we set
\begin{equation} \label{data completion TSVD sigma}
\Sigma^\dagger = \mbox{diag} (\sigma^{-1}_{-J},\sigma^{-1}_{-J+1},\cdots,\sigma^{-1}_{-J+J_\alpha-1},0,\cdots,0),
\end{equation}
so that
\begin{equation} \label{data completion TSVD}
\mathbb{P}^\dagger = \mathbb{U} \Sigma^\dagger \mathbb{U}^*.
\end{equation}
Later on, we use $\mathbb{P}^\dagger$ to approximate $\mathbb{P}^{-1}$ in the data completion procedure.

The physical intuition is related to Slepian's spheroidal wave functions: By introducing $\xi_m:=\sum_{j=-J}^JU_{jm} \phi_j$ we can find prolate spheroidal wave functions $\{\xi_m\}_{m=-J}^{J}$ such that
\begin{equation*}
\int_{-\alpha}^\alpha \xi_m(\theta) \overline{\xi_n} (\theta) = \delta_{mn} \sigma_m, \qquad \int_{-\pi}^\pi \xi_m(\theta) \overline{\xi_n} (\theta) = \delta_{mn}, \quad -J\le m,n\le J,
\end{equation*}
where $\delta_{mn}$ denotes the Kronecker delta.
This means that, when $\sigma_m$ is close to $1$, the ``principal energy'' of $\xi_m$ is on the interval $(-\alpha,\alpha)$; and when $\sigma_m$ is close to $0$, the ``principal energy'' of $\xi_m$ is on the interval $(-\pi,\pi)\backslash [-\alpha,\alpha]$. The data completion based on truncated SVD \eqref{data completion TSVD} means that we only use the prolate spheroidal wave functions $\{\xi_m\}_{m=-J}^{-J+J_\alpha-1}$ (corresponding to eigenvalues that are not close to $0$) which have ``enough energies'' over the aperture $(-\alpha,\alpha)$ where we have access to the {limited-aperture} measurements.

\textbf{Regularization II}: The second choice is to consider the regularization such that
\begin{equation} \label{data completion Tikhonov 1}
\mathbb{P}^\dagger = \mathbb{U} \left(\frac{1}{ \sigma_j + \eps}\right) \mathbb{U}^*,
\end{equation}
or the Tikhonov regularization
\begin{equation} \label{data completion Tikhonov 2}
\mathbb{P}^\dagger  = \mathbb{U} \left(\frac{\sigma_j}{\sigma_j^2 +\epsilon}\right) \mathbb{U}^*,
\end{equation}
where $\eps>0$ is a regularization parameter. In this case, we take $\mathbb{P}^\dagger$ to replace $\mathbb{P}^{-1}$. From the point of view of Slepian's spheroidal wave functions, this method attempts to use some information of the spheroidal wave functions with ``small energy'' on the interval $[-\alpha,\alpha]$.

\subsection{Backscattering multi-static response matrix}
The following symmetric property of the backscattering multi-static response (MSR) matrix is helpful in the data completion algorithms. To begin with, let
$$
\theta_{\hat{x}_j}:=(j-1)2\pi/M -\alpha, \qquad \theta_{d_j}:=  (j-1)2\pi/M-\alpha+\pi, \qquad j=1,2,\ldots,M.
$$

The MSR matrix $\mathbb{F}_{full}\in \C^{M\times M}$ is defined as
\be\label{MSR}
\mathbb{F}_{full}
:= \left(
    \begin{array}{cccc}
      u_{11}^\infty\quad u_{12}^\infty\quad \cdots\quad u_{1M}^\infty \\
      u_{21}^\infty\quad u_{22}^\infty\quad \cdots\quad u_{2M}^\infty \\
      \vdots\,\qquad \vdots\,\quad \ddots\,\qquad \vdots \\
      u_{M1}^\infty\,\, u_{M2}^\infty\,\, \cdots\quad u_{MM}^\infty \\
      \end{array}
  \right),
\en
where $u^{\infty}_{ij}=u^\infty(\theta_{\hat{x}_j};\theta_{d_i})$ for $1\leq i, j\leq M$.
\begin{theorem}\label{blocksymmetric}
\be\label{MSR-SYM}
\mathbb{F}_{full} = \mathbb{F}_{full}^{\rm T}.
\en
\end{theorem}
\begin{proof}
This follows directly by the well known reciprocity relation \cite{CK}
\ben
u^\infty(\theta_{\hat{x}_j};\theta_{d_i})= u^\infty(\theta_{d_i}+\pi;\theta_{\hat{x}_j}+\pi).
\enn
This completes the proof.
\end{proof}

Assume that only {limited-aperture} far field pattern can be measured. The measured data corresponds to a sub-matrix of $\mathbb{F}_{full}$
\be\label{MSR-l}
\mathbb{F}^{(L)}_{limit}
:= \left(
    \begin{array}{cccc}
      u_{11}^\infty\, u_{12}^\infty\, \cdots\, u_{1L}^\infty \\
      u_{21}^\infty\, u_{22}^\infty\, \cdots\, u_{2L}^\infty \\
      \vdots\,\quad \vdots\,\quad \ddots\,\quad \vdots \\
      u_{L1}^\infty\, u_{L2}^\infty\, \cdots\, u_{LL}^\infty \\
      \end{array}
  \right), \quad 1\leq L<M.
\en
By partitioning the $M$-by-$M$ MSR matrix $\mathbb{F}_{full}$ to a $2$-by-$2$ block matrix
\be\label{Fpartition}
\mathbb{F}_{full}=\left(
             \begin{array}{cc}
               \mathbb{F}_{11} & \mathbb{F}_{12} \\
               \mathbb{F}_{21} & \mathbb{F}_{22} \\
             \end{array}
           \right),
\en
where $\mathbb{F}_{11}=\mathbb{F}^{(L)}_{limit}$. We shall discuss how to recover $\mathbb{F}_{full}$ from $\mathbb{F}^{(L)}_{limit}$ in the next section.

\section{Data completion}\label{Data completion}

\subsection{Data completion based on Fourier series}
In this subsection, we introduce the first data completion algorithm. The idea is to represent the full-aperture data in the form of double Fourier series, and to relate the corresponding Fourier coefficients to the limited-aperture data via two prolate matrices.

\vspace{1\baselineskip}

\noindent\textit{Limited-aperture data using Fourier basis}: For the limited-aperture backscattering far field measurements $u^{\infty}(\theta_{\hx},\theta_d)$ with
$\theta_{\hx}\in [-\alpha, \alpha],\, \theta_d\in [\pi-\alpha, \pi+\alpha]$, we introduce the the infinite dimensional matrix $B^\alpha_\infty$ with $pq-$entry given by $b^\alpha_{pq}$
\begin{equation} \label{def B alpha}
b_{pq}^{\alpha}:=\int_{-\alpha}^\alpha\int_{\pi-\alpha}^{\pi+\alpha} u^{\infty} (\theta_{\hx},\theta_d) \ov{\phi_p(\theta_{\hx}) \phi_q(\theta_d)}ds(\theta_{\hx})ds(\theta_{d}),\qquad p,q = 0, \pm 1,\cdots,
\end{equation}
here we recall that the Fourier basis is given by $\phi_m(\theta) = \frac{1}{\sqrt{2\pi}}e^{im\theta}, \quad m=0, \pm1, \pm2,\cdots$.

\vspace{1\baselineskip}

\noindent\textit{Full-aperture data using Fourier basis}: For the full-aperture backscattering far field measurements $u^{\infty}(\theta_{\hx},\theta_d)$, we introduce  the infinite dimensional matrix $B_\infty$ with $pq-$entry given by $b_{pq}$
\begin{equation} \label{def B}
b_{pq} :=\int_{-\pi}^{\pi}\int_{-\pi}^{\pi} u^{\infty}(\theta_{\hx},\theta_d)\ov{\phi_p(\theta_{\hx}) \phi_q(\theta_d)}ds(\theta_{\hx})ds(\theta_{d}), \qquad p,q = 0, \pm 1,\cdots.
\end{equation}
The full-aperture backscattering far field measurements in the  Fourier basis correspond to the infinite dimensional matrix $B_\infty$.

Furthermore, given the knowledge of $B_\infty$, we can write down $u^{\infty}$ in (double) Fourier series as
\begin{equation}  \label{B to far field infinite}
u^{\infty}(\theta_{\hx},\theta_d) = \sum_{m=-\infty}^{\infty} \sum_{n=-\infty}^{\infty}b_{mn}\phi_m(\theta_{\hx}) \phi_n(\theta_d), \quad \theta_{\hat{x}},\theta_d \in [-\pi,\pi],
\end{equation}
and approximate $u^{\infty}$ using a truncated Fourier series as
\begin{equation} \label{B to far field finite}
u^{\infty}(\theta_{\hx},\theta_d) \approx \sum_{m=-J}^{J}\sum_{n=-J}^{J}b_{mn}\phi_m(\theta_{\hx}) \phi_n(\theta_d), \quad \theta_{\hx},\theta_d  \in [-\pi,\pi],
\end{equation}
for $J$ large enough so that the approximation error is sufficiently small in the $L^2$ sense.

\vspace{1\baselineskip}

\noindent\textit{Relation between limited-aperture data and full-aperture data}:
We first derive a relation between the limited-aperture data and full-aperture data as follows.
\begin{lemma} \label{thm far field case partial to full infinite}
Let $B^\alpha_\infty$ and $B_\infty$ be given by \eqref{def B alpha} and \eqref{def B} respectively. It holds that
\begin{equation} \label{far field case partial to full infinite}
B^{\alpha}_\infty = \mathbb{P}_{\hx,\infty }B_\infty \mathbb{P}_{d,\infty },
\end{equation}
where the infinite dimensional matrix $\mathbb{P}_{\hx,\infty }$ is the prolate matrix (with dimension infinity)  given by \eqref{def Prolate Matrix}, and the infinite dimensional matrix $\mathbb{P}_{d,\infty }$ is given by $\mathbb{P}_{d,\infty }:=\Big((-1)^{m-n}\mathbbm{p}_{mn}\Big)$.
\end{lemma}

\begin{proof}
Assume that the full-aperture measurements are given, then there is the double Fourier series expansion \eqref{B to far field infinite}. Plugging this expression into the definition $\eqref{def B alpha}$  yields
$$
b^\alpha_{pq} = \mathbbm{p}_{mp} b_{mn} \mathbbm{p}_{nq} (-1)^{(n-q)}.
$$
This proves \eqref{far field case partial to full infinite} and completes the proof.
\end{proof}

%
%
%
%

\vspace{1\baselineskip}

\noindent\textit{Finite dimensional case}:
In practice, the measurements are discrete data. This motivates us to consider a finite dimensional space consisting of   $\phi_m(\theta) = \frac{1}{\sqrt{2\pi}}e^{im\theta}, \, m=0, \pm1, \cdots,\pm J$ for a sufficiently large $J$. We shall discuss more details on the choice of $J$ in the foregoing analysis and numerical examples.

The following theorem in the finite dimensional case follows immediately from Lemma \ref{thm far field case partial to full infinite}.
\begin{theorem} \label{thm far field case partial to full finite}
Let $B^\alpha:=\Big(b^\alpha_{pq}\Big)_{-J \le p,q\le J}$ and $B:=\Big(b_{pq}\Big)_{-J \le p,q\le J}$ with $b^\alpha_{pq}$ and $b_{pq}$ given by \eqref{def B alpha} and \eqref{def B} respectively. It holds that
\begin{equation} \label{far field case partial to full finite}
B^{\alpha} = \mathbb{P}_{\hx}B\mathbb{P}_{d},
\end{equation}
where $\mathbb{P}_{\hx}$ is the $(2J+1)\times (2J+1)$ prolate matrix given by \eqref{def Prolate Matrix}, and $\mathbb{P}_{d}$ is the $(2J+1)\times (2J+1)$ matrix   given by $\mathbb{P}_{d}:=\Big((-1)^{m-n}\mathbbm{p}_{mn}\Big)$.
\end{theorem}

\vspace{1\baselineskip}

\noindent\textit{From limited-aperture data to full-aperture data}:
Now it is clear that the limited-aperture data is related to the full-aperture data via  \eqref{far field case partial to full finite}. Our goal is then to find $B$ or its approximation from $B^\alpha $ via \eqref{far field case partial to full finite}. From the properties of the prolate matrix in Lemma \ref{Mproperties}, we have that the eigenvalues of $\mathbb{P}_{\hx}$ (and $\mathbb{P}_{d}$) are all positive, but they are clustered near $1$ and $0$, and hence the matrix $\mathbb{P}_{\hx}$ (and $\mathbb{P}_{d}$) is severely ill-conditioned. In fact the eigenvalues decay exponentially to $0$ when $J$ becomes large. Therefore we can only hope to invert $\mathbb{P}_{\hx}$ (and $\mathbb{P}_{d}$) using regularization techniques in order to find $B$ from $B^\alpha$. We shall apply
Regularizations I-II \eqref{data completion TSVD}-\eqref{data completion Tikhonov 2}
to find approximate inverses of $\mathbb{P}_{\hx}$ and $\mathbb{P}_{d}$.


Now we are ready to summarize the first data completion algorithm named by {\bf DC-FS}, which is based on the Fourier series.

\begin{algorithm}\label{algorithm 1}
{\bf (DC-FS)}:
\begin{itemize}
\item Step I: Compute $B^{\alpha}=\Big(b_{pq}^{\alpha}\Big)$ from the measurements $\{u^{\infty}(\theta_{\hx},\theta_d): \, \theta_{\hx}\in [-\alpha, \alpha],\, \theta_d\in [\pi-\alpha, \pi+\alpha]\}$ by \eqref{def B alpha}.
\item Step II: Approximate $B$ by $ \mathbb{P}_{\hx}^{\dagger} B^\alpha \mathbb{P}_{d}^{\dagger}$, where $\mathbb{P}^\dagger_{\hx}$ (resp. $\mathbb{P}_{d}^{\dagger}$) is the approximate inverse of $\mathbb{P}$ (resp. $\mathbb{P}_{d}$) using
Regularizations I-II \eqref{data completion TSVD}-\eqref{data completion Tikhonov 2}.
\item Step III: Recover the full-aperture data by \eqref{B to far field finite}.
\end{itemize}
\end{algorithm}

\subsection{Data completion by solving integral equations}
This subsection is devoted to a different data completion algorithm by solving an integral equation and the symmetric structure of the MSR matrix. Precisely, for each incident direction, we represent the far field data by the single layer potential, and find that the  Fourier coefficients of the density are related to the  limited-aperture data via a prolate matrix (which turns out to be same for every incident direction).

In a recent paper \cite{LiuSun19}, the authors introduce a data completion algorithm based on solving the following integral equation
\be\label{uinf-partial}
u^{\infty}(\theta_{\hx}) = \int_{\pa B_r}e^{-ik \hx\cdot y}\bbphi(y)ds(y),\quad  \theta_{\hx}\in [-\alpha,\alpha]\subset [-\pi,\pi].
\en
with a density $\bbphi\in L^2(\pa B_r)$. Here, $\hx=(\cos\theta_{\hx}, \sin\theta_{\hx})^{\rm T}$, $B_r$ is a ball large enough such that the unknown object $\ov{D}\subset B_r$.
The idea is first to compute a regularized solution $\bbphi$ of the equation \eqref{uinf-partial} with limited-aperture data on $[-\alpha,\alpha] $, and to insert the regularization $\bbphi$  into \eqref{uinf-partial} to obtain the full aperture data on $[-\pi,\pi]$. We introduce here a novel technique to obtain an approximate solution $\bbphi$ of the boundary integral equation \eqref{uinf-partial}. To begin with, we set the polar coordinates $y=r(\cos\theta_y,\sin\theta_y)^{\rm T}$ for any $y \in \partial B_r$.

\begin{lemma} \label{lemma DC-IE}
Assume that
\be \label{def bbphi}
\bbphi(y)
= \sum_{m=-\infty}^{\infty}\gamma_m(r)\phi_{m}(\theta_y), \quad y\in\pa B_r
\en
solves equation \eqref{uinf-partial}. Let $\mathfrak{B}_m$ be the Bessel functions of order $m$.  Then the infinite dimensional vector $\Gamma_\infty=2\pi((-i)^{m}\mathfrak{B}_{m}(kr)\gamma_m(r))$ is related to the limited-aperture data via

\be\label{CG}
C_\infty = \mathbb{P}_\infty \Gamma_\infty,
\en
where $\mathbb{P}_\infty=\Big(\mathbbm{p}_{mn}\Big)$ is the infinite dimensional prolate matrix and $C_\infty$ is given by the limited-aperture data
\be \label{cgamma}
C_\infty=\Big(c_{n} \Big) :=\left(\int_{-\alpha}^\alpha u^{\infty}(\theta_{\hx})\phi_{-n}(\theta_{\hx})d \theta_{\hx}  \right).
\en
\end{lemma}
\begin{proof}
Recall the Jacobi-Anger expansion \cite{CK}
\ben
e^{-ik\hx\cdot y} = 2\pi\sum_{m=-\infty}^{\infty}(-i)^{m}\mathfrak{B}_{m}(kr)\phi_{m}(\theta_{\hx})\phi_{-m}(\theta_y),\quad y\in\R^2.
\enn
Inserting the above   expansion and \eqref{def bbphi} into \eqref{uinf-partial}, by the orthogonality of $\phi_{m}(\theta)$,  we have
\be\label{unif-partial2}
u^{\infty}(\theta_{\hx}) = 2\pi\sum_{m=-\infty}^{\infty}(-i)^{m}\mathfrak{B}_{m}(kr)\gamma_m(r)\phi_{m}(\theta_{\hx}),\quad \theta_{\hx}\in [-\alpha,\alpha].
\en
Multiplying \eqref{unif-partial2} by $\phi_{-n}(\theta_{\hx})$ and integrating over $[-\alpha,\alpha]$, we obtain from the definition of $c_n$ \eqref{cgamma} that
\ben
c_{n}= \sum_{m=-\infty}^{\infty}(-i)^{m}\mathfrak{B}_{m}(kr)\gamma_m(r)\mathbbm{p}_{mn}, \quad n=0,\pm1,\cdots,\pm\infty
\enn
with $\Big(\mathbbm{p}_{mn}\Big)$ being the prolate matrix. This proves \eqref{CG} and completes the proof.
 \end{proof}

\vspace{1\baselineskip}

\noindent\textit{Finite dimensional case}:
When only using a finite dimensional space consisting of  $\phi_m(\theta) = \frac{1}{\sqrt{2\pi}}e^{im\theta}, \,m=0, \pm1, \cdots,\pm J$ for a sufficiently large $J$, we have the following truncated version immediately from Lemma \ref{lemma DC-IE}.

\begin{theorem} \label{theorem DC-IE}
Let $\bbphi_J(y)
= \sum_{m=-J}^{J}\gamma_m(r)\phi_{m}(\theta_y)$ be an approximation of the density $\bbphi(y)$ \eqref {def bbphi}. The coefficients $\gamma_m(r)$ satisfies the following truncated form of \eqref{CG}, i.e.
\be\label{CG finite}
C = \mathbb{P} \Gamma,
\en
where $C:=\Big(c_n\Big)_{-J\le n \le J}$, $\mathbb{P}=\Big(\mathbbm{p}_{mn}\Big)_{-J\le m,n \le J}$, and $\Gamma:=2\pi\Big((-i)^{m}\mathfrak{B}_{m}(kr)\gamma_m(r)\Big)_{-J\le m \le J}$.
\end{theorem}
After solving for $\Gamma$ (or equivalently $\bbphi_J$) using regularization techniques, then we can insert $\bbphi_J$ into \eqref{uinf-partial} to get an approximation of the full-aperture data.
Now we can summarize the second algorithm named by {\bf DC-IE}, which is obtained with the help of solving an integral equation.
\begin{algorithm} \label{algorithm 2}
{\bf (DC-IE):}
\begin{itemize}
\item Step I: For every incident direction $d \in [\pi-\alpha,\pi+\alpha]$, do:

\begin{itemize}
\item Compute the $2J+1$ dimentional vector $C=\Big(c_m\Big)$ from the measurements $\{u^{\infty}(\theta_{\hx}): \, \theta_{\hx}\in [-\alpha, \alpha]\}$ by \eqref{cgamma}.
\item Find an approximate inverse $\mathbb{P}^{\dagger}$   of $\mathbb{P}$ using Regularizations I-II \eqref{data completion TSVD}-\eqref{data completion Tikhonov 2}.
\item Approximate $2\pi\Gamma$ by $ 2\pi \widetilde{\Gamma}:=\mathbb{P}^{\dagger} C$.
\end{itemize}
Recover the full-aperture data (e.g. $(\mathbb{F}_{11},\mathbb{F}_{12})$) approximately by
$$
u^{\infty} (\theta_{\hx}) \approx \sum_{m=-J}^{J} (2\pi \widetilde{\Gamma})_m \phi_m(\theta_{\hx}),\quad \theta_{\hx} \in [-\pi,\pi].
$$
\item Step II: Take $\mathbb{F}_{21}:=\mathbb{F}_{12}^{\rm T}$.
\item Step III: Repeat Step I for incident directions $d \in \mathbb{S} \backslash [\pi-\alpha,\pi+\alpha]$ to get $\mathbb{F}_{22}$ to complete the algorithm.
\end{itemize}
\end{algorithm}

\begin{remark}
There is another perspective of Data Completion Algorithm \ref{algorithm 2}. Recall again the limited-aperture backscattering far field measurements $u^{\infty}(\theta_{\hx},\theta_d)$ with
$\theta_{\hx}\in [-\alpha, \alpha],\, \theta_d\in [\pi-\alpha, \pi+\alpha]$, we can define for each incident direction $d$ that
\ben
b_{p}^{\alpha}(d):=\int_{-\alpha}^\alpha u^{\infty}(\theta_{\hx},\theta_d) \ov{\phi_p(\theta_{\hx})  }d \theta_{\hx}.
\enn
Performing the same argument as in the proof of Lemma \ref{thm far field case partial to full infinite} and noting the definition of $b_{mn}$ in \eqref{def B}, it follows that
\be \label{CG another perspective}
B^{\alpha}(d) = \mathbb{P} B(d),
\en
where $
B^{\alpha}(d):=\Big(b_{m}^{\alpha}(d)\Big)_{-J\le m\le J}$ and $B(d):=\Big(\sum_{n=-J}^{J}b_{mn}\phi_n(\theta_d)\Big)_{-J\le m\le J}$
respectively.
The relation between the limited-aperture and full-aperture data in \eqref{CG another perspective} is equivalent (up to a termwise scaling) to \eqref{CG finite} derived from solving the integral equations.

\end{remark}

\section{Imaging method for object reconstructions}\label{DSMs}
\setcounter{equation}{0}
After the data completion, many numerical methods using full-aperture data can be applied for reconstructing the shape and location of the underlying objects.
In this paper, we recall two well known non-iterative methods: the factorization method \cite{Kirsch98,KirschGrinberg} and the direct sampling method \cite{LiuIP17}.
%


\vspace{1\baselineskip}

\noindent\textit{Factorization Method} (FM): We introduce the far-field operator $F: L^2(\mathbb{S}) \to L^2(\mathbb{S})$ by
\begin{equation} \label{far field operator}
(F g)({\hx}):= \int_{\mathbb{S}} u^{\infty}({\hx},d) g(d) d s(d), \quad {\hx} \in \mathbb{S}.
\end{equation}
For any sampling point $z$ in a sampling region, the factorization method yields the imaging function
\begin{equation}
I_{FM}(z):= \frac{1}{\|g_z\|^2},
\end{equation}
where $g_z$ is the (regularized) solution to
\be\label{FM}
(F_\sharp^{1/2} g_z)(\hx) = e^{-ik \hx\cdot z}, \quad {\hx} \in \mathbb{S}.
\en
Here $F_\sharp : = |\Re F| + |\Im F|$. Theoretically, apart from some possible wave numbers, the equation \eqref{FM} is solvable if and only if $z\in D$.
This implies that $I_{FM}(z)$ is small for $z\in \mathbb{R}^2\backslash \overline{D}$ and relatively large for $z\in D$ \cite{Kirsch98,KirschGrinberg}.

\vspace{1\baselineskip}

\noindent\textit{Direct Sampling Method} (DSM): For any sampling point $z$ in a sampling region, we consider the following imaging function
\begin{equation}
I_{DSM}(z):= \left| \int_{\mathbb{S}}\int_{\mathbb{S}} u^{\infty}({\hx},d) e^{-ikd\cdot z} e^{ik \hat{x} \cdot z} d s(d) d s(\hx)  \right|.
\end{equation}
The imaging function $I_{DSM}(z)$ is expected to peak when $z\in \partial D$ and decays like the Bessel functions for sampling points away from the boundary. We refer to \cite{LiuIP17} for the corresponding theoretical analysis and its connections with the other sampling methods.

Now we are ready to outline our imaging method for object reconstructions with limited-aperture data.

\medskip

\noindent\textbf{Imaging Algorithm.}
\
\begin{itemize}
\item Data completion: Recover the full-aperture data approximately via Data Completion Algorithm \textbf{DC-FS} or \textbf{DC-IE} from limited-aperture data.
\item Sampling method: Reconstruct the object by the imaging function $I_{FM}$ or $I_{DSM}$.
\end{itemize}

Finally, we remark that both the data completion algorithms proposed in this paper and the sampling methods we considered for the object reconstructions are independent of the a priori information of the unknown objects, which is quite important in many practical applications.

\section{Numerical examples and discussions}\label{NumExamples}
\setcounter{equation}{0}

In this section, we present some numerical examples to illustrate the performance of the data completion and imaging algorithms proposed in the previous sections.
The numerical examples are divided into two groups.
We first present some numerical examples to demonstrate how to use the data completion algorithms to recover the full-aperture data.
The second group of numerical examples is to use the recovered data in the classical factorization method and direct sampling method for imaging.


The boundaries of the objects in our numerical experiments are parameterized as follows (see Fig.~\ref{truedomains}):
\ben
\label{peanut}&\mbox{\rm Peanut:}&\quad x(t)\ =\sqrt{3\cos^2 t+1}(\cos t, \sin t),\quad 0\leq t\leq2\pi,\\
\label{disk}&\mbox{\rm Disk:}&\quad x(t)\ = 2(\cos t, \sin t),\quad 0\leq t\leq2\pi.
\enn
\begin{figure}[htbp]
  \centering
  \subfigure[\textbf{Peanut}]{
    \includegraphics[width=1.4in]{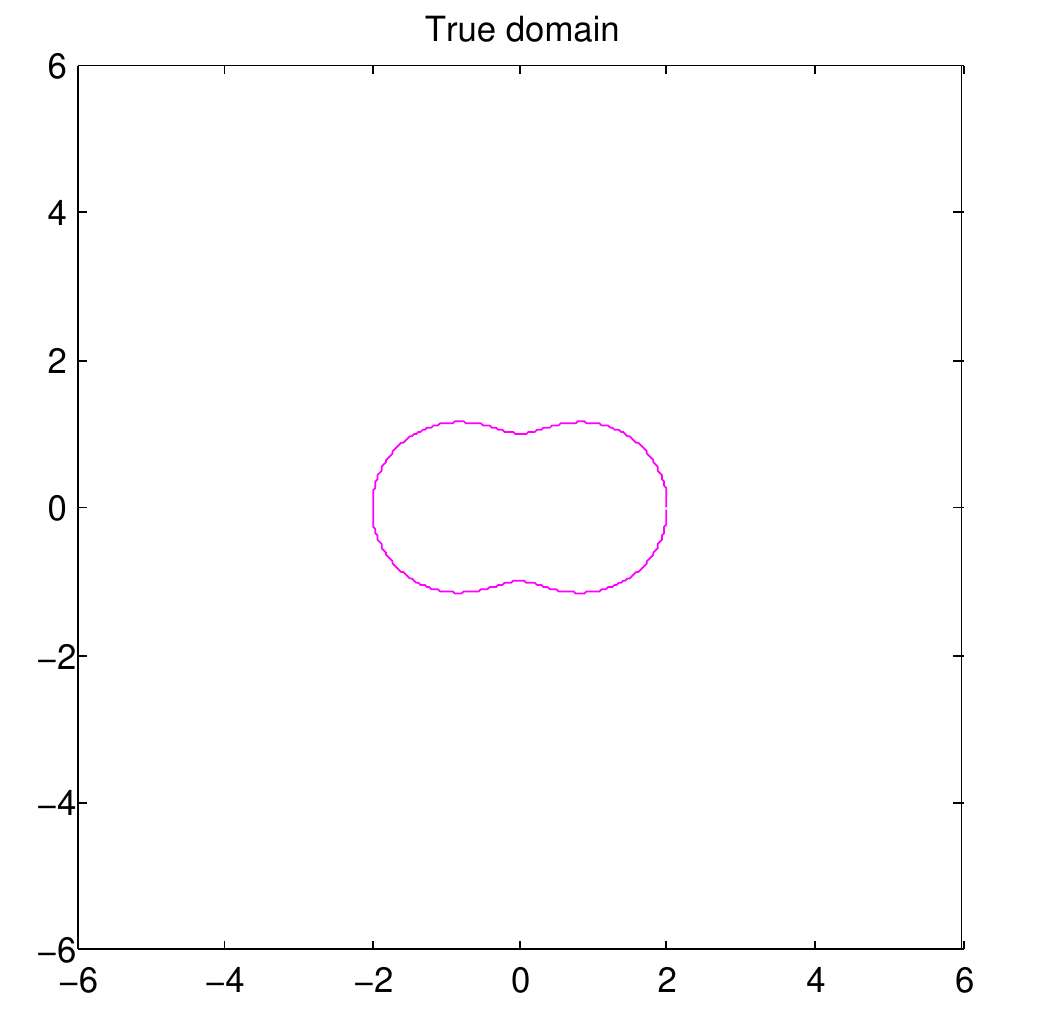}}\quad
  \subfigure[\textbf{Disk}]{
    \includegraphics[width=1.4in]{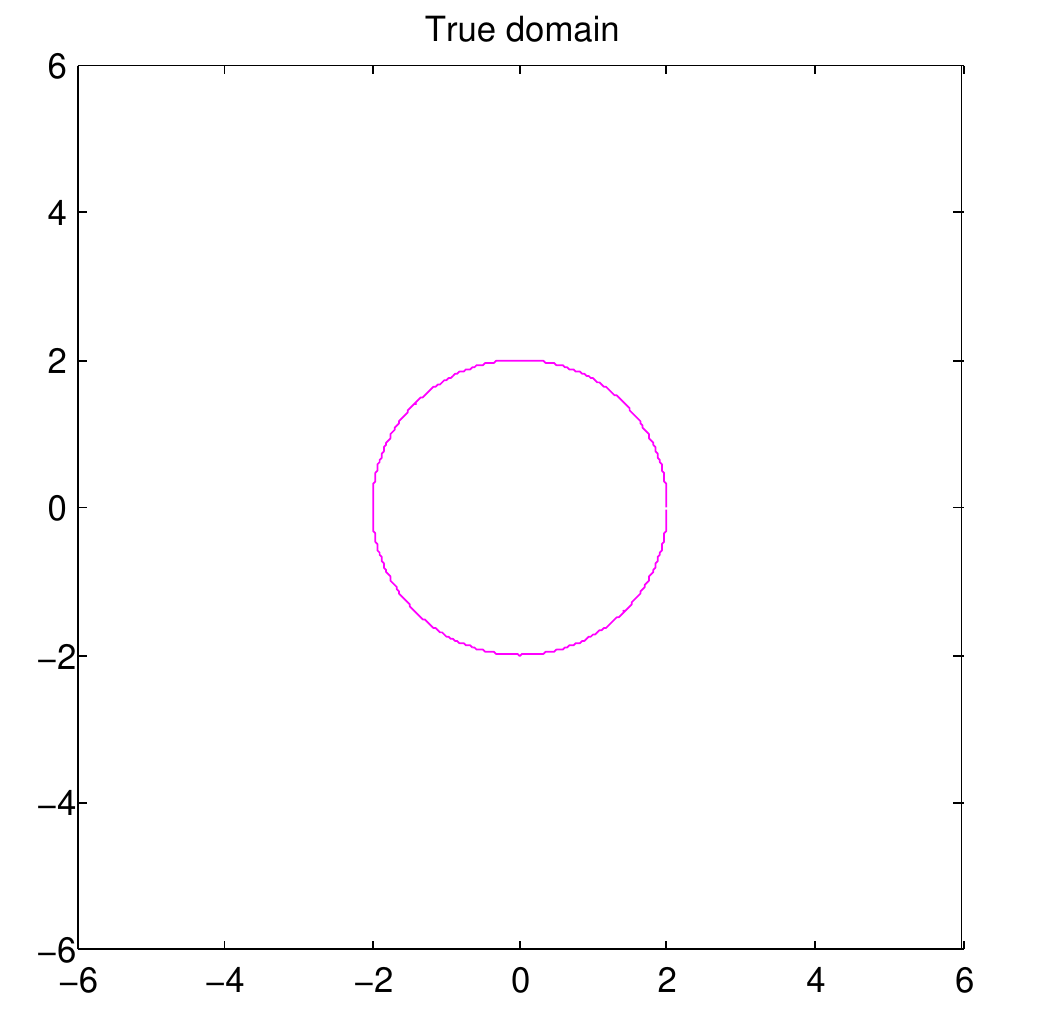}}
\caption{Domains considered. }
\label{truedomains}
\end{figure}

In our simulations, the boundary integral equation method is used to compute the
limited-aperture backscattering far field patterns $u^{\infty}(\theta_{\hx},\theta_d)$
for $L$ equidistantly distributed observation directions  and $L$ equidistantly distributed incident directions over the limited-aperture. This gives the limited-aperture backscattering multi-static response (MSR) matrix $\mathbb{F}^{(L)}_{limit}$ given by \eqref{MSR-l}.
We further perturb $\mathbb{F}^{(L)}_{limit}$ by random noise using
\ben
\mathbb{F}^{(L),\delta}_{limit}\ =\ \mathbb{F}^{(L)}_{limit} +\delta\|\mathbb{F}^{(L)}_{limit}\|\frac{R_1+R_2 i}{\|R_1+R_2 i\|},
\enn
where $R_1$ and $R_2$ are two $L \times L$ matrices containing pseudo-random values
drawn from a normal distribution with mean zero and standard deviation one. The
value of $\delta$ used in our code is $\delta:=\|\mathbb{F}^{(L),\delta}_{limit} -\mathbb{F}^{(L)}_{limit}\|/\|\mathbb{F}^{(L)}_{limit}\|$ which represents the relative error.

In all the subsequent examples, we set the wave number $k=5$ and consider $\delta=5\%$  error level. There are  $L=128$ equidistantly distributed observation directions over the upper half circle.
\\

\subsection{Data completion results and discussions}
This subsection is devoted to verifying the validity of the data completion algorithms {\bf DC-FS} and {\bf DC-IE} proposed in Section \ref{Data completion}.
We take the peanut shaped domain shown in Figure \ref{truedomains} as the unknown object.  To stablize the data completion algorithms, the reconstructed data will be manually set to zero if its magnitude exceeds an appropriately chosen threshold.

To begin with, we illustrate the eigensystem of the prolate matrix. Figure \ref{ProlateEIG} shows the corresponding eigenvalues and five prolate eigenfunctions with $J=39$. Obviously, the eigenvalues are located in $[0,1]$, the first half of them are close to $1$ while the second half are close to $0$. This makes it facile to choose the cut-off value when using TSVD.  The eigenvalues are symmetric with respect to $0.5$.
The prolate eigenfunction $\xi_{30}$, which corresponds to the eigenvalue close to $1$, is almost zero in $(\pi, 2\pi)$.
The prolate eigenfunction $\xi_{50}$, which corresponds to the eigenvalue close to $0$, is almost zero in $(0, \pi)$.
In particular, the prolate eigenfunctions look like dumbbell in their nontrivial parts.
These observations indicate that, to compute the prolate eigenfunction expansion from the partial measurements, it is well-conditioned (resp.  ill-conditioned) to obtain the coefficients of the prolate eigenfunctions with relatively large (resp.  almost zero) eigenvalues.
\\

\begin{figure}[htbp]
  \centering
  \subfigure[\textbf{Eigenvalues}]{
    \includegraphics[width=2in]{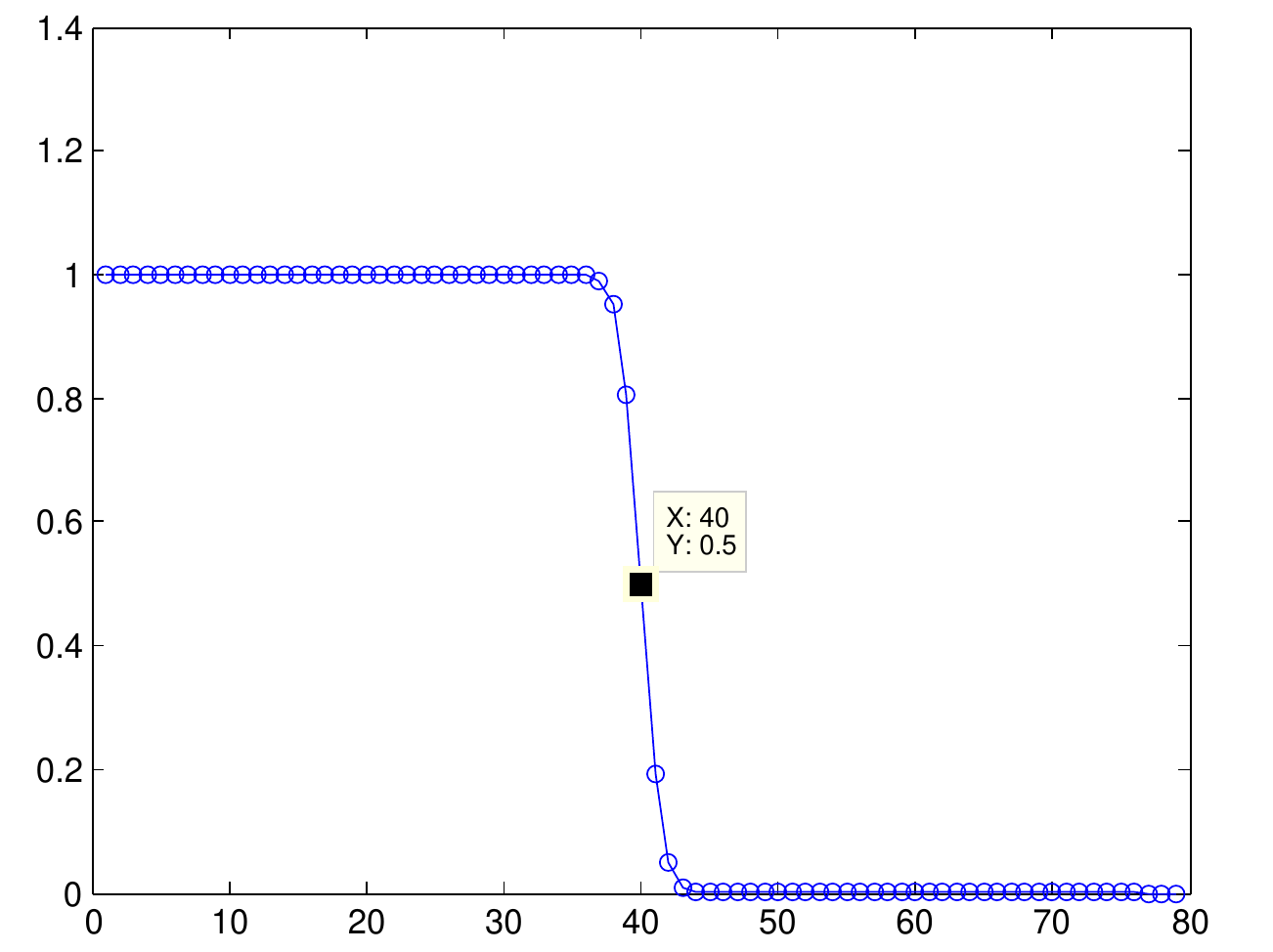}}
  \subfigure[\textbf{$\Re(\xi_{30})$}]{
    \includegraphics[width=2in]{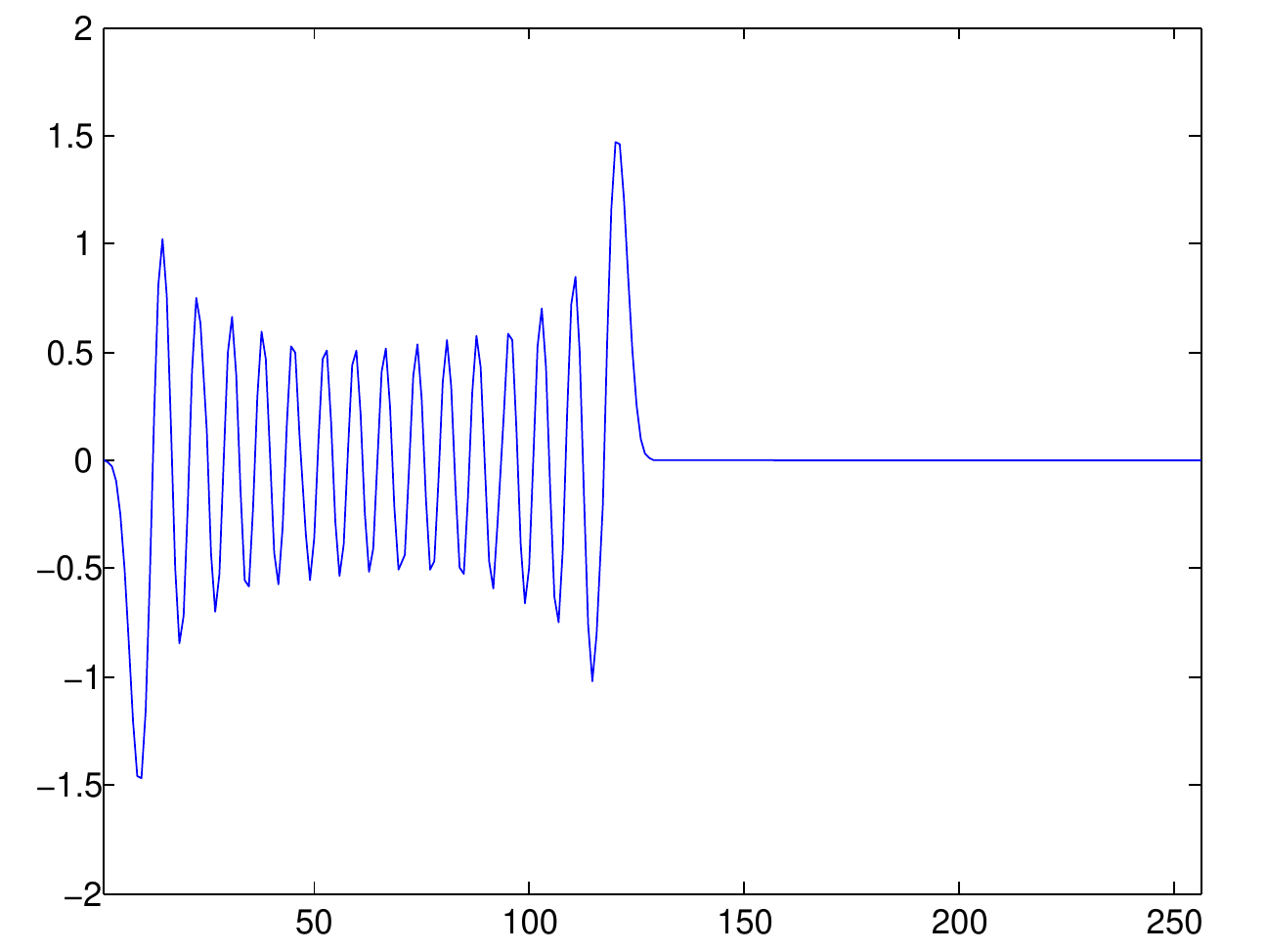}}
  \subfigure[\textbf{$\Re(\xi_{38})$}]{
  \includegraphics[width=2in]{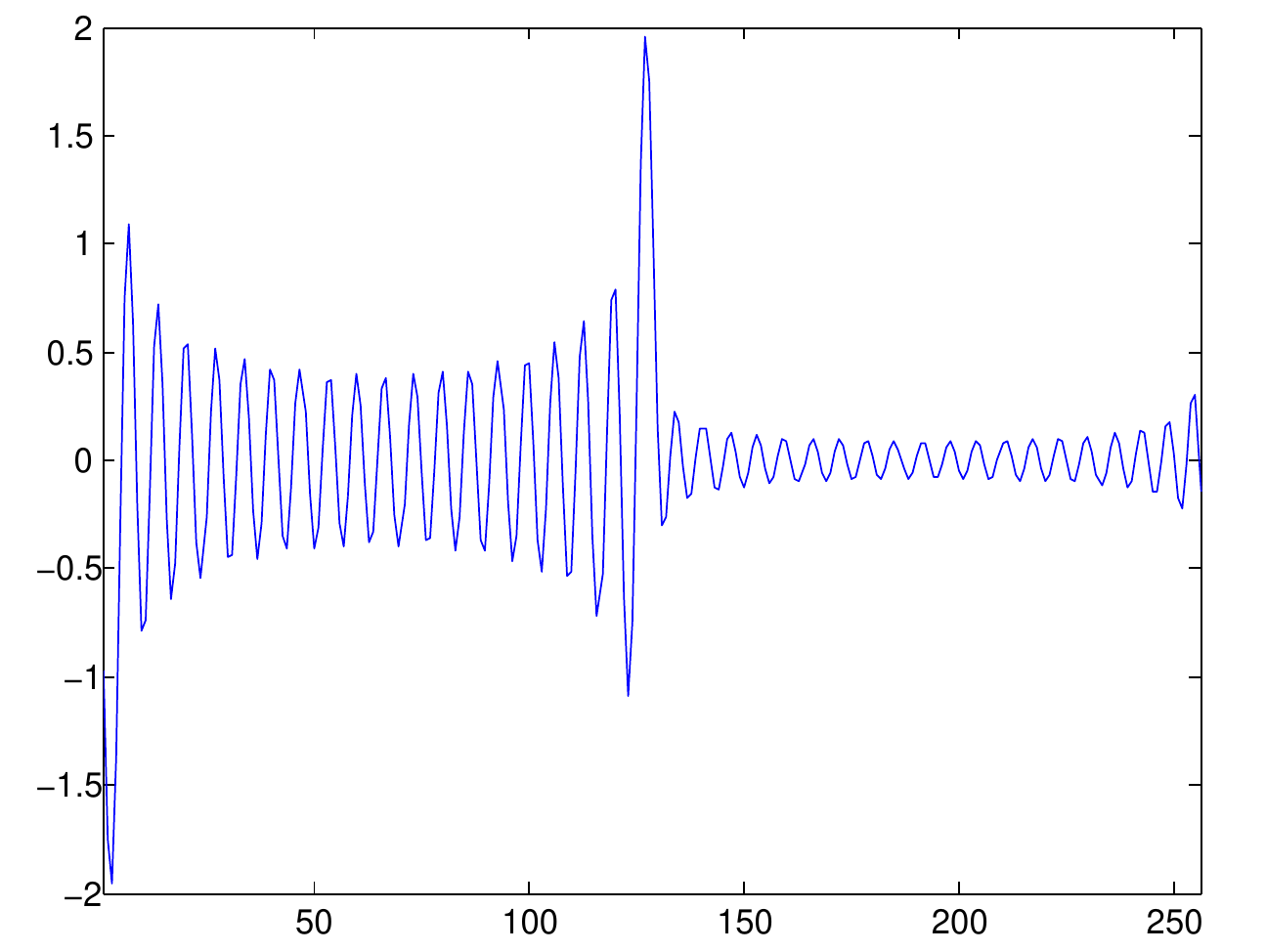}}\\
  \subfigure[\textbf{$\Re(\xi_{40})$}]{
    \includegraphics[width=2in]{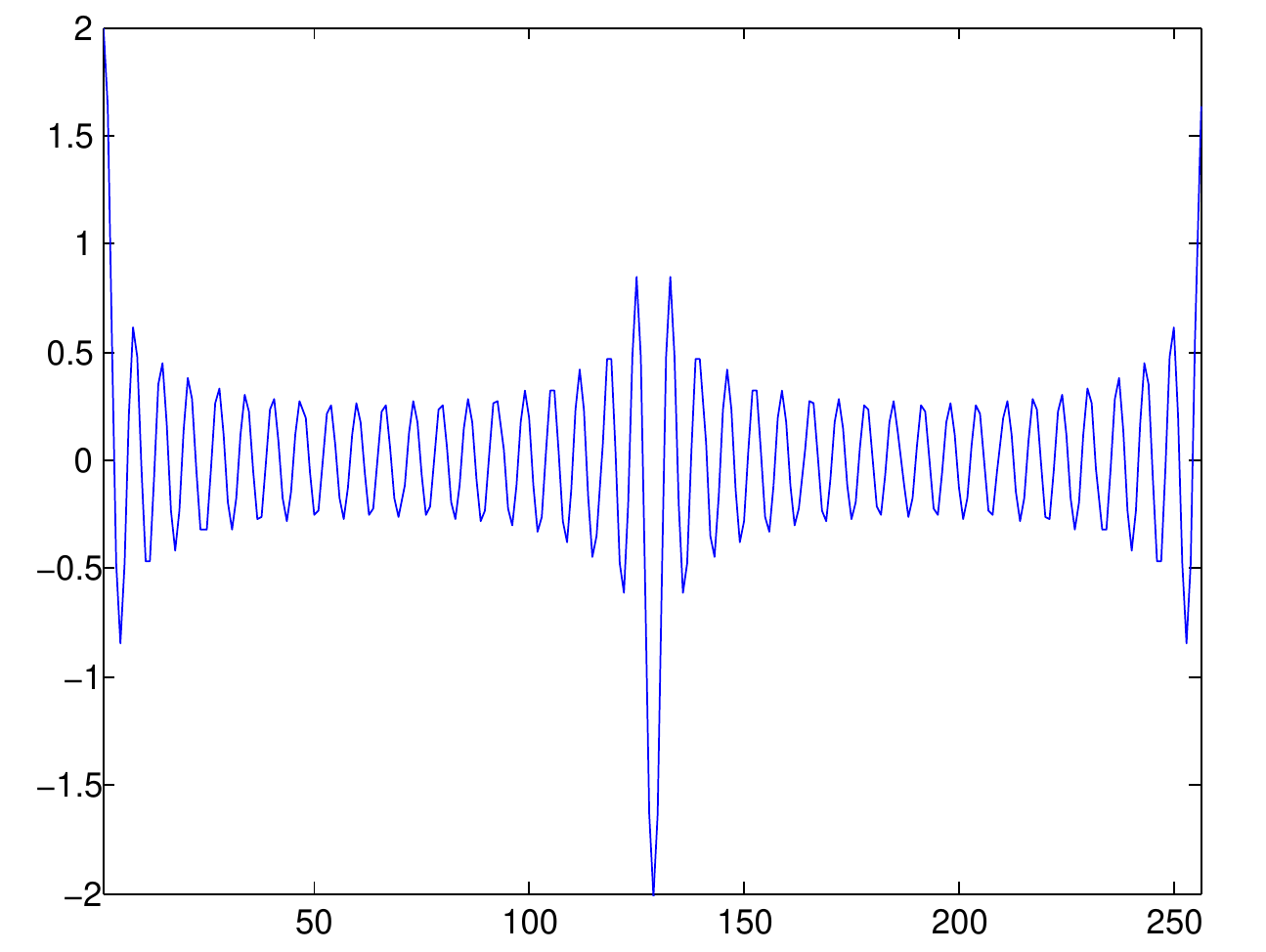}}
  \subfigure[\textbf{$\Re(\xi_{42})$}]{
    \includegraphics[width=2in]{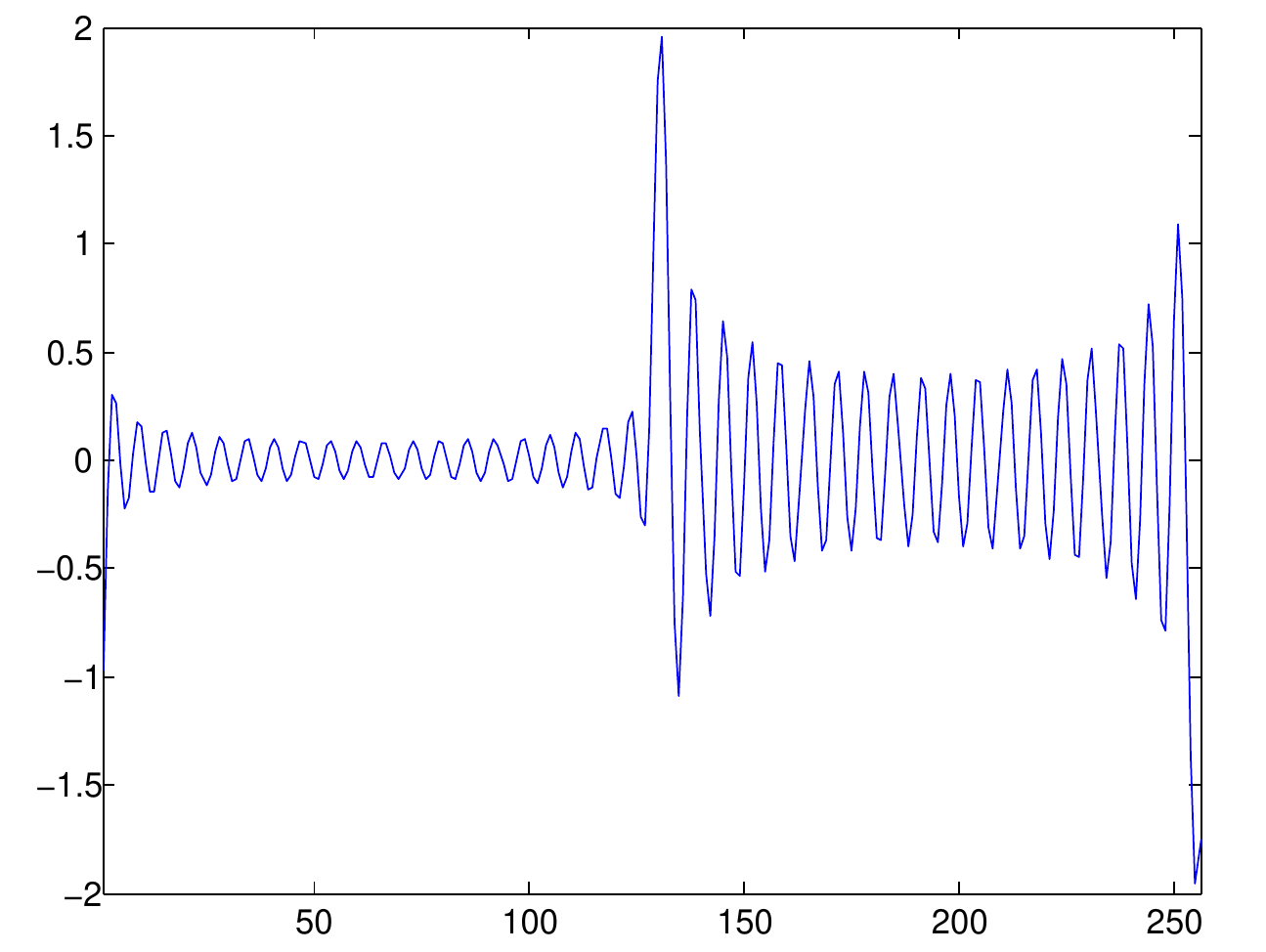}}
  \subfigure[\textbf{$\Re(\xi_{50})$}]{
  \includegraphics[width=2in]{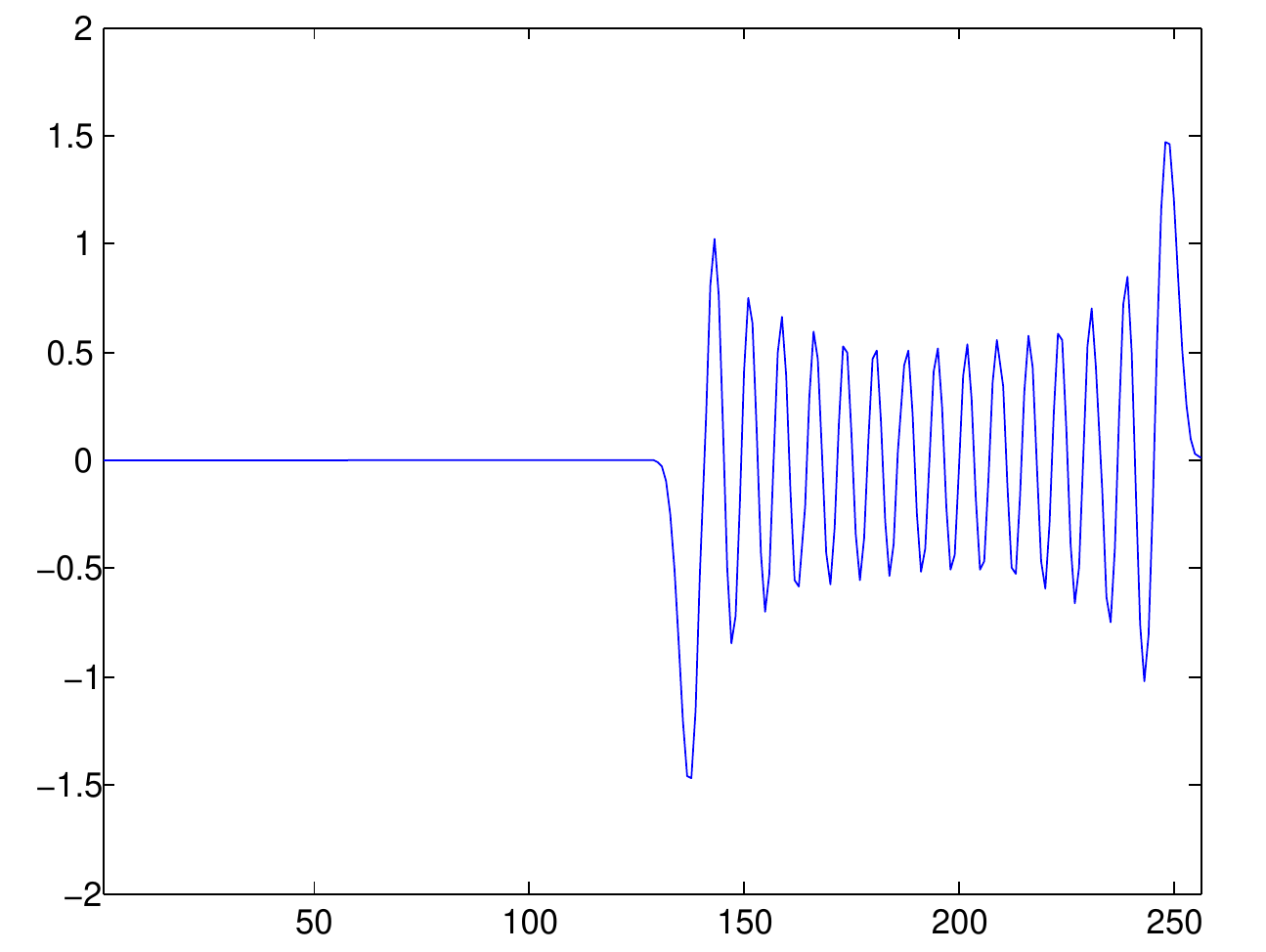}}
\caption{The eigenvalues and eigenfunctions of the prolate matrix with $J=39$: $\xi_m$ is the eigenfunction corresponding to the $m$-th   eigenvalue.}
\label{ProlateEIG}
\end{figure}

Figures \ref{DC-2}-\ref{TSVD-DC-1} show the data completion results using the two proposed data completion algorithms {\bf DC-FS} and {\bf DC-IE}.

Figures \ref{DC-2}-\ref{DC-1} present the results using Regularization II \eqref{data completion Tikhonov 1}   with parameter $\eps = 10^{-3}$. We find that $J=9$ is a sufficiently good choice since the reconstructed data has been matched to the measurements in $[0,\pi]$. However, in the unavailable part
$[\pi, 2\pi]$, the reconstructed data take the form of a dumbbell, which is the result of the dumbbell behavior of the prolate eigenfunctions corresponding to small eigenvalues.

Figures \ref{TSVD-DC-2}-\ref{TSVD-DC-1} present the results using Regularization I \eqref{data completion TSVD}   with cut-off value $\sigma = 0.1$. Due to the truncated SVD, the eigenfunctions corresponding to eigenvalues smaller than $\sigma = 0.1$ are not used. For large $J=39$, as shown in  Figures \ref{TSVD-DC-2}-\ref{TSVD-DC-1}(c)(f), the reconstructed data is small over $(\pi,2\pi)$. This is because that we mostly use the prolate eigenfunctions corresponding to large eigenvalues  and those prolate eigenfunctions are small over $(\pi,2\pi)$, see for example $\Re(\xi_{30})$ in Figure \ref{ProlateEIG}(b).

We also observe that the performance of data completion algorithm {\bf DC-IE} seems a little bit  better compared with the data completion algorithm
{\bf DC-FS}. This may be due to the fact that we have computed the inversion of only one prolate matrix in {\bf DC-IE}.

We make a remark on the choice of $J$. Since there are $L=128$ equidistantly distributed directions (which will be used as quadrature points in the computation of $B^\alpha$ \eqref{def B alpha} or $C$ \eqref{cgamma}) on the limited aperture with length $\pi$, this implies that the frequency of  $e^{iJ\theta}$ cannot be too large in order to ensure a good quadrature approximation of $B^\alpha$ \eqref{def B alpha} or $C$ \eqref{cgamma}. In particular, in the wavelength $\frac{2\pi}{J}$ of  $e^{iJ\theta}$,  the number of quadrature points is $\frac{2\pi}{J} \frac{L}{\pi}=2L/J$. This gives a way to finding a good $J$ in a heuristic way.
This is further illustrated by  Figure \ref{DC-1}.
When $J=39$, one observes that the limited-aperture data on $(0,\pi)$ were not well approximated in  Figure \ref{DC-1} (c), where the  number of quadrature points is approximately $\frac{2\pi}{J} \frac{L}{\pi}=2L/J\approx 6$ which seems not sufficiently enough.


\begin{figure}[htbp]
  \centering
\subfigure[\textbf{DC-FS with $J=4$}]{
    \includegraphics[width=2in]{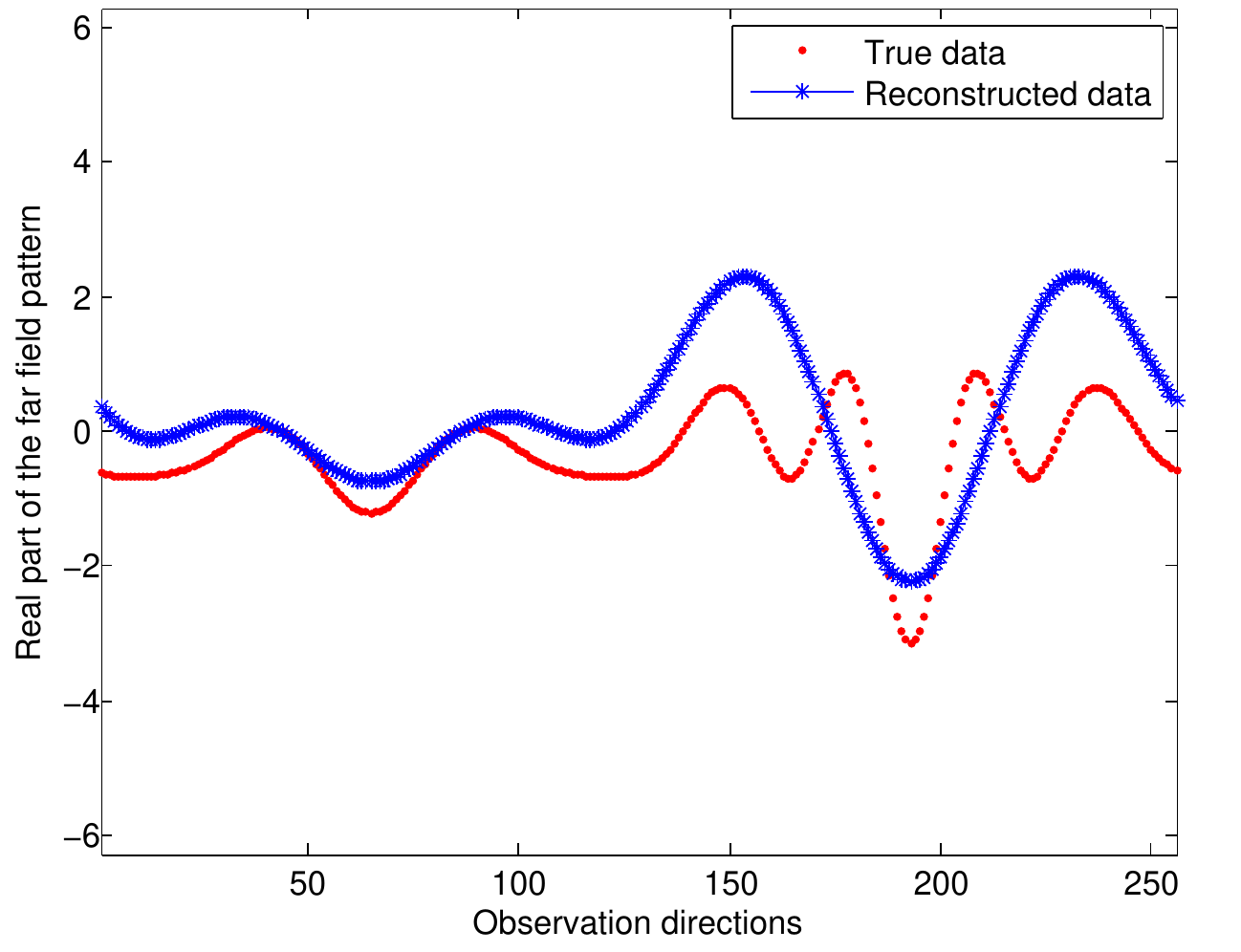}}
  \subfigure[\textbf{DC-FS with $J=9$}]{
    \includegraphics[width=2in]{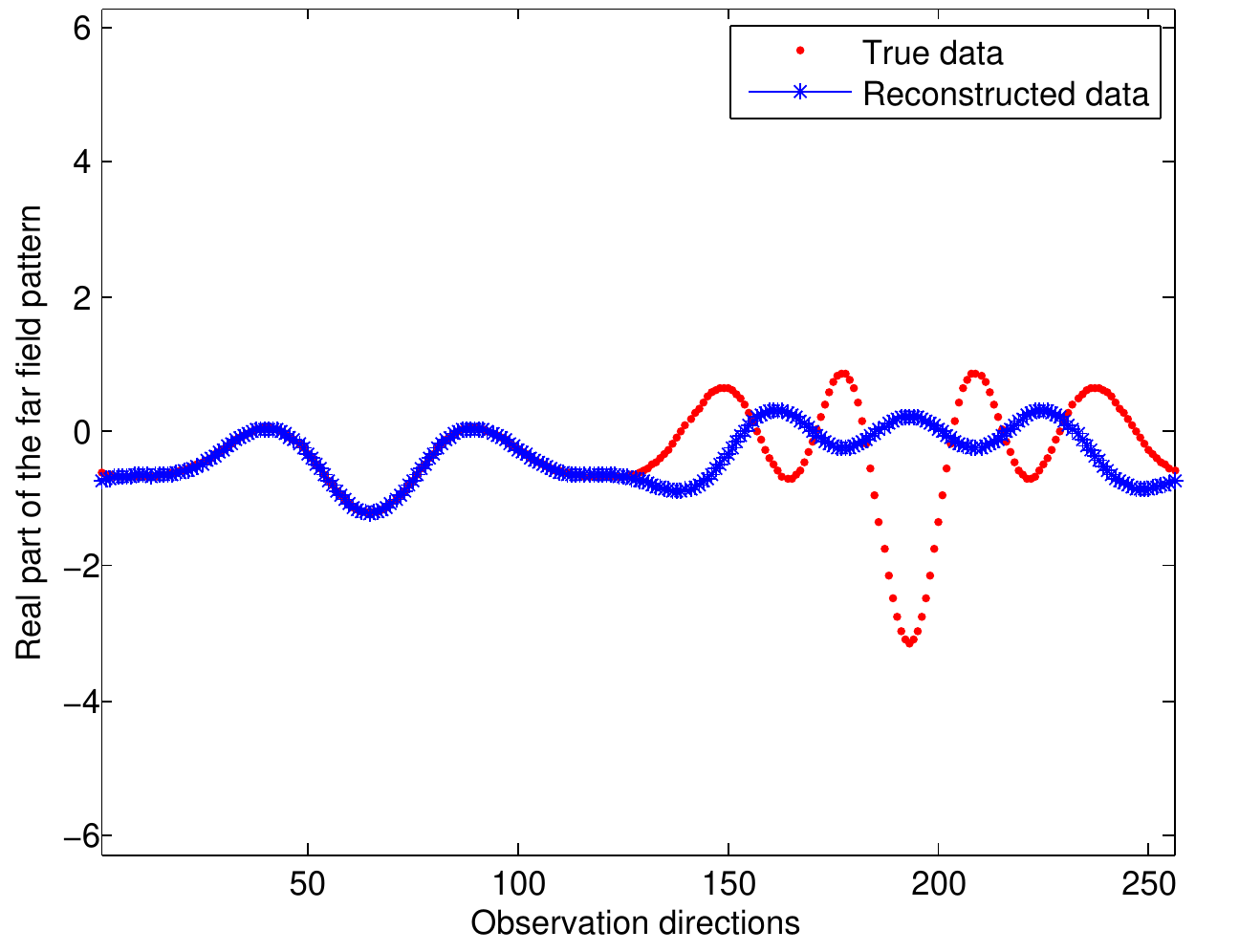}}
  \subfigure[\textbf{DC-FS with $J=39$}]{
    \includegraphics[width=2in]{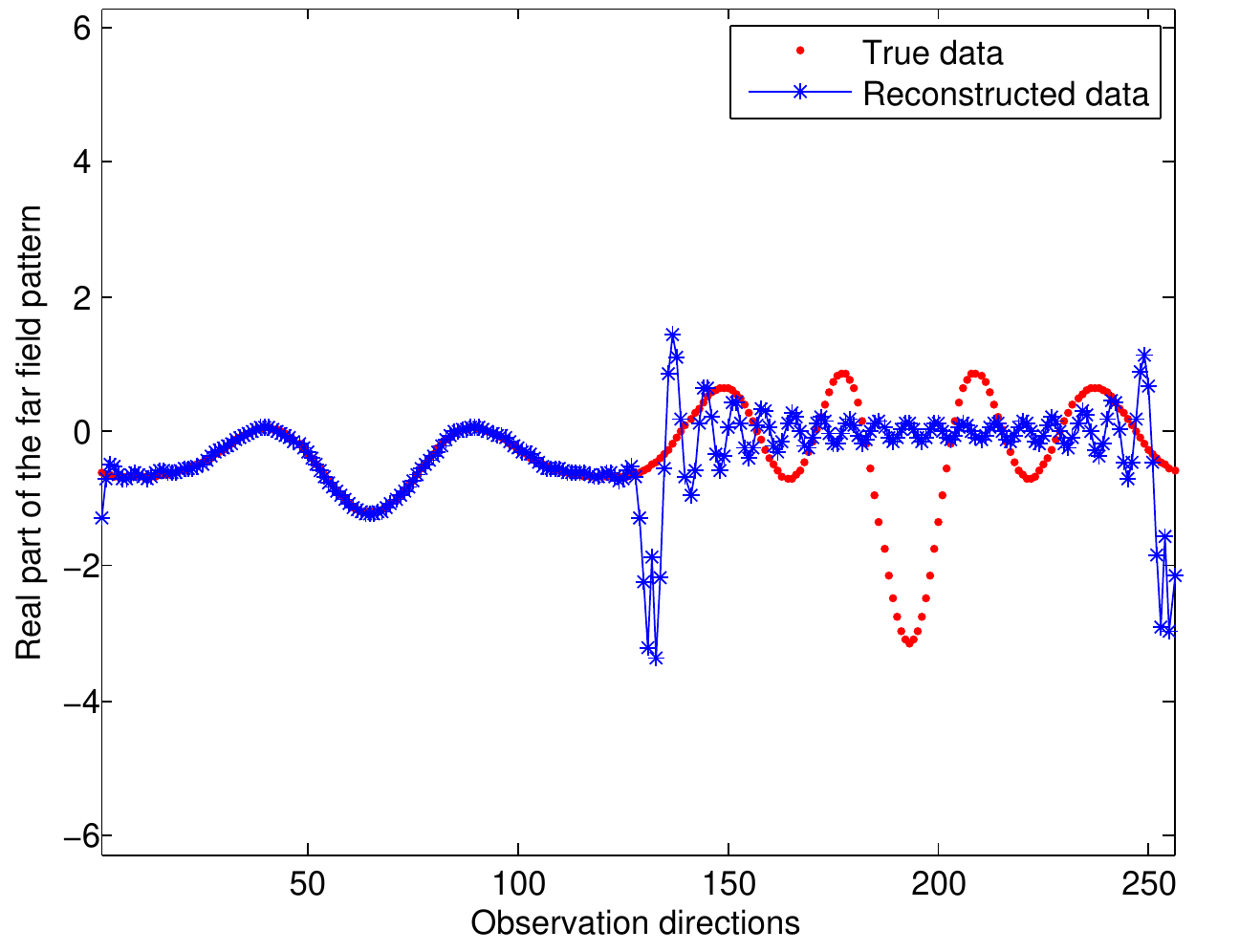}}\\
  \subfigure[\textbf{DC-IE with $J=4$}]{
    \includegraphics[width=2in]{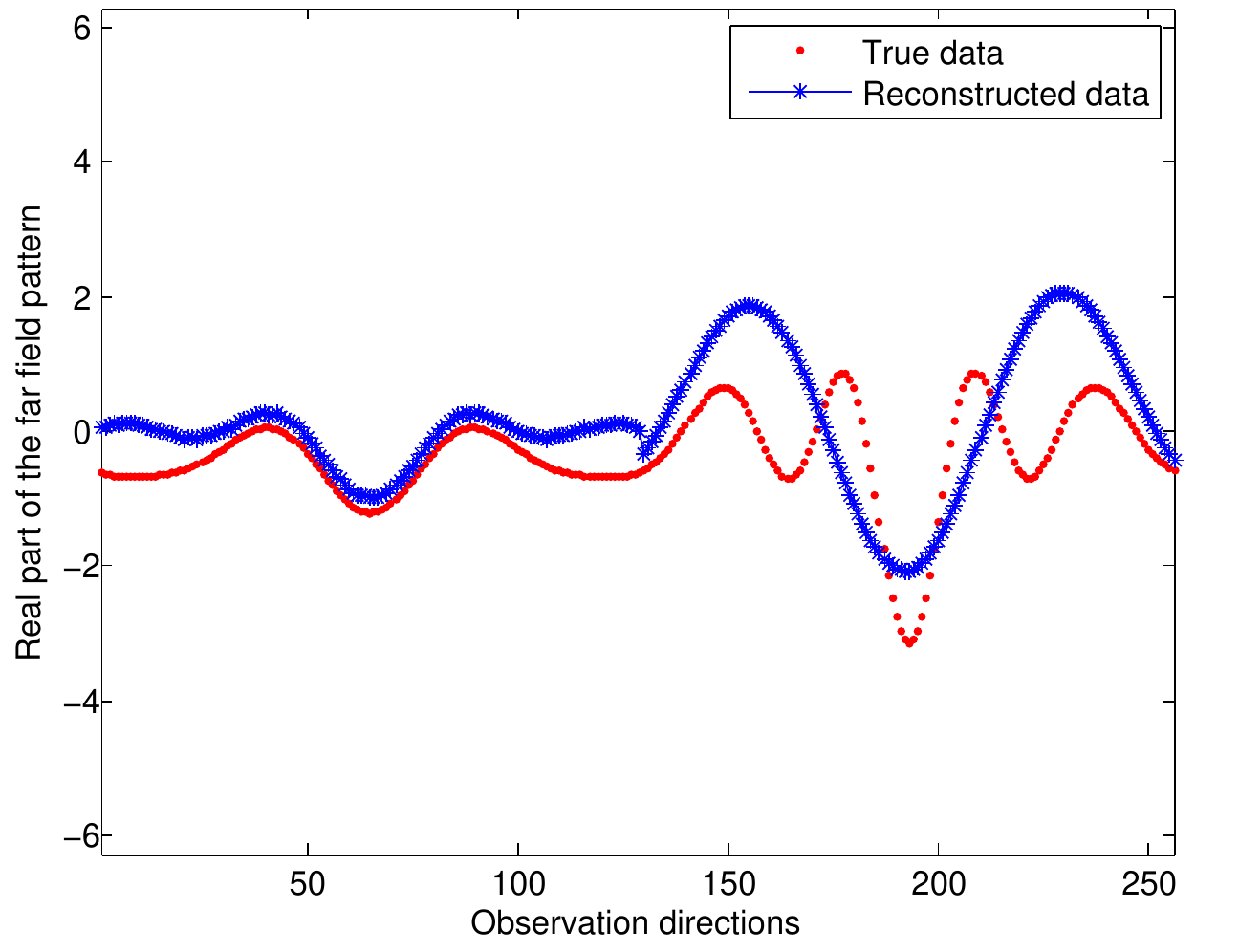}}
  \subfigure[\textbf{DC-IE with $J=9$}]{
    \includegraphics[width=2in]{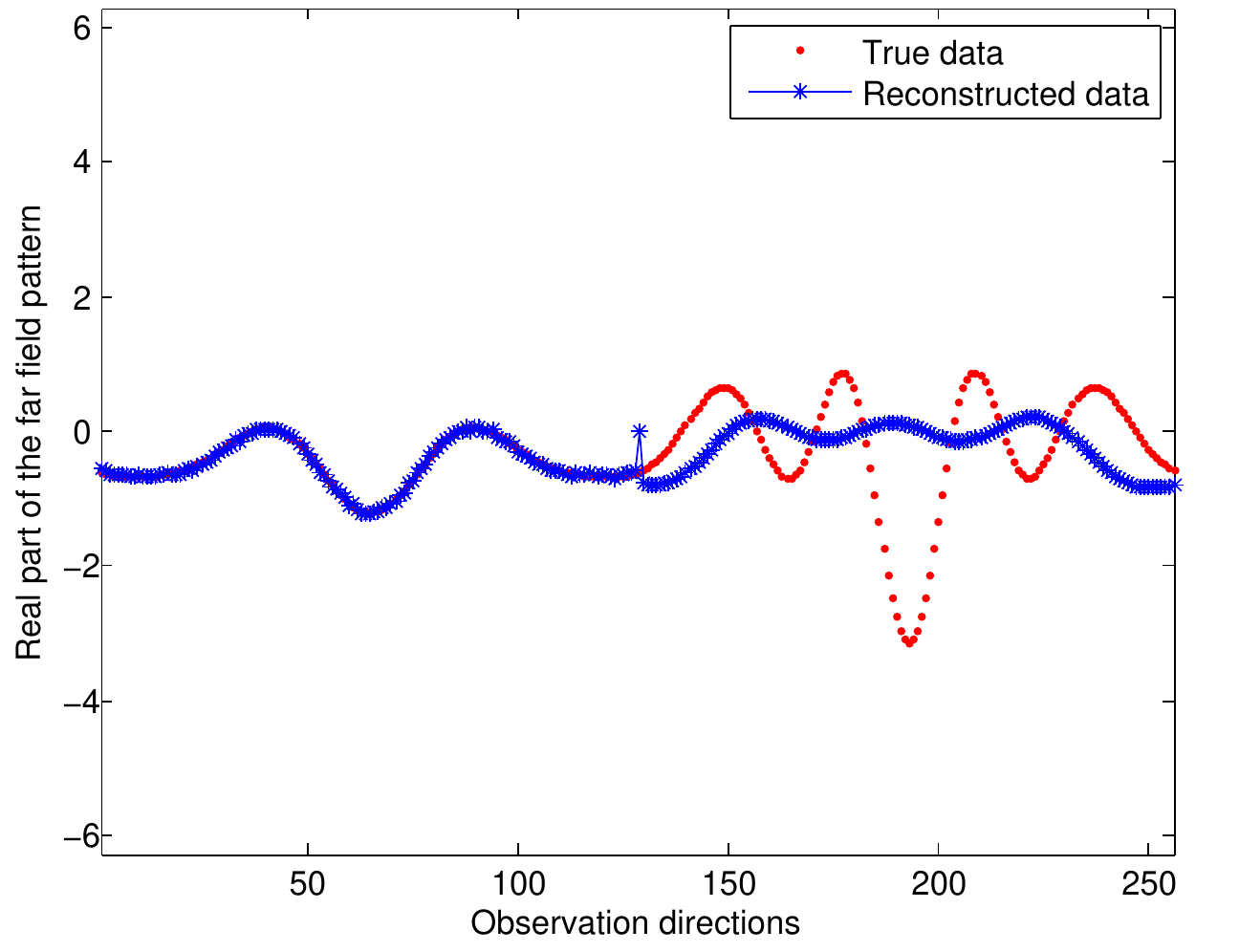}}
  \subfigure[\textbf{DC-IE with $J=39$}]{
    \includegraphics[width=2in]{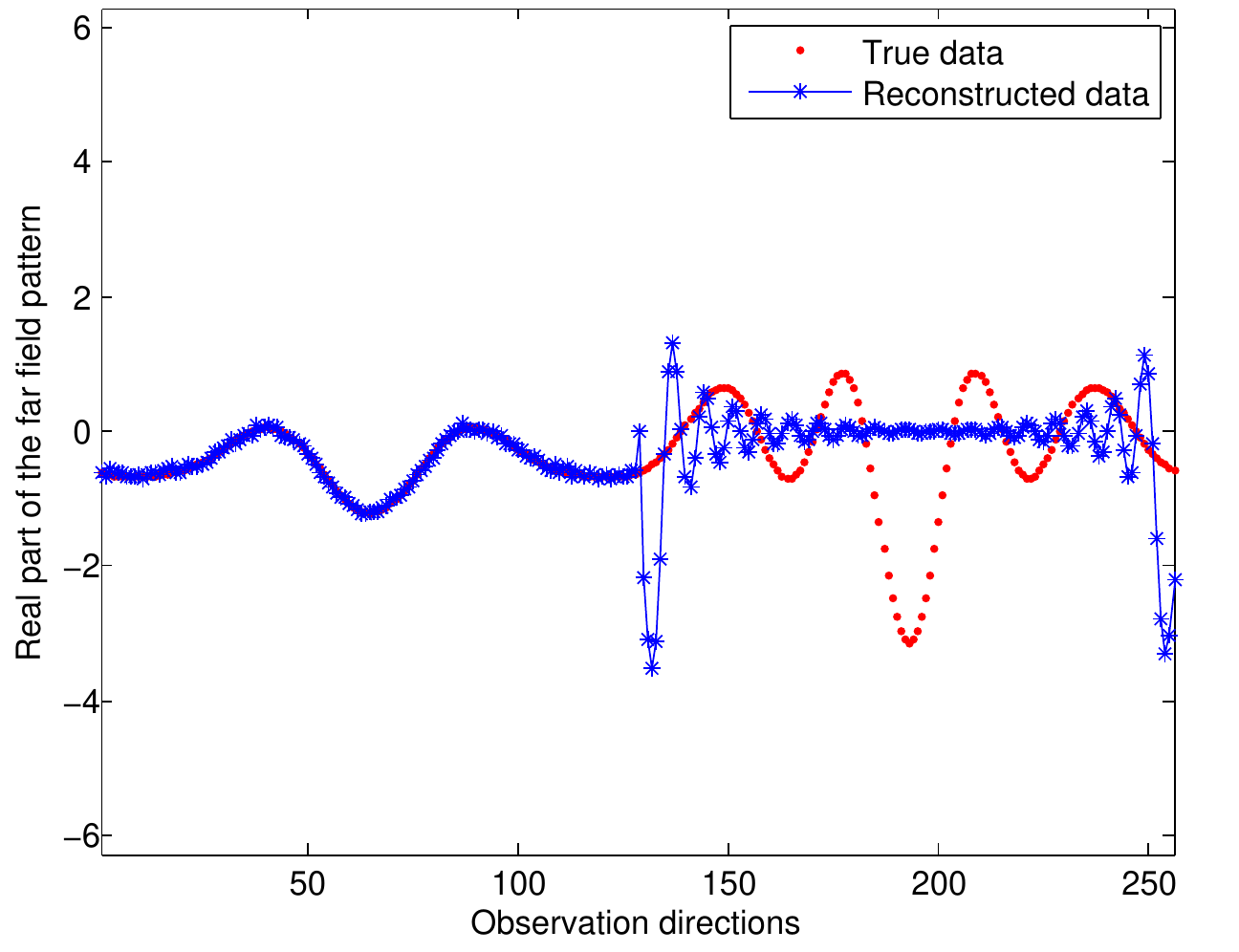}}
\caption{\bf Data completion for the far field pattern with incident direction $d=(0,-1)$. Regularization II \eqref{data completion Tikhonov 1} with parameter $\eps = 10^{-3}$.}
\label{DC-2}
\end{figure}

\begin{figure}[htbp]
  \centering
\subfigure[\textbf{DC-FS with $J=4$}]{
    \includegraphics[width=2in]{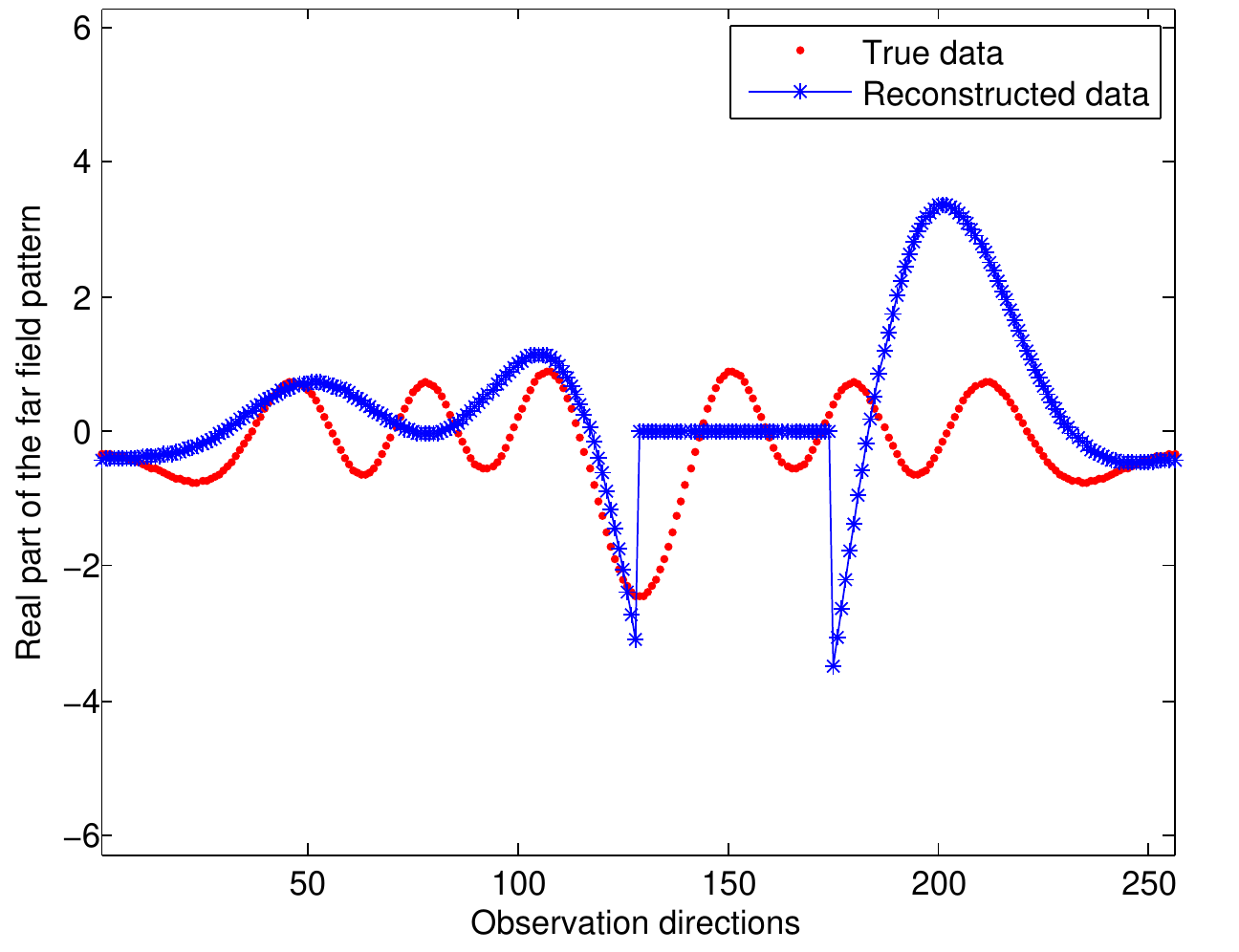}}
  \subfigure[\textbf{DC-FS with $J=9$}]{
    \includegraphics[width=2in]{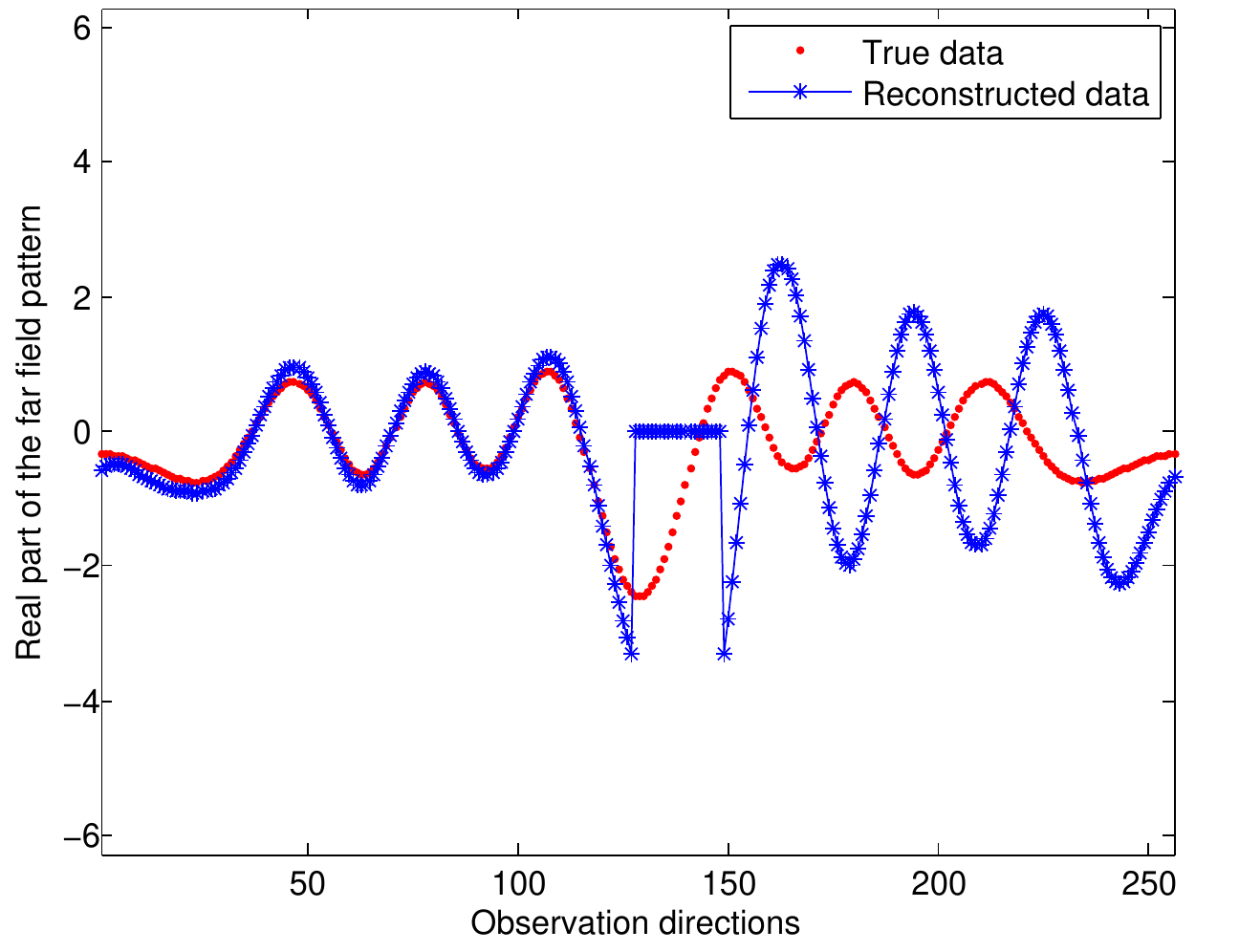}}
  \subfigure[\textbf{DC-FS with $J=39$}]{
    \includegraphics[width=2in]{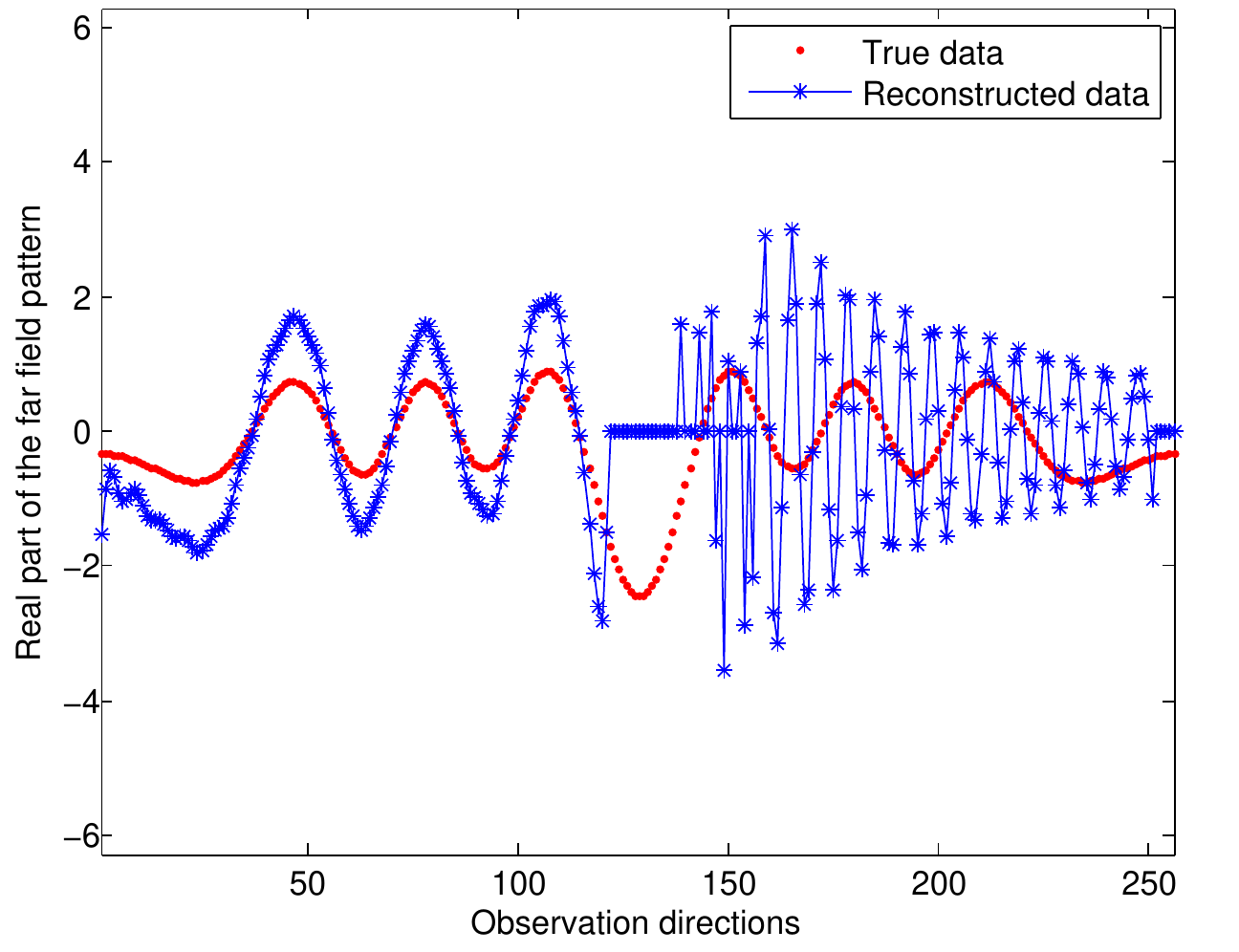}}\\
  \subfigure[\textbf{DC-IE with $J=4$}]{
    \includegraphics[width=2in]{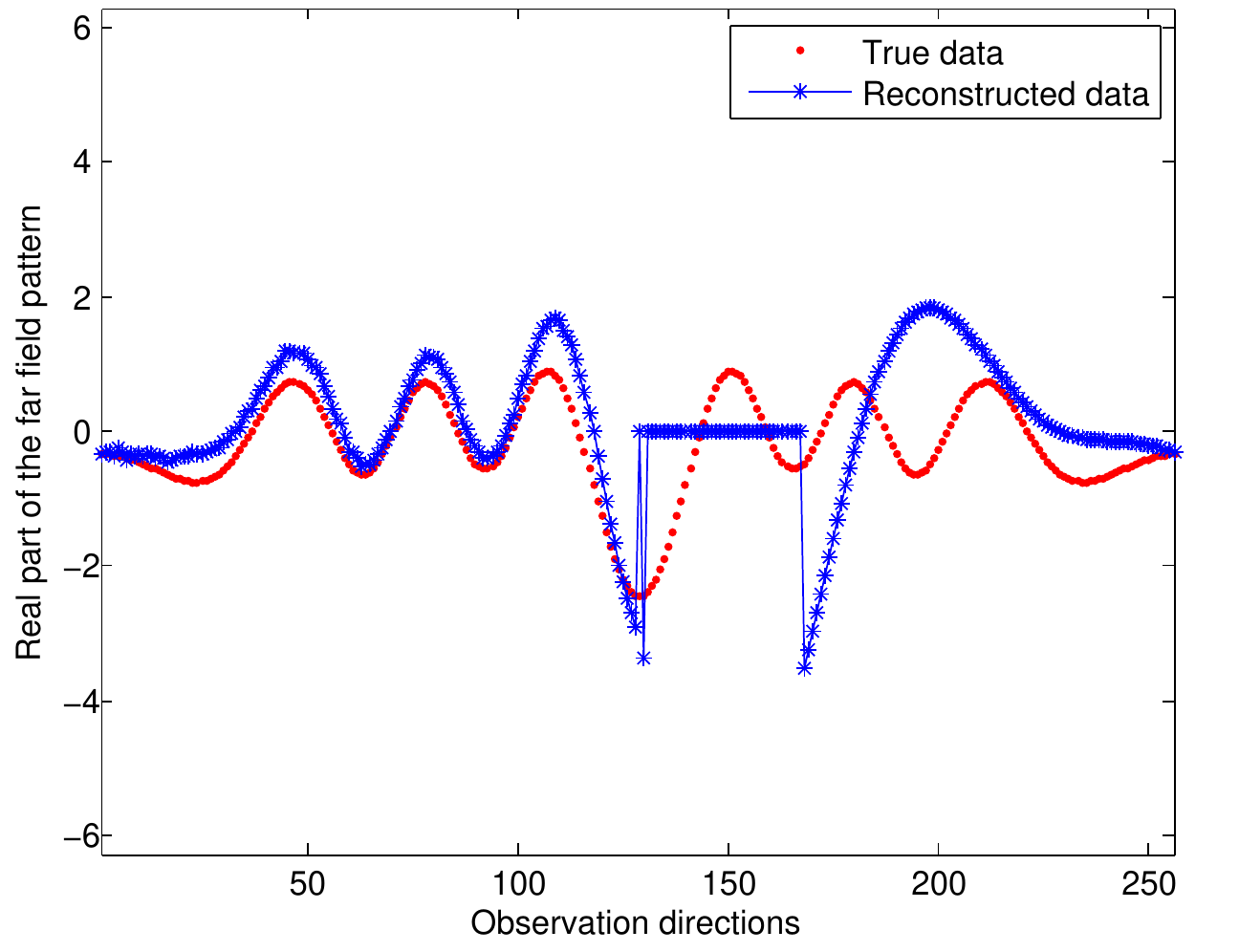}}
  \subfigure[\textbf{DC-IE with $J=9$}]{
    \includegraphics[width=2in]{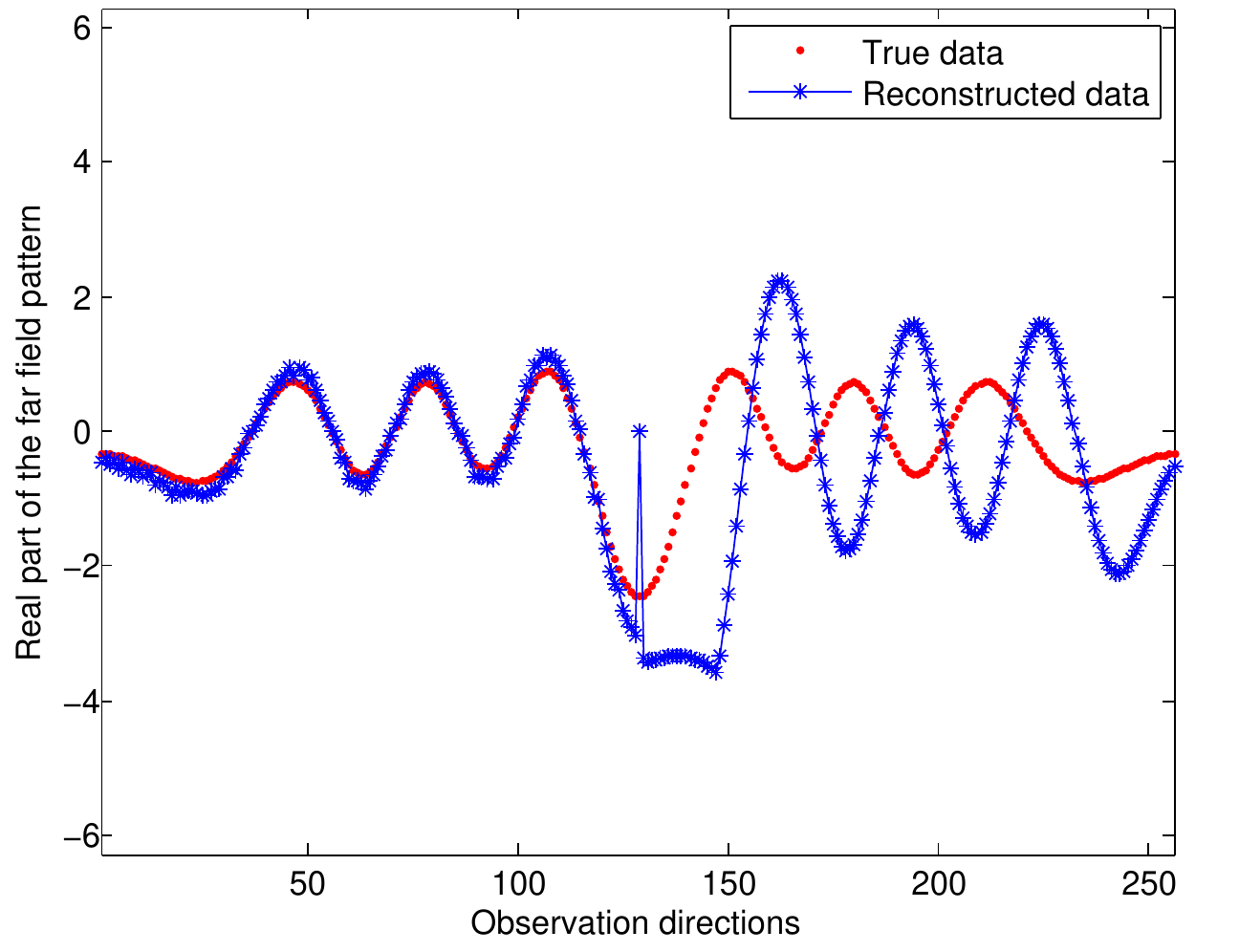}}
  \subfigure[\textbf{DC-IE with $J=39$}]{
    \includegraphics[width=2in]{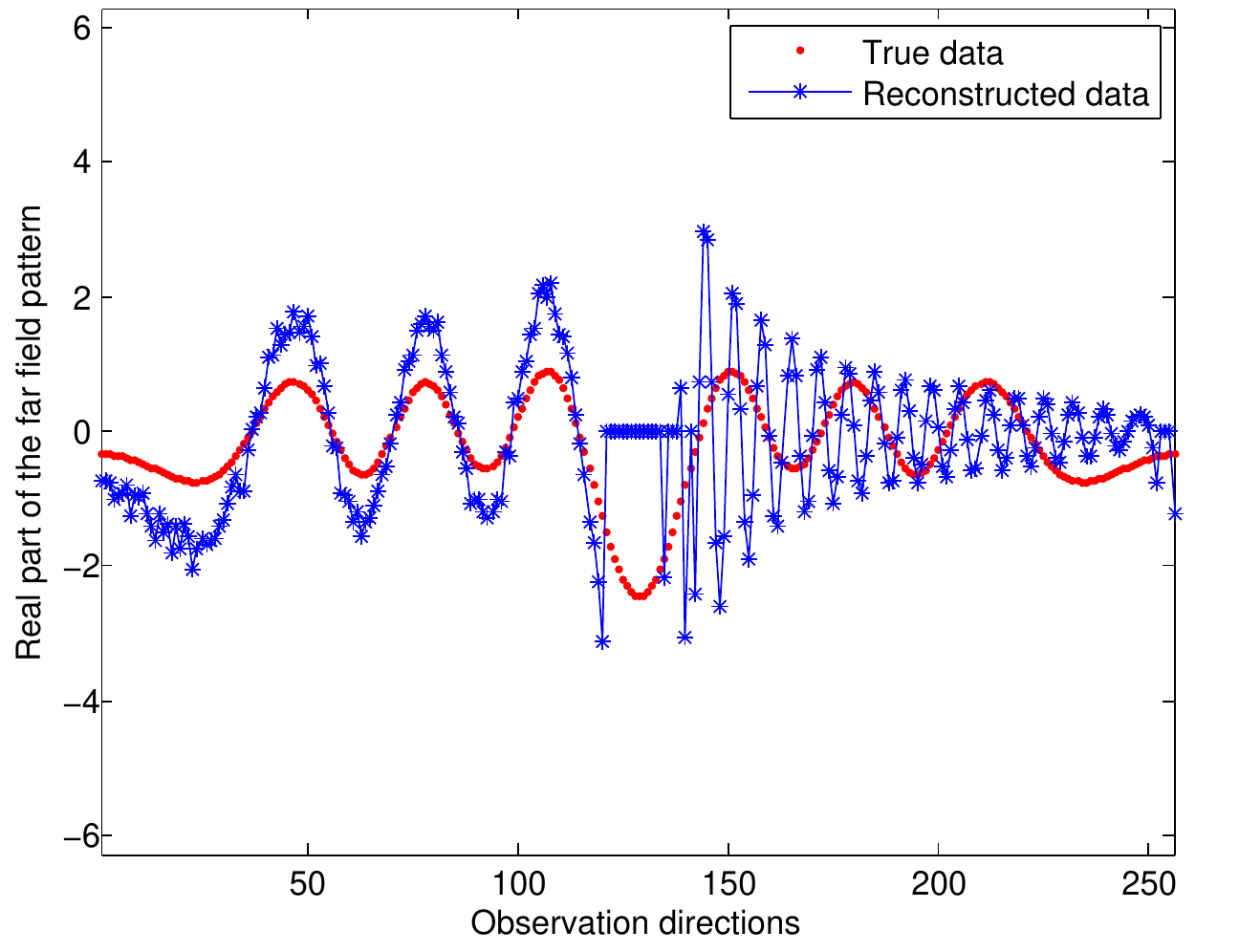}}
\caption{\bf Data completion for the far field pattern with incident direction $d=(-1,0)$. Regularization II \eqref{data completion Tikhonov 1}  with parameter $\eps = 10^{-3}$.}
\label{DC-1}
\end{figure}

\begin{figure}[htbp]
  \centering
\subfigure[\textbf{DC-FS with $J=4$}]{
    \includegraphics[width=2in]{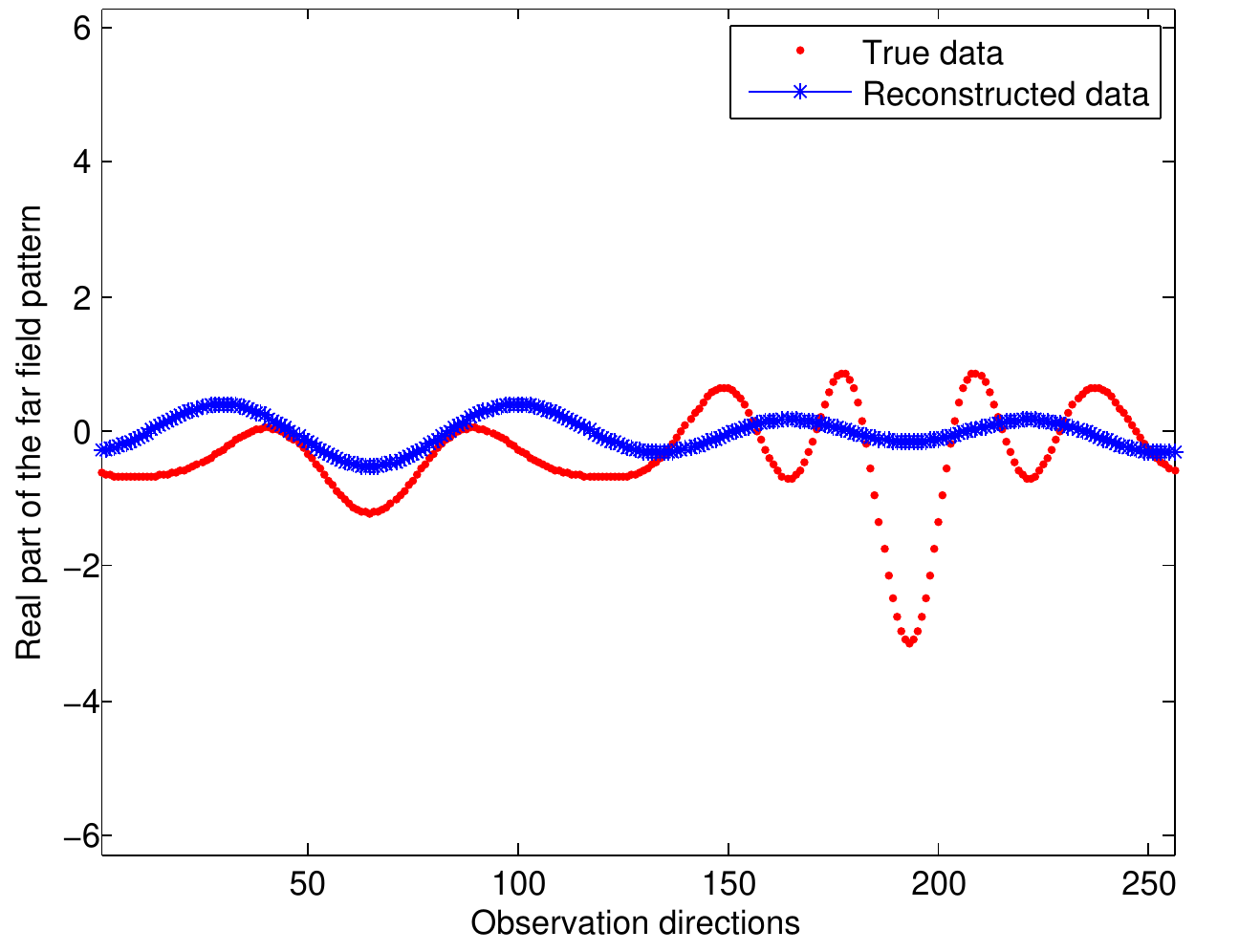}}
  \subfigure[\textbf{DC-FS with $J=9$}]{
    \includegraphics[width=2in]{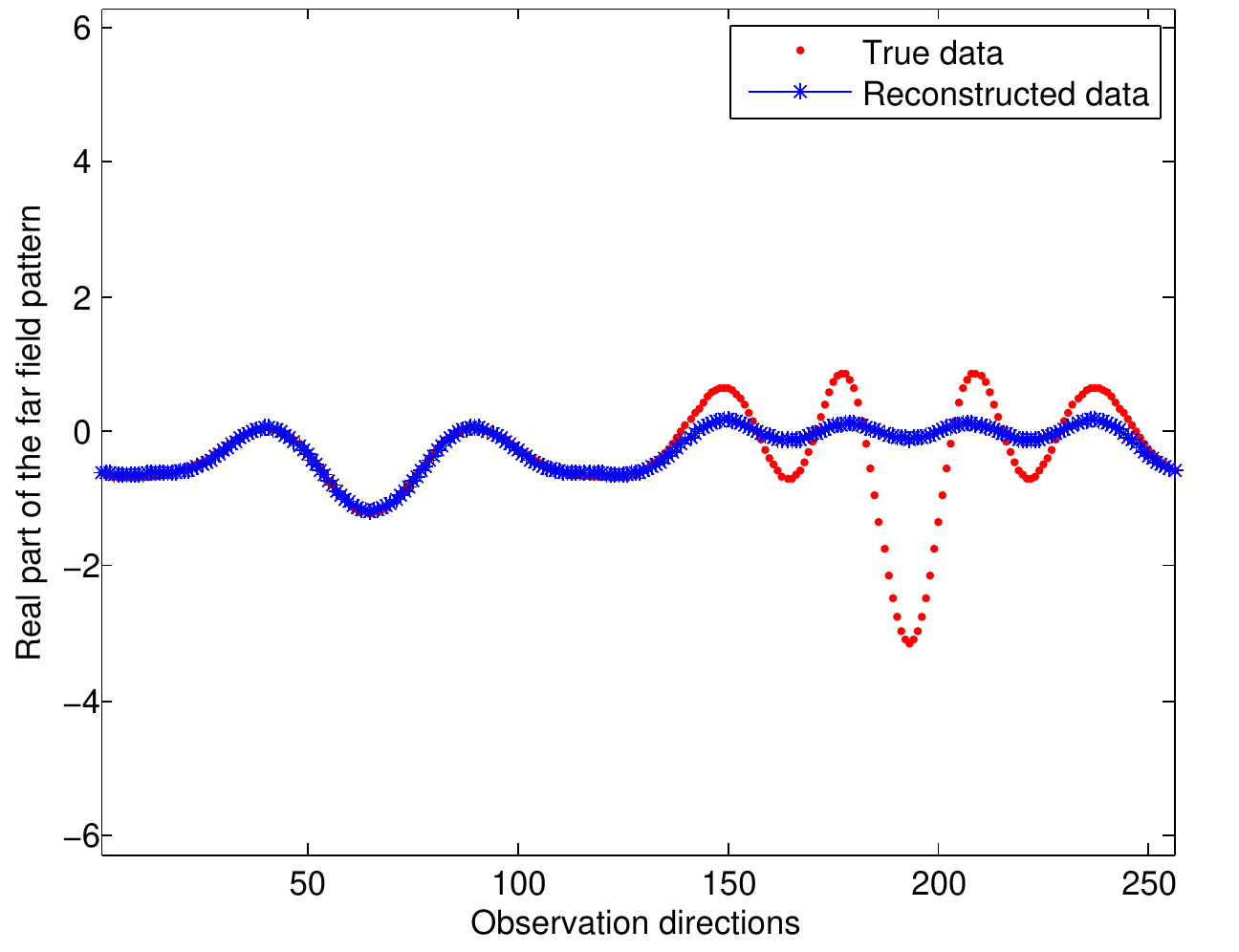}}
  \subfigure[\textbf{DC-FS with $J=39$}]{
    \includegraphics[width=2in]{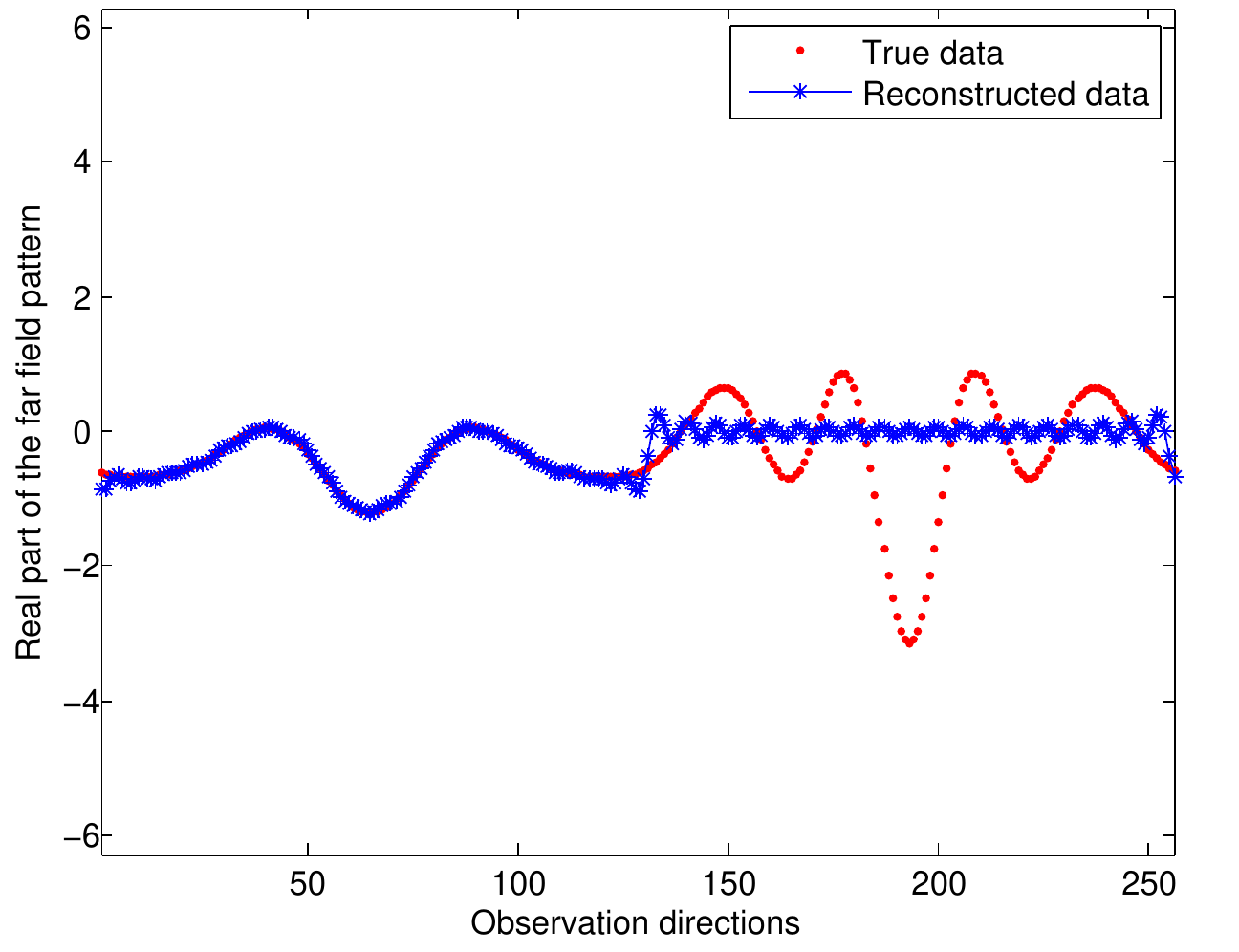}}\\
  \subfigure[\textbf{DC-IE with $J=4$}]{
    \includegraphics[width=2in]{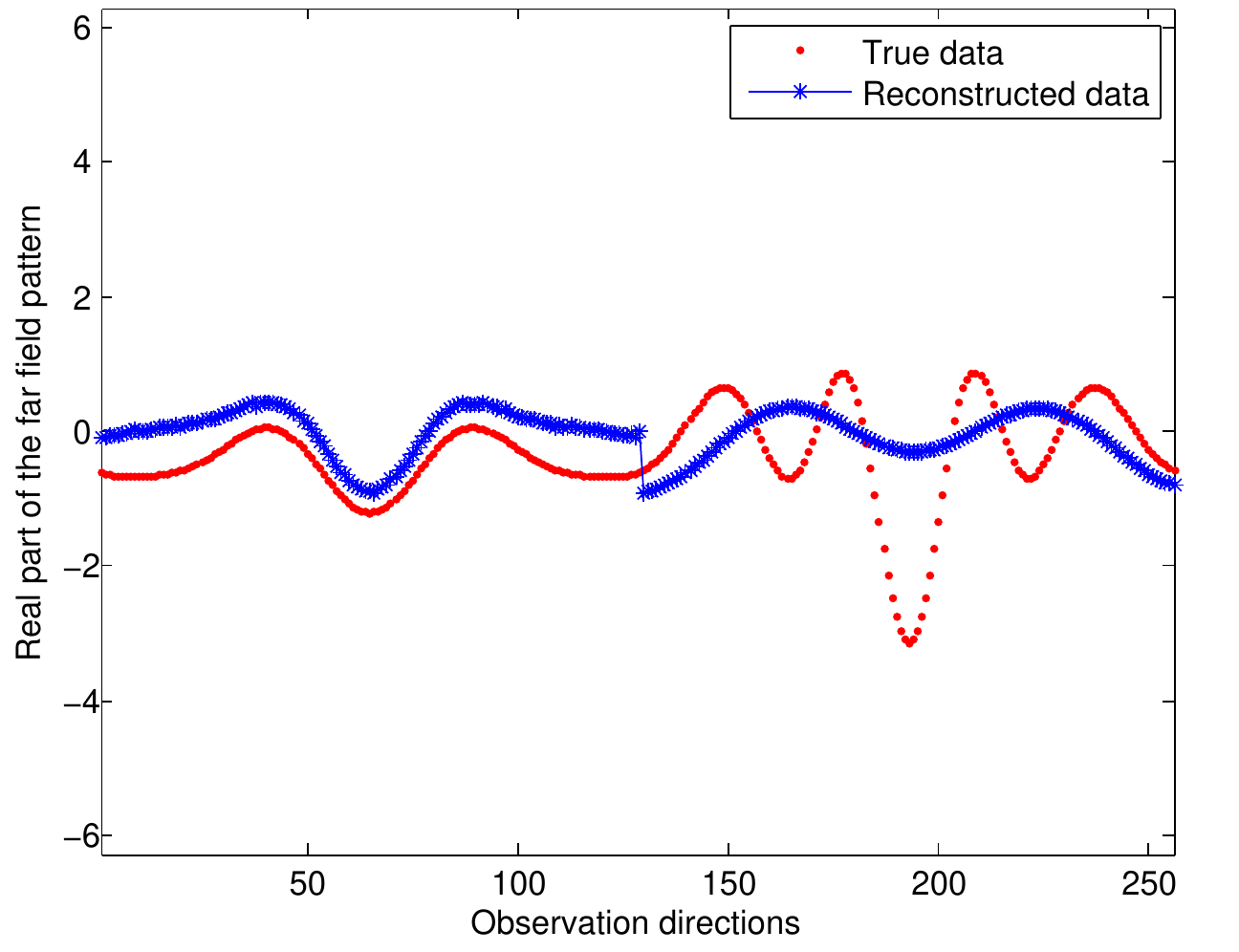}}
  \subfigure[\textbf{DC-IE with $J=9$}]{
    \includegraphics[width=2in]{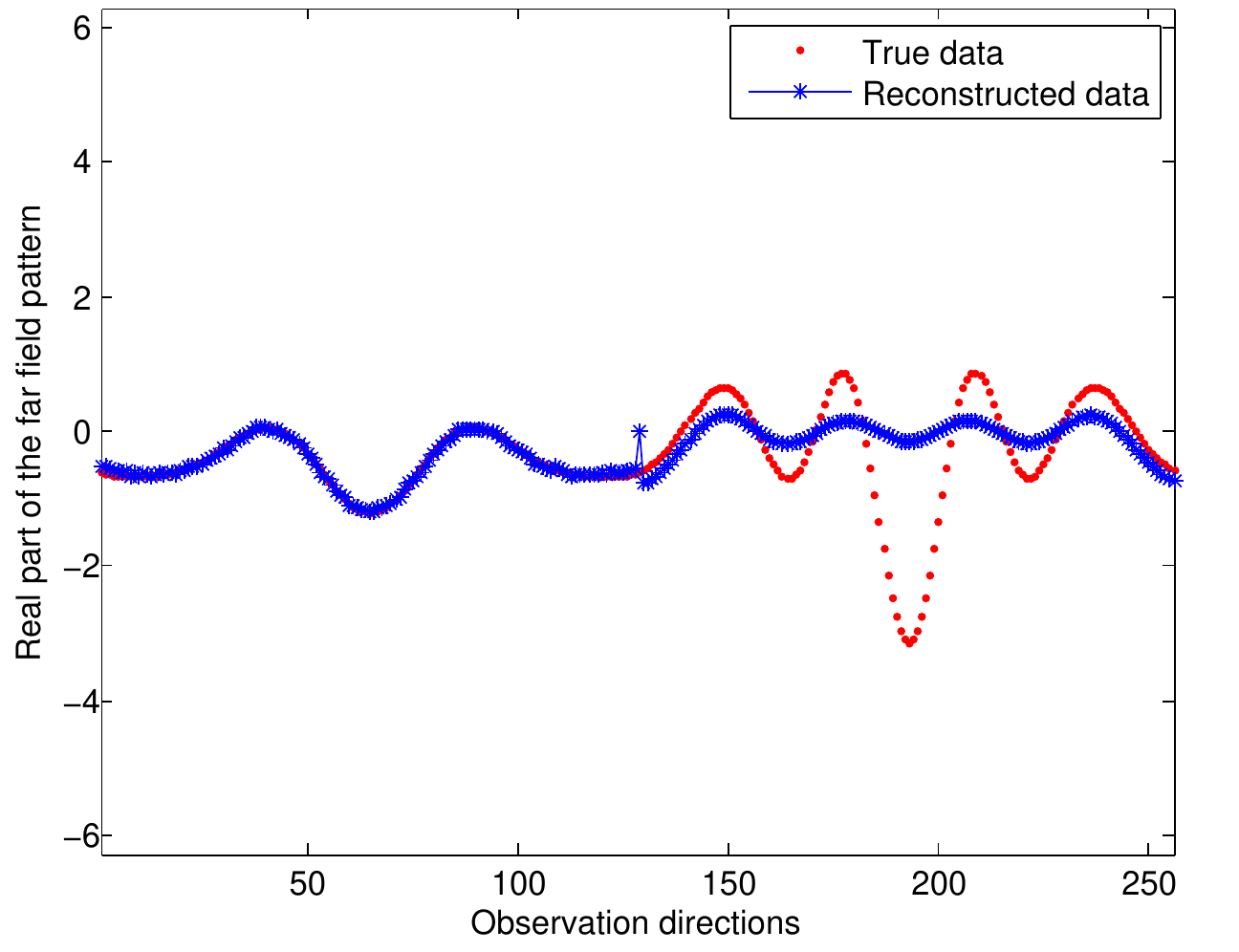}}
  \subfigure[\textbf{DC-IE with $J=39$}]{
    \includegraphics[width=2in]{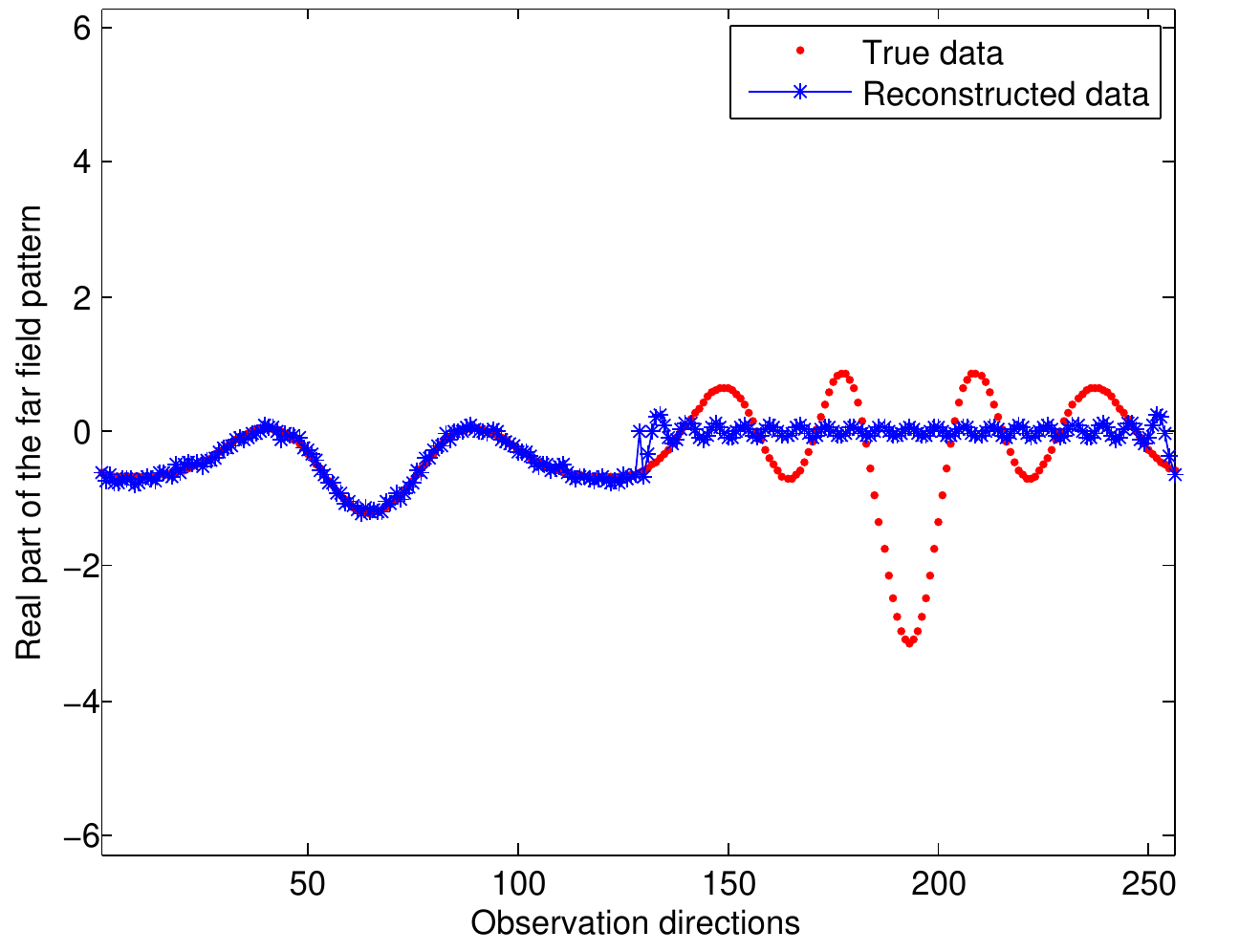}}
\caption{\bf Data completion for the far field pattern with incident direction $d=(0,-1)$. Regularization I \eqref{data completion TSVD} with   cut-off value $\sigma = 0.1$.}
\label{TSVD-DC-2}
\end{figure}
\begin{figure}[htbp]
  \centering
\subfigure[\textbf{DC-FS with $J=4$}]{
    \includegraphics[width=2in]{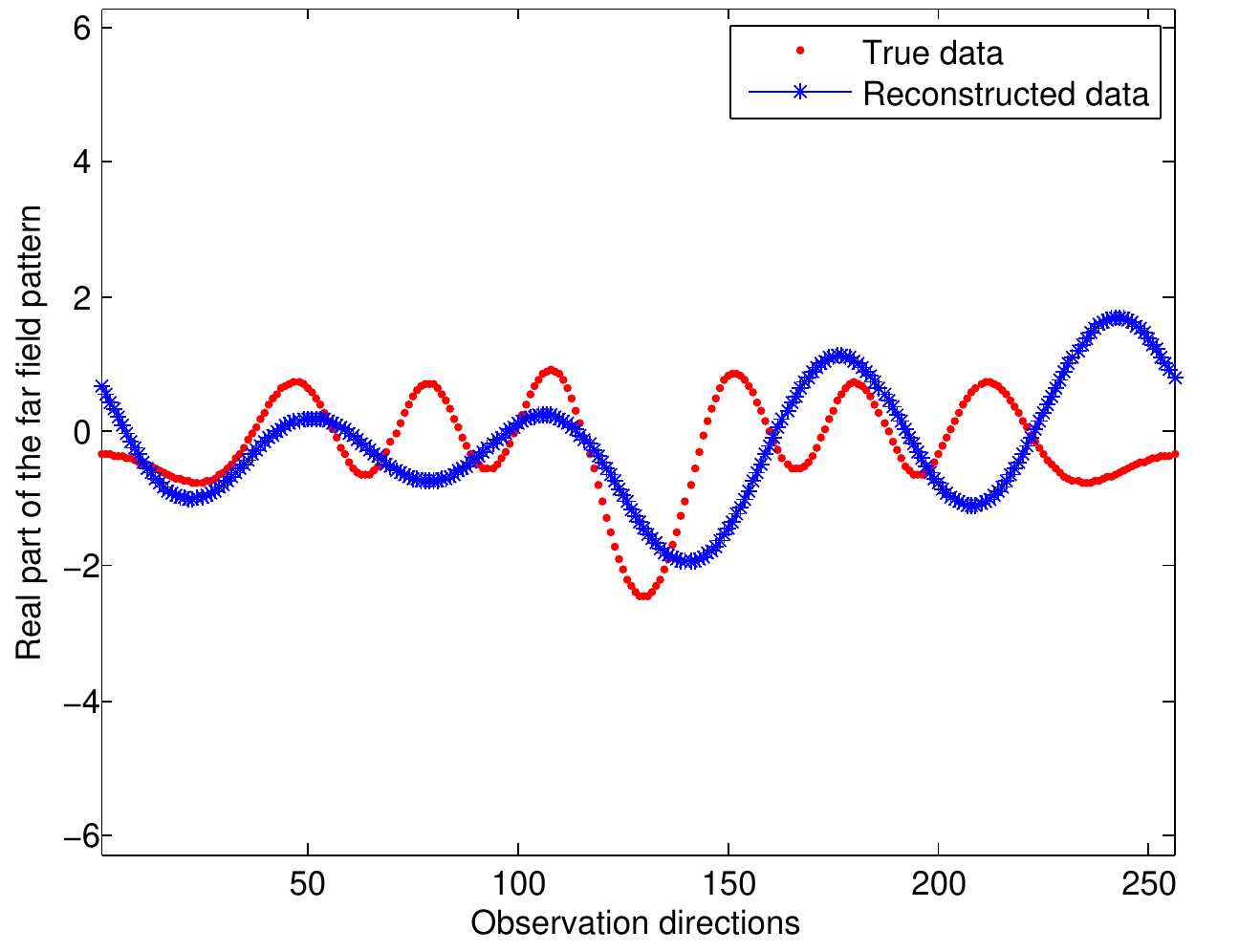}}
  \subfigure[\textbf{DC-FS with $J=9$}]{
    \includegraphics[width=2in]{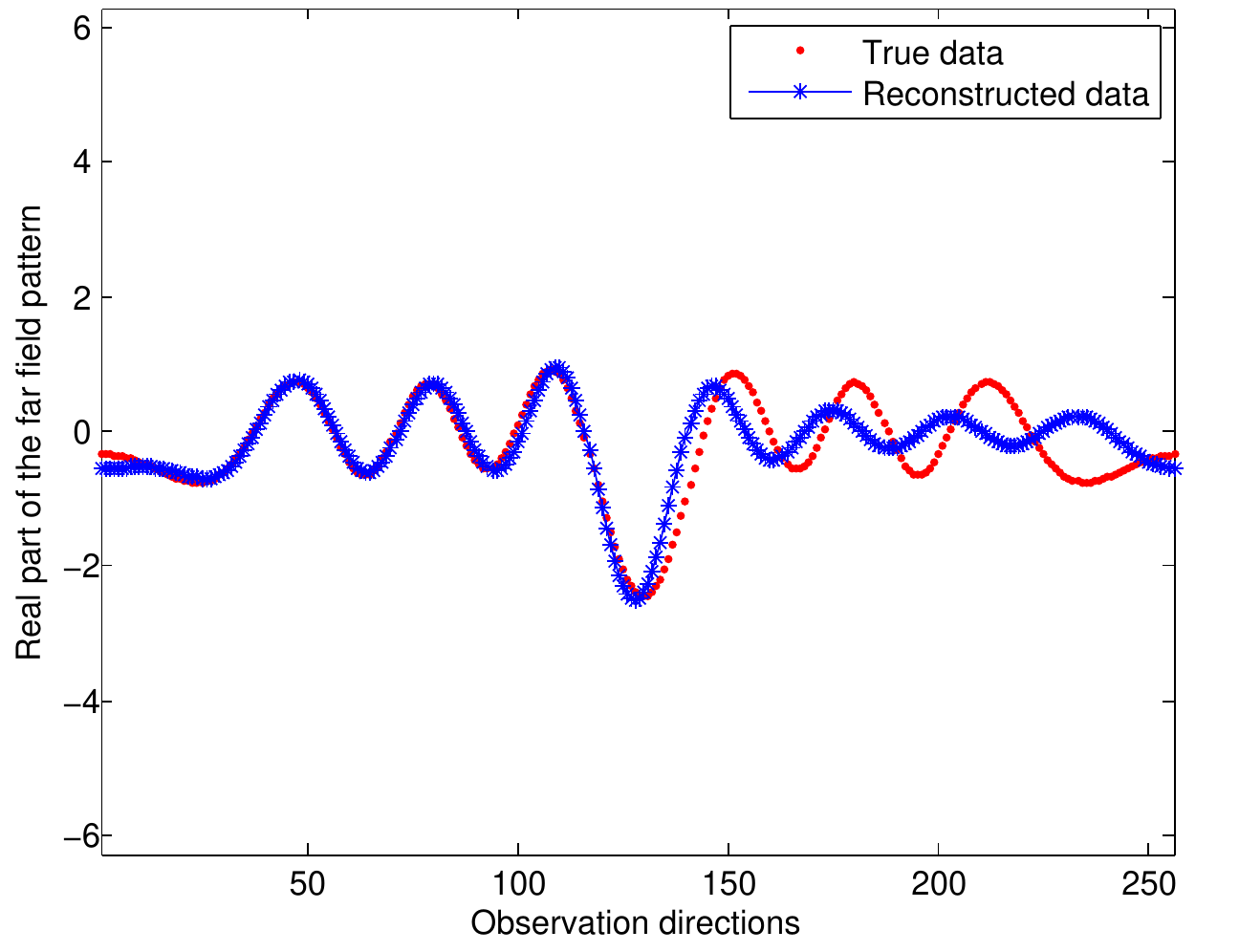}}
  \subfigure[\textbf{DC-FS with $J=39$}]{
    \includegraphics[width=2in]{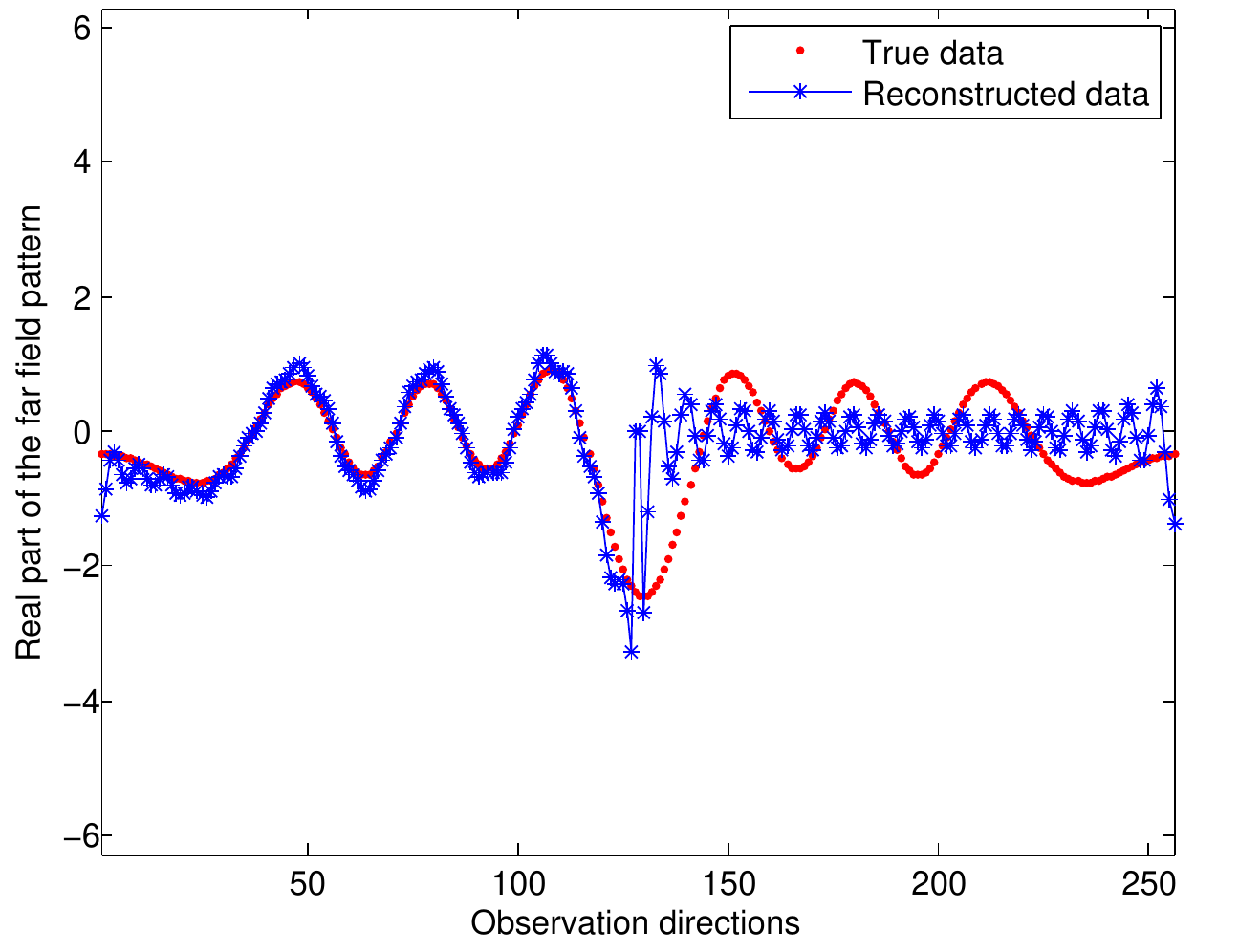}}\\
  \subfigure[\textbf{DC-IE with $J=4$}]{
    \includegraphics[width=2in]{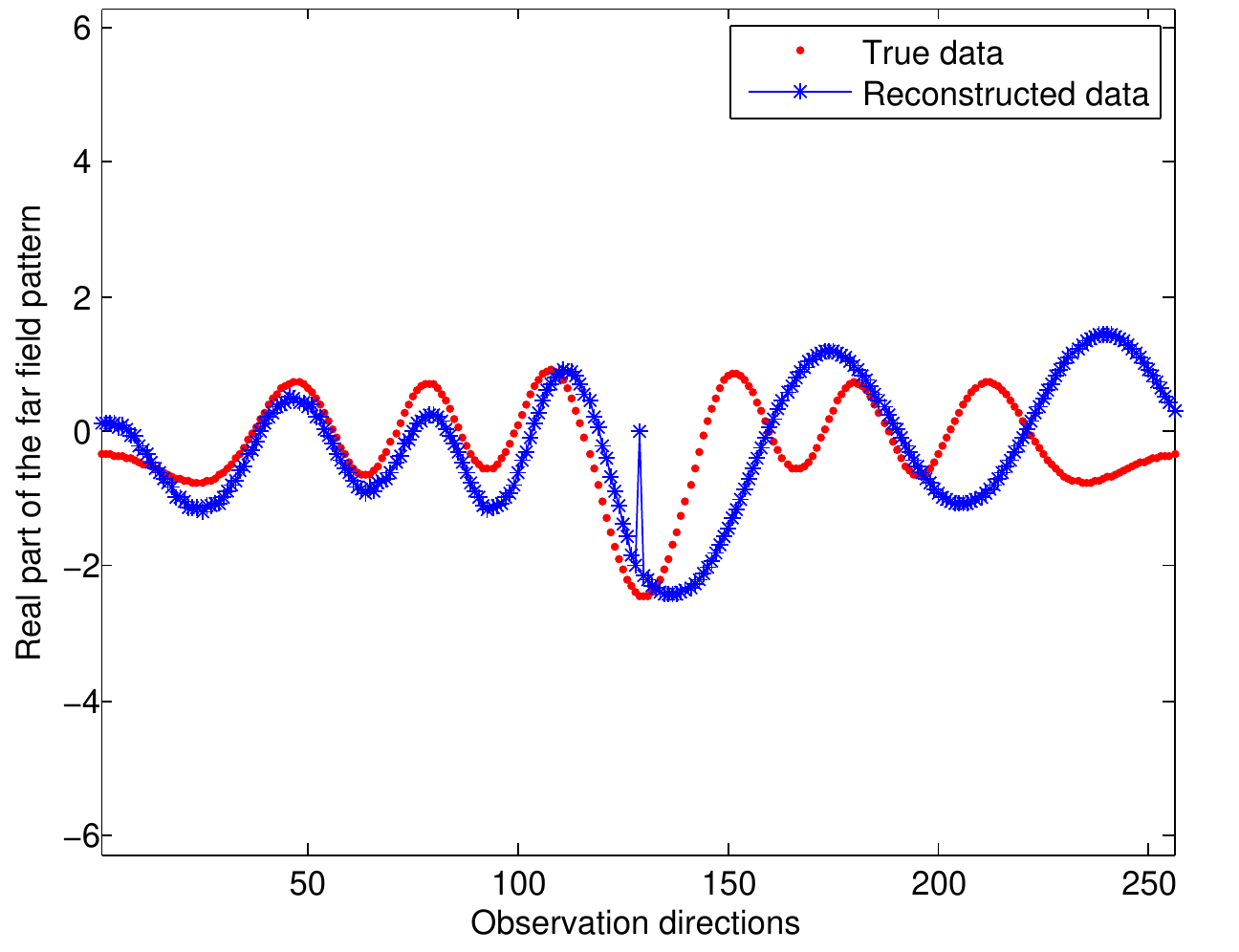}}
  \subfigure[\textbf{DC-IE with $J=9$}]{
    \includegraphics[width=2in]{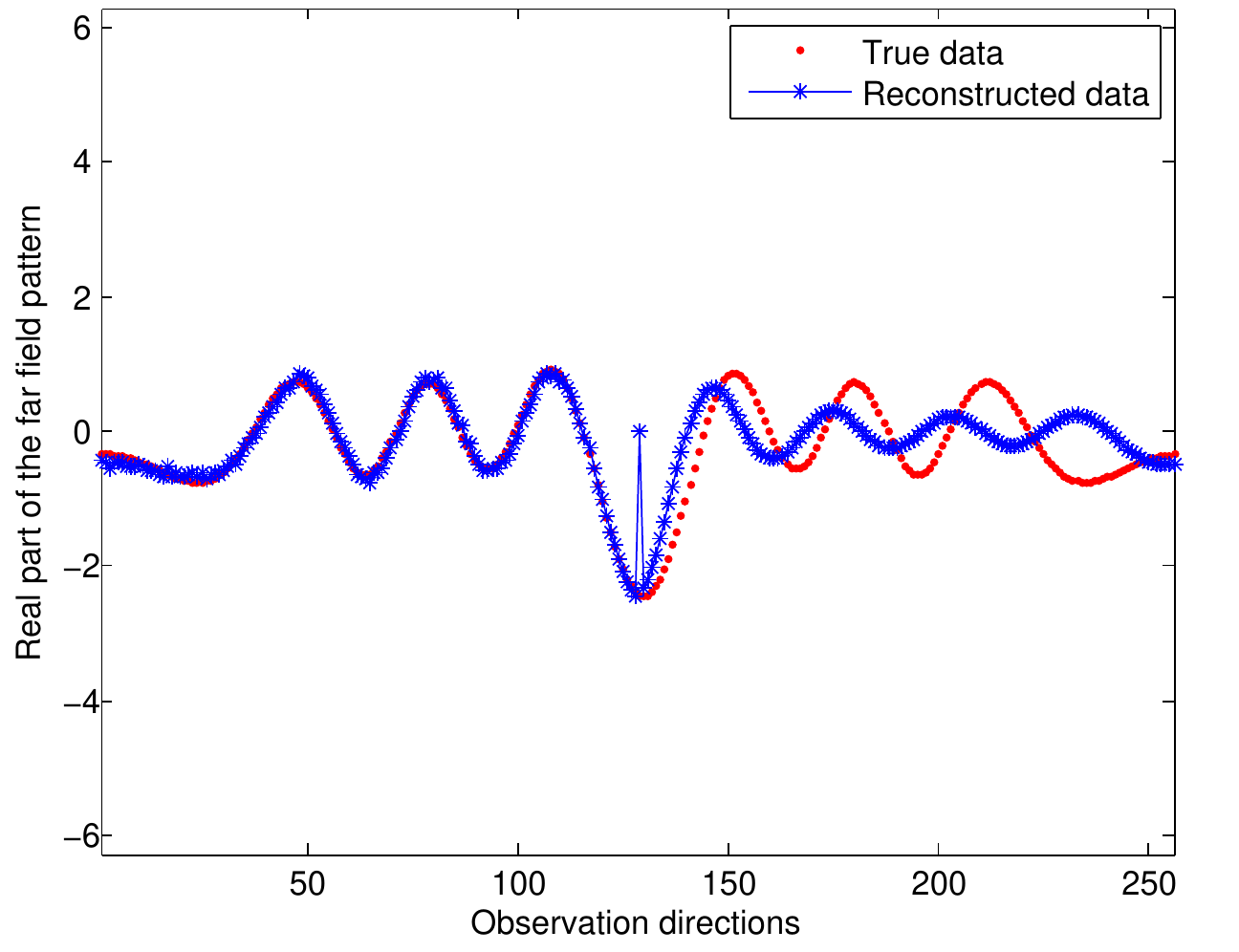}}
  \subfigure[\textbf{DC-IE with $J=39$}]{
    \includegraphics[width=2in]{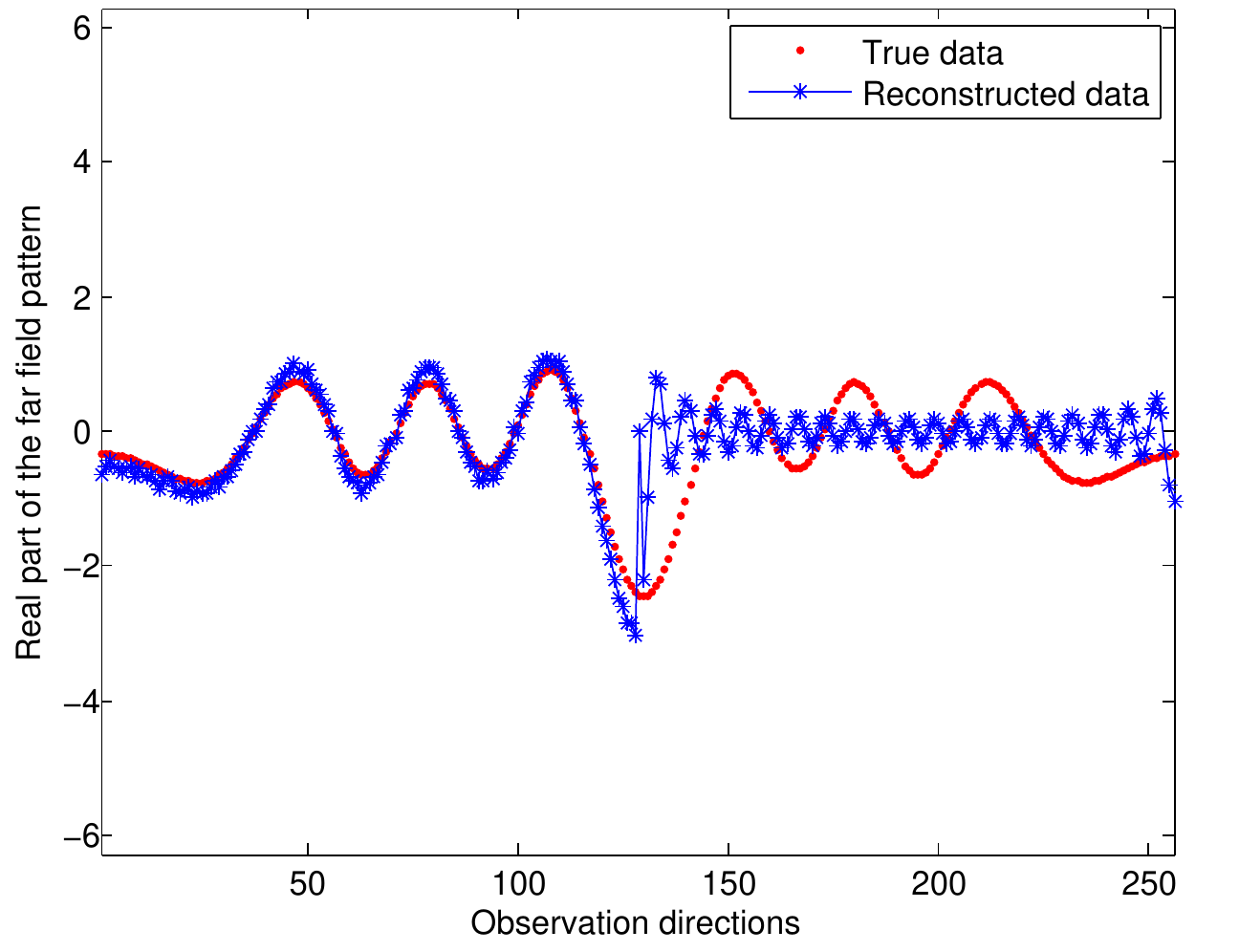}}
\caption{\bf Data completion for the far field pattern with incident direction $d=(-1,0)$. Truncated SVD is used  with the cut-off value $\sigma = 0.1$.}
\label{TSVD-DC-1}
\end{figure}

\subsection{Object reconstruction results and discussions}
We consider the following four object reconstruction methods
\begin{itemize}
  \item DSM using the limited aperture backscattering data directly;
  \item DSM with {\bf DC-FS};
  \item DSM with {\bf DC-IE};
  \item FM with {\bf DC-IE}.
\end{itemize}
For comparison of these reconstruction methods, we take $J=9$ when using the data completion algorithms.

If not otherwise stated, for the data completion algorithm, we use Regularization II \eqref{data completion Tikhonov 1}  with parameter $\eps = 10^{-3}$.
Figure \ref{Peanut-D} shows the reconstructions of a sound-soft peanut with the above mentioned  object reconstruction methods.
Obviously, the reconstructions improve with the help of the data completion algorithms. This can also be found in Figure \ref{Disk-D} for a sound-soft disk.
Numerically, the DSM seems to give better reconstructions, in particular on the illuminated parts.
To illustrate the performance of imaging using different regularizations, Figure \ref{Peanut-D-2} shows the reconstructions using {\bf DC-IE} with TSVD    and Tikhonov regularization (with Morozov discrepancy principle). It is observed that Figure \ref{Peanut-D-2} is comparable to Figure \ref{Peanut-D}.

We emphasize again that our data completion algorithms and object reconstruction methods are independent of the topological and physical properties of the unknown objects. To illustrate this, we show in Figures \ref{Peanut-N}-\ref{Disk-N}  the  reconstruction results with Neumann boundary condition. It is observed that the lower half parts can also be well reconstructed, which further illustrate the potential of the data completion and imaging algorithms.

\begin{figure}[htbp]
  \centering
\subfigure[\textbf{DSM using partial data directly}]{
    \includegraphics[width=2in]{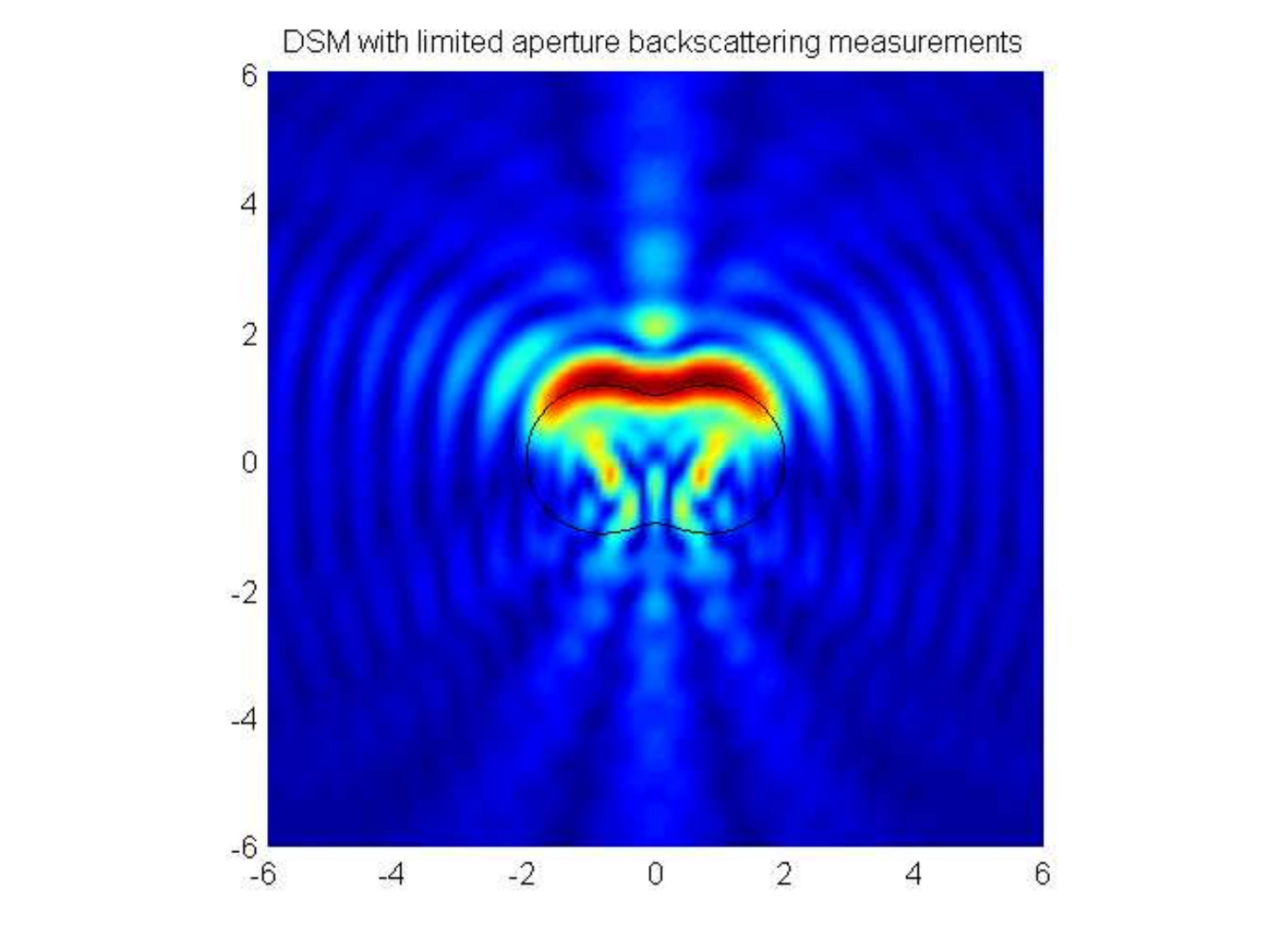}}
  \subfigure[\textbf{DSM with DC-FS}]{
    \includegraphics[width=2in]{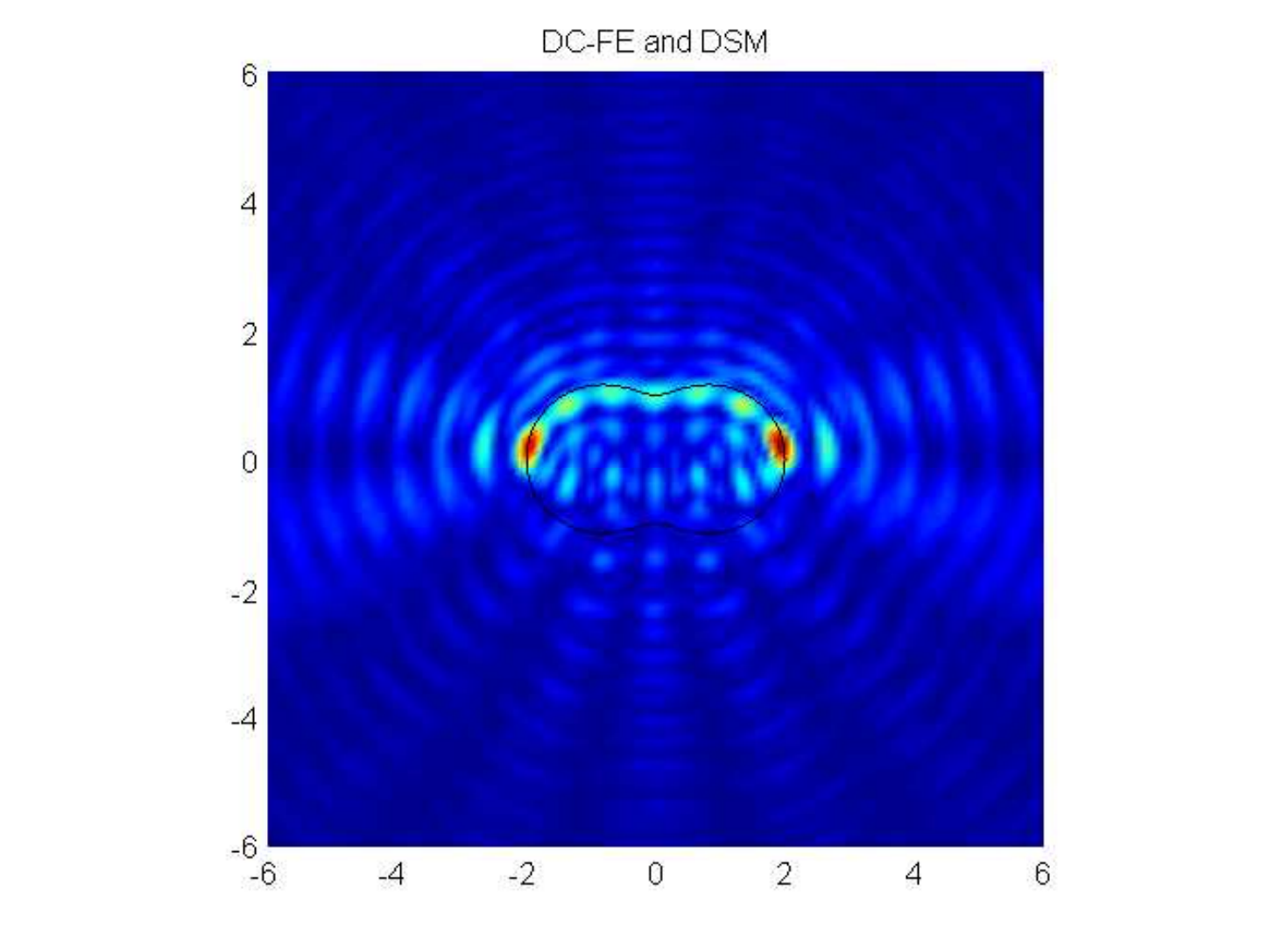}}\\
  \subfigure[\textbf{DSM with DC-IE}]{
    \includegraphics[width=2in]{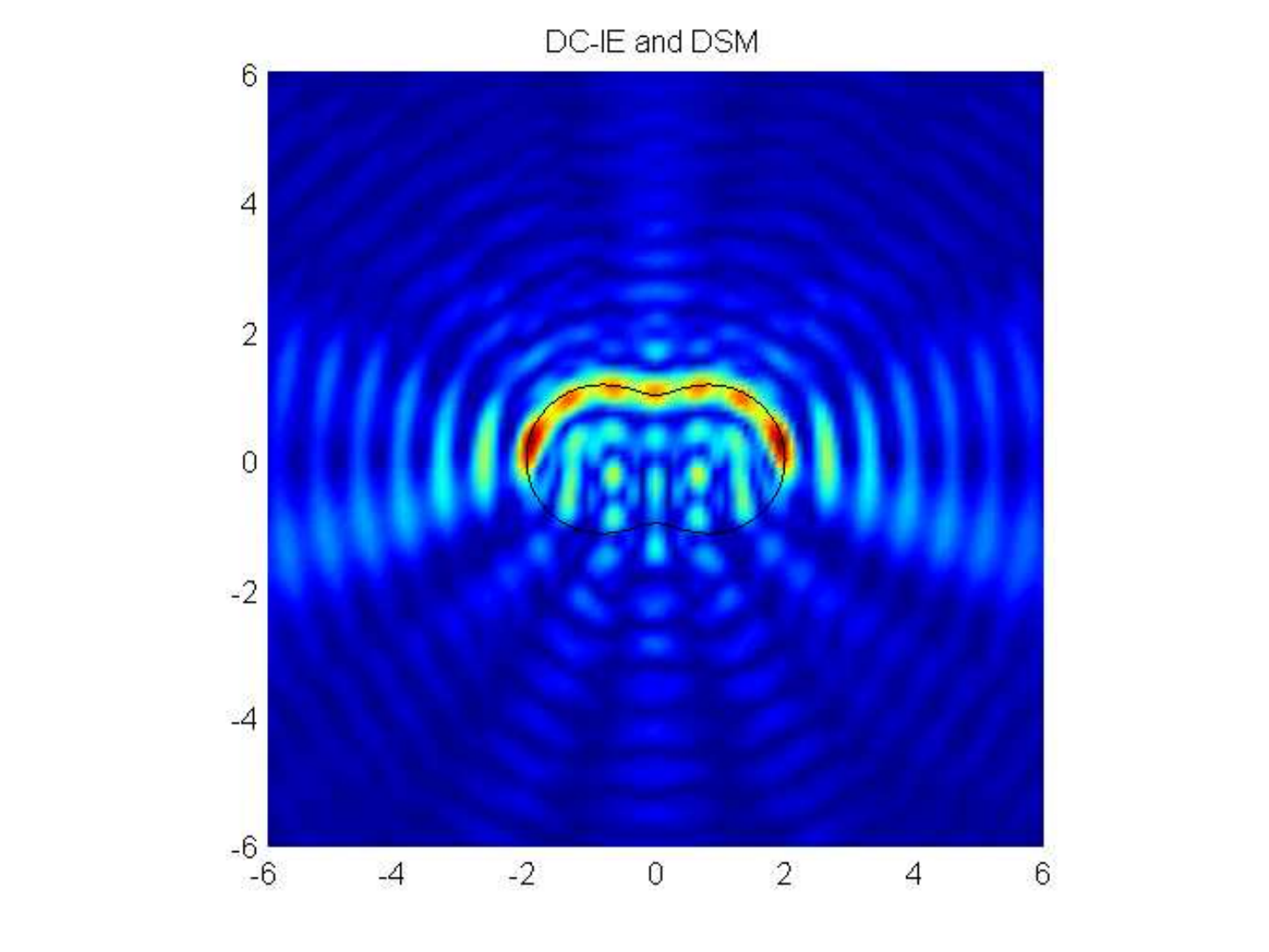}}
  \subfigure[\textbf{FM with DC-IE}]{
    \includegraphics[width=2in]{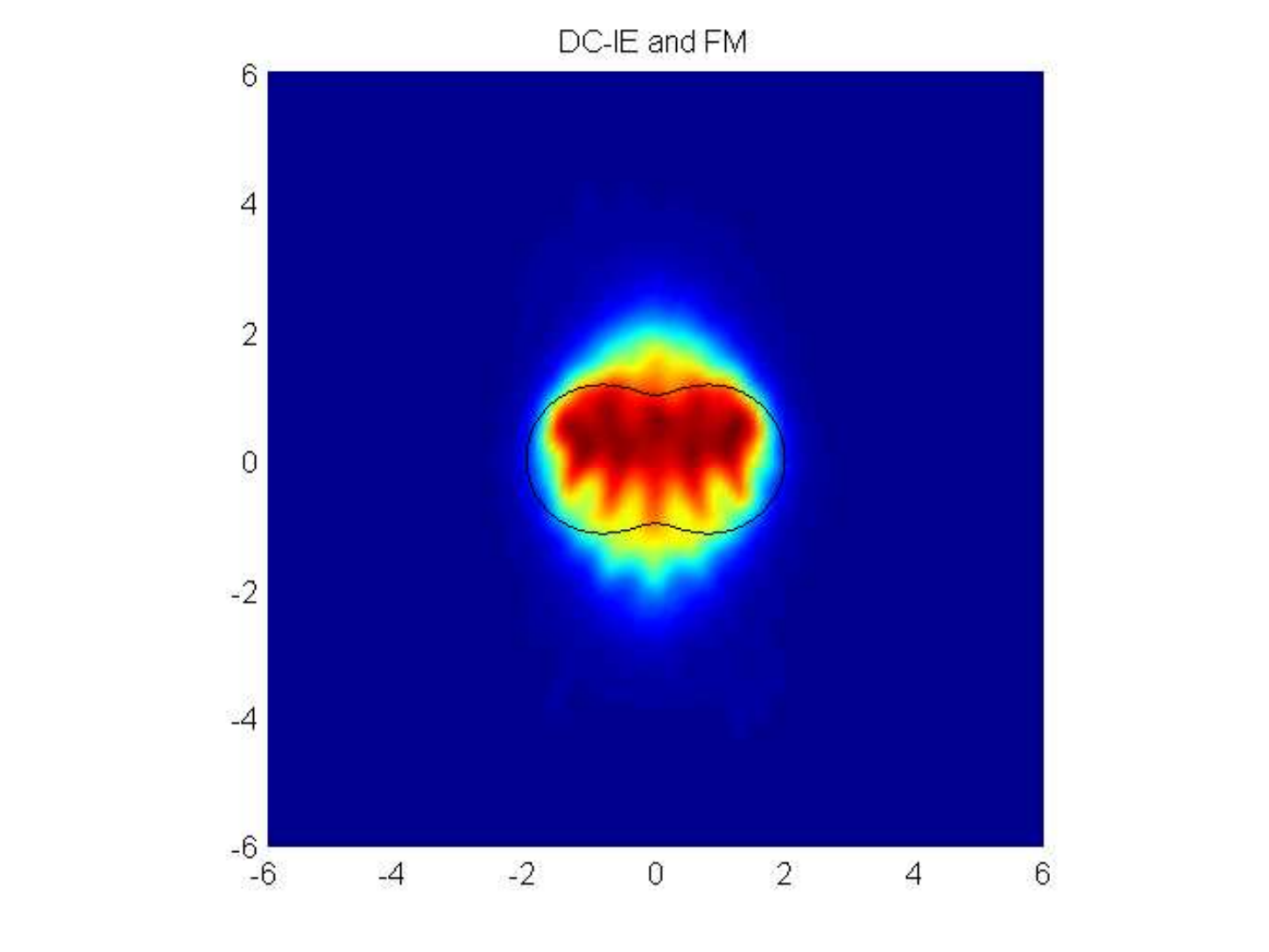}}
\caption{\bf Reconstructions of the sound-soft peanut.}
\label{Peanut-D}
\end{figure}

\begin{figure}[htbp]
  \centering
 \subfigure[\textbf{DSM using DC-IE with TSVD}]{
    \includegraphics[width=2in]{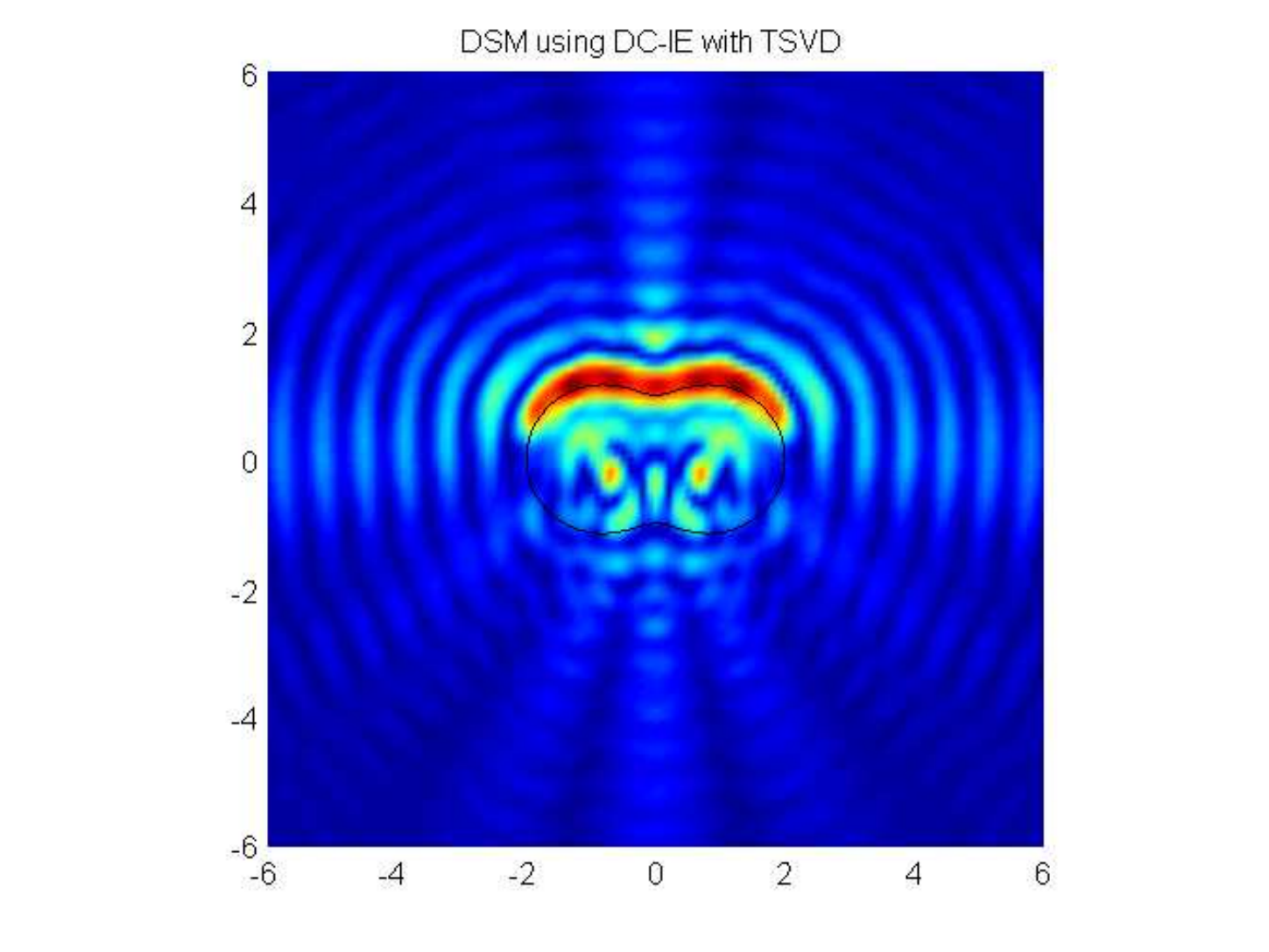}}\qquad
  \subfigure[\textbf{FM using DC-IE with TSVD}]{
    \includegraphics[width=2in]{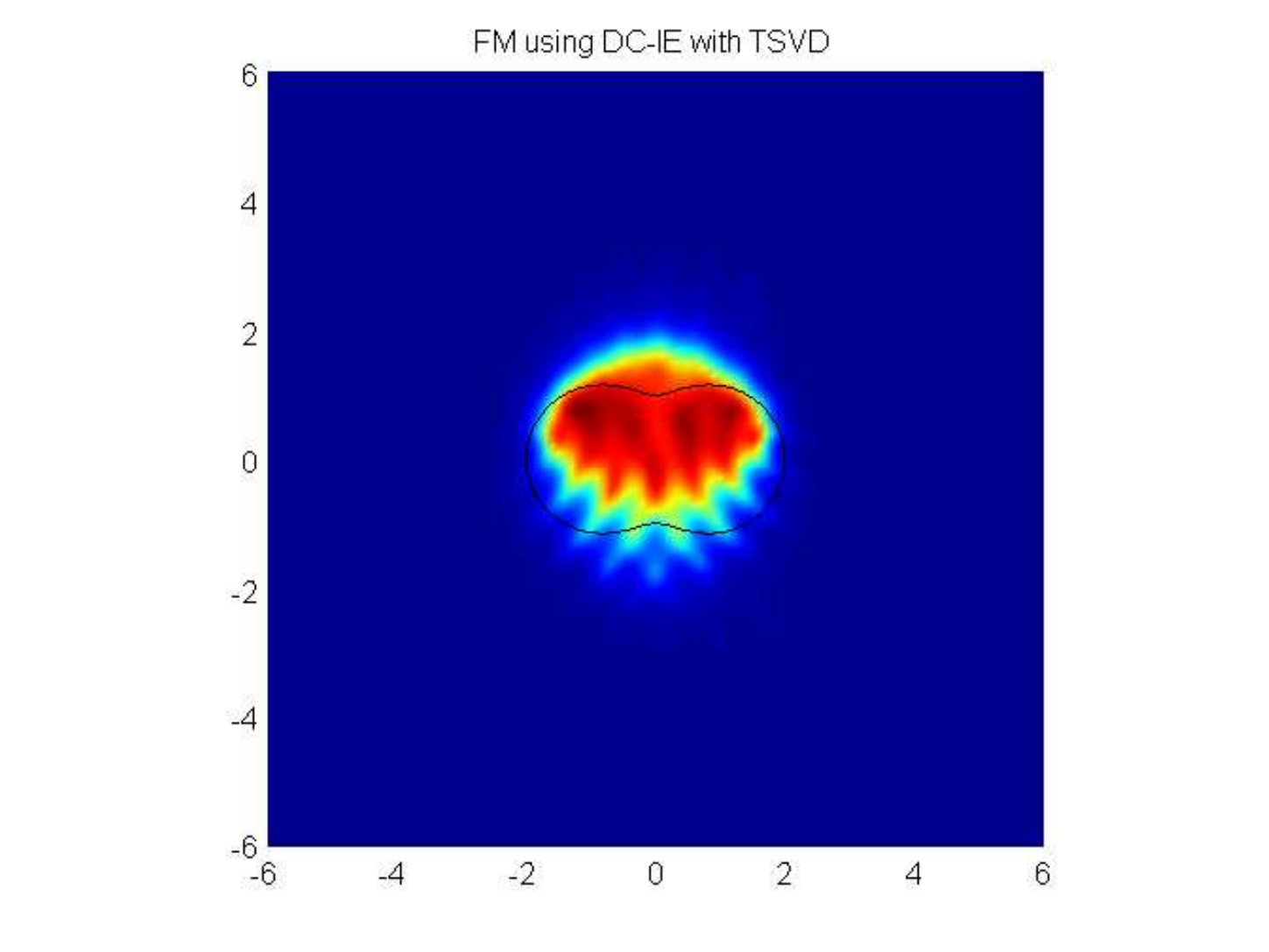}}\\
     \subfigure[\textbf{DSM using DC-IE with Tikhonov}]{
    \includegraphics[width=2in]{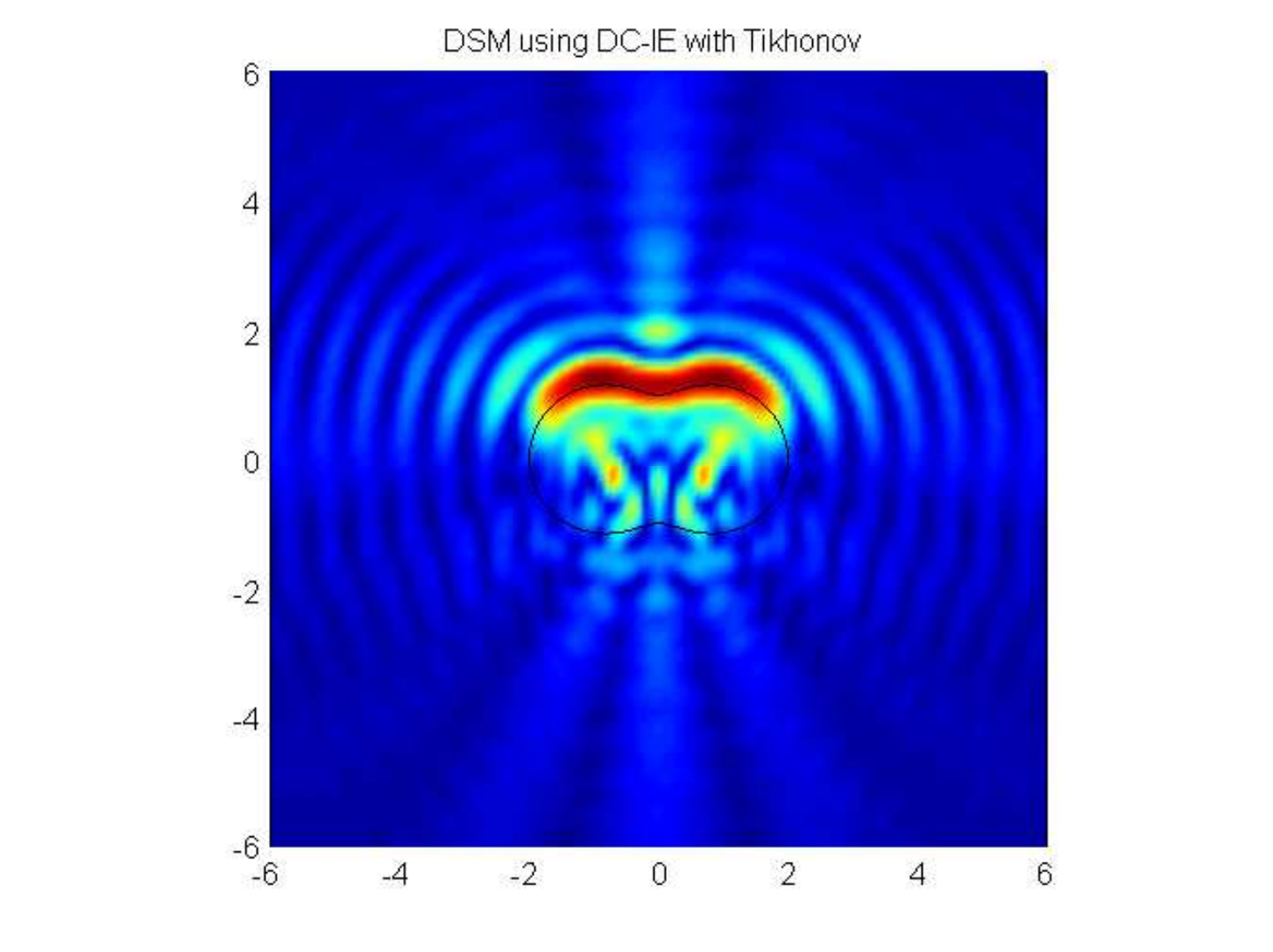}}\qquad
  \subfigure[\textbf{FM using DC-IE with Tikhonov}]{
    \includegraphics[width=2in]{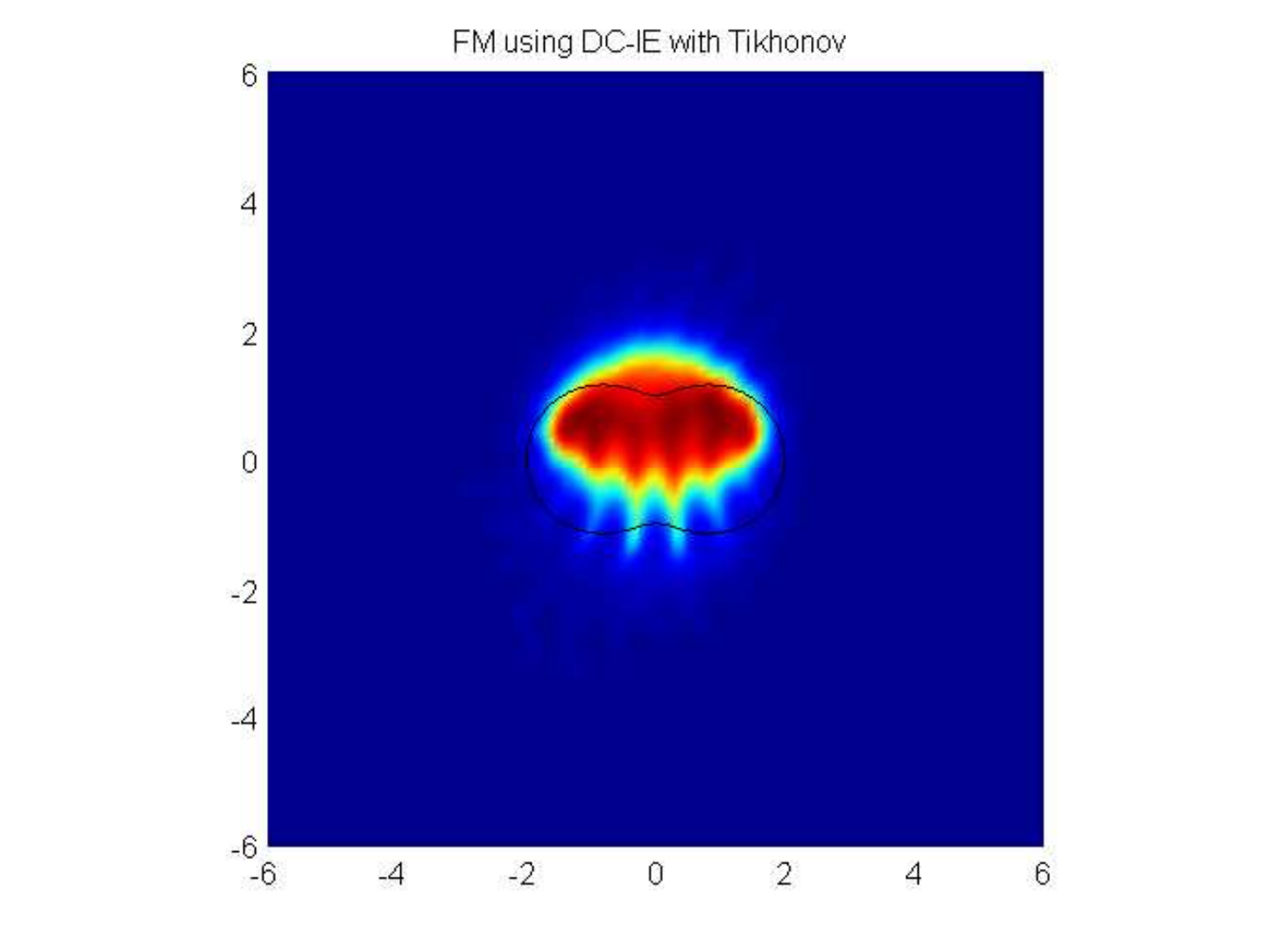}}
\caption{\bf Reconstructions of the sound-soft peanut using TSVD and Tikhonov regularization.}
\label{Peanut-D-2}
\end{figure}
\begin{figure}[htbp]
  \centering
\subfigure[\textbf{DSM using partial data directly}]{
    \includegraphics[width=2in]{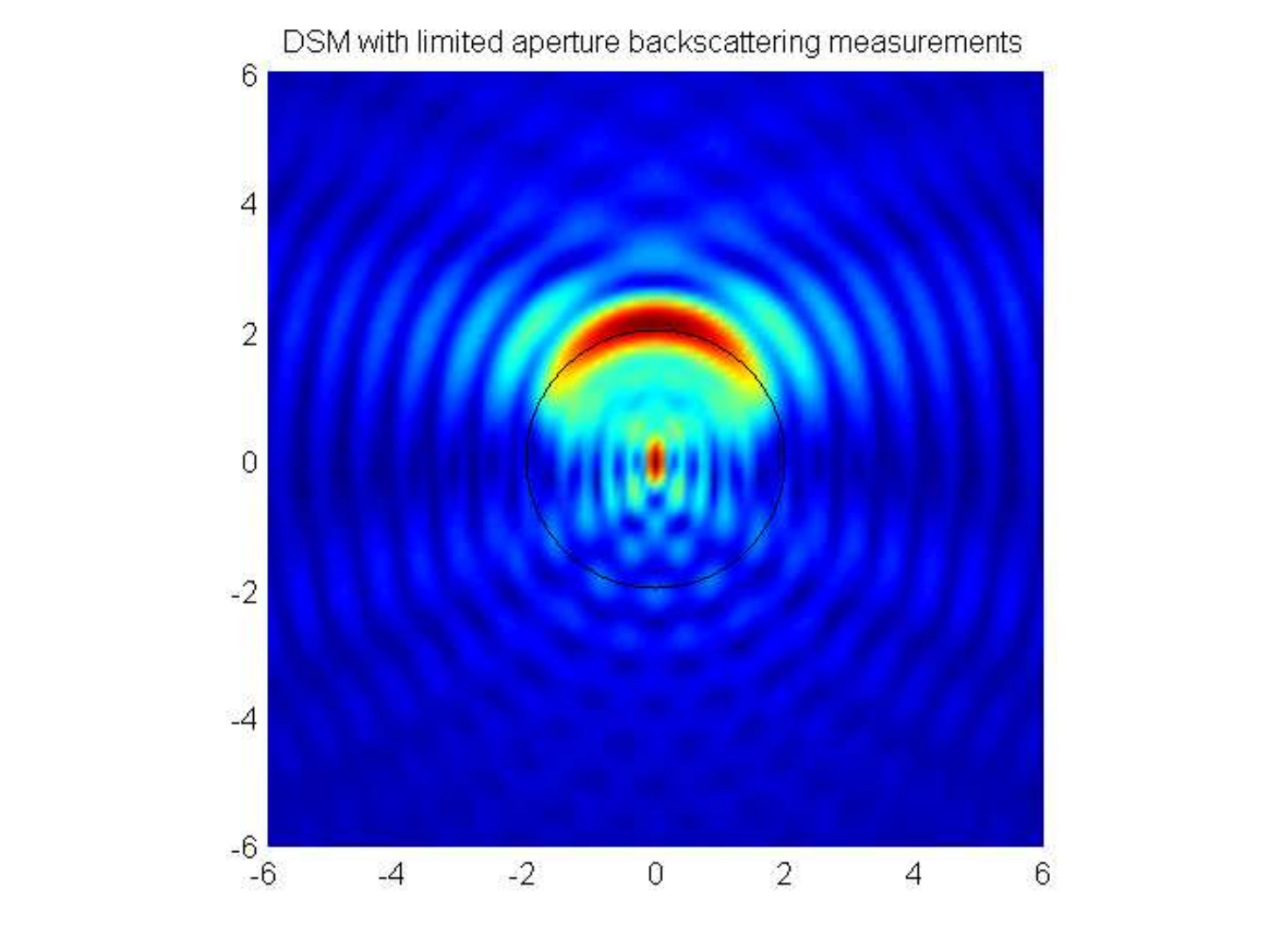}}
  \subfigure[\textbf{DSM with DC-FS}]{
    \includegraphics[width=2in]{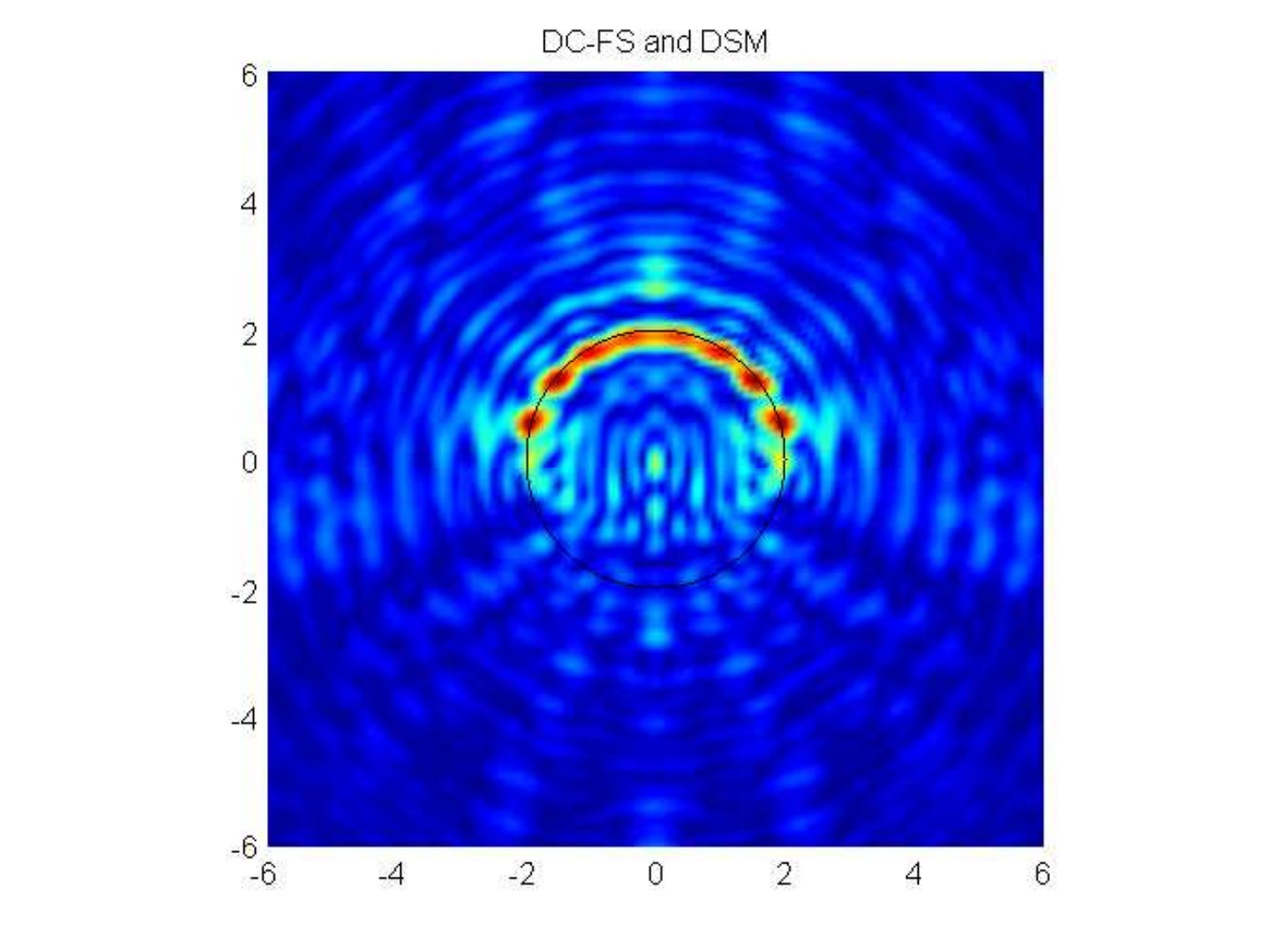}}\\
  \subfigure[\textbf{DSM with DC-IE}]{
    \includegraphics[width=2in]{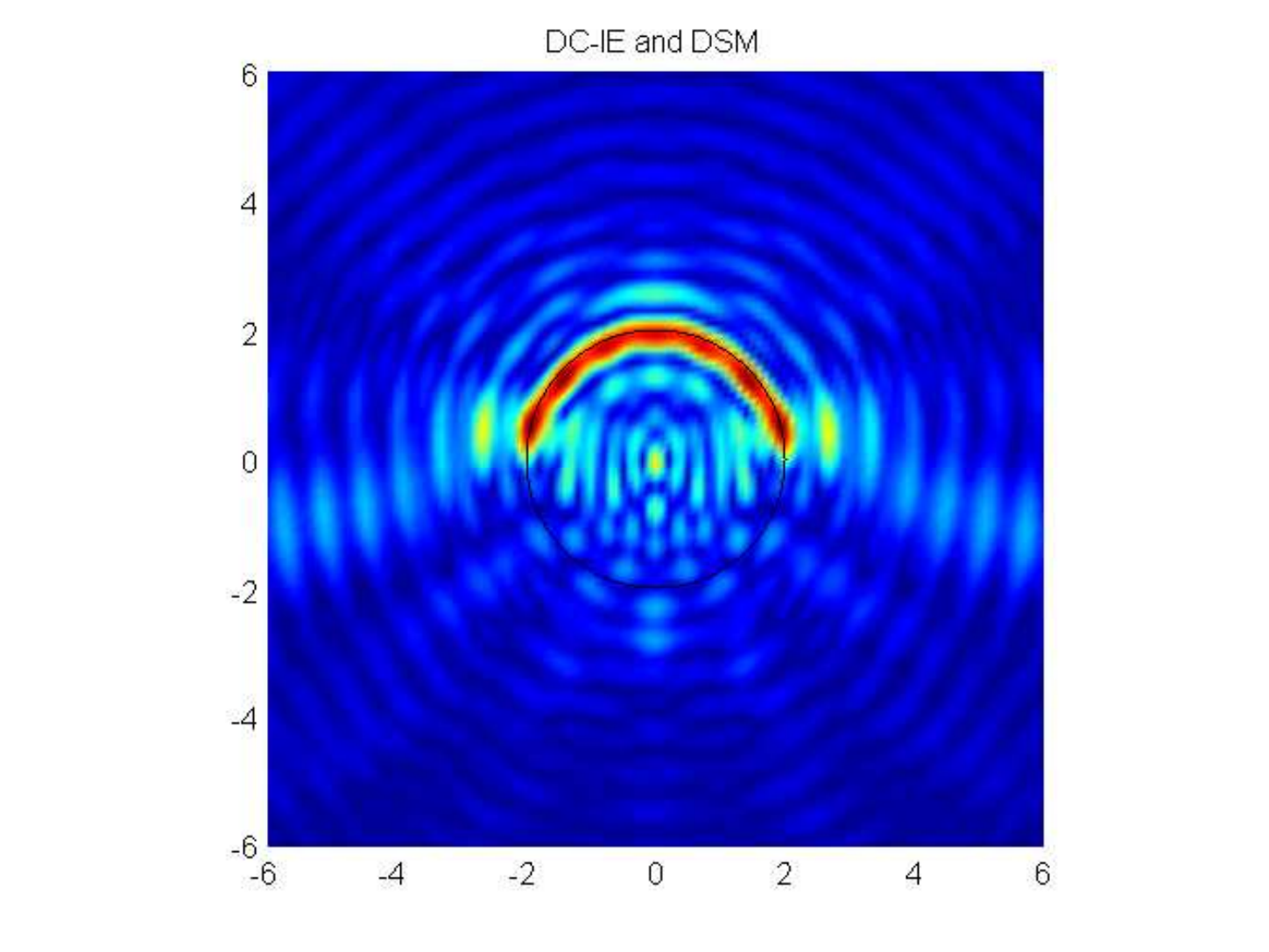}}
  \subfigure[\textbf{FM with DC-IE}]{
    \includegraphics[width=2in]{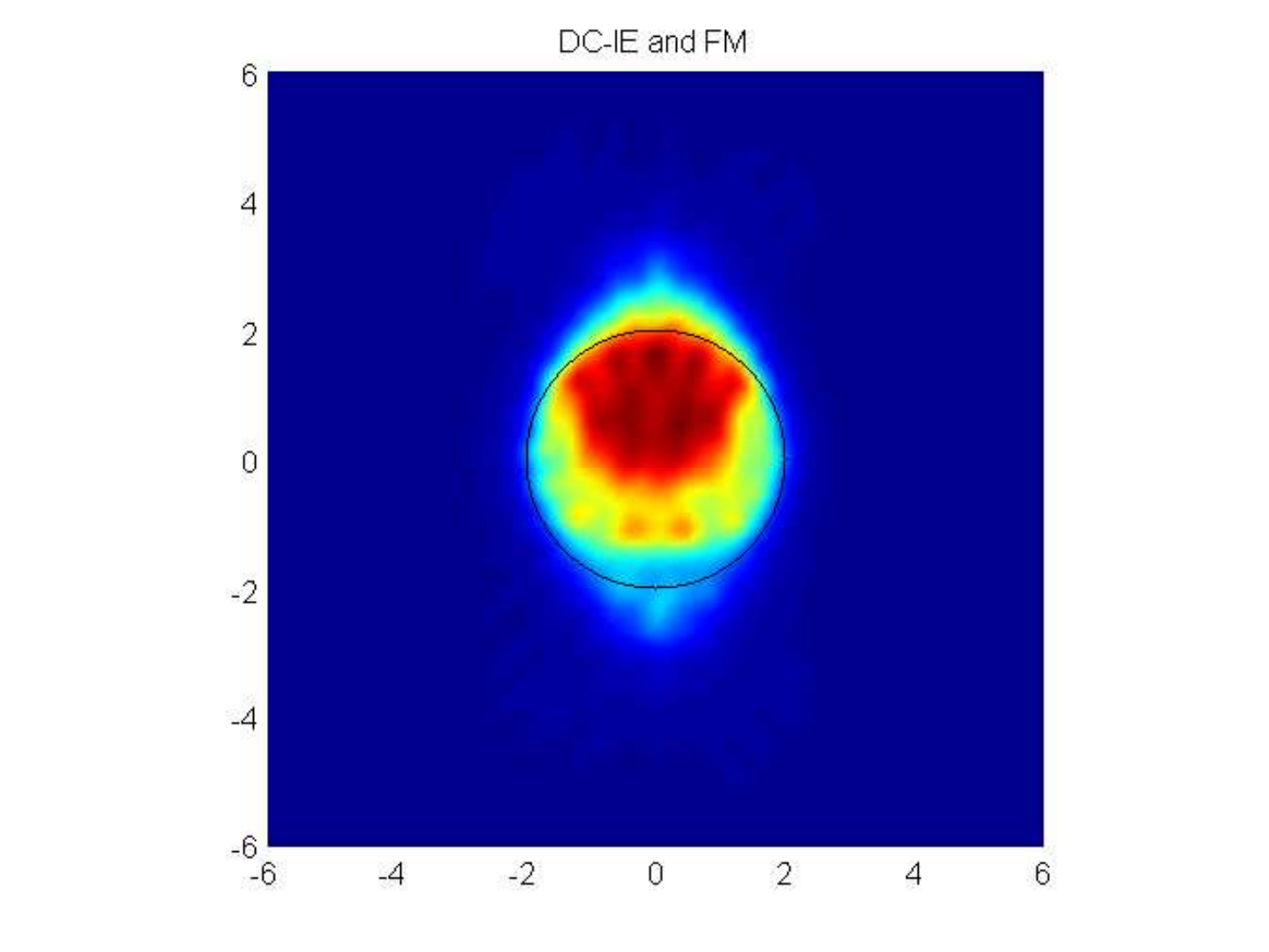}}
\caption{\bf Reconstructions of the sound-soft disk.}
\label{Disk-D}
\end{figure}

\begin{figure}[htbp]
  \centering
\subfigure[\textbf{DSM using partial data directly}]{
    \includegraphics[width=2in]{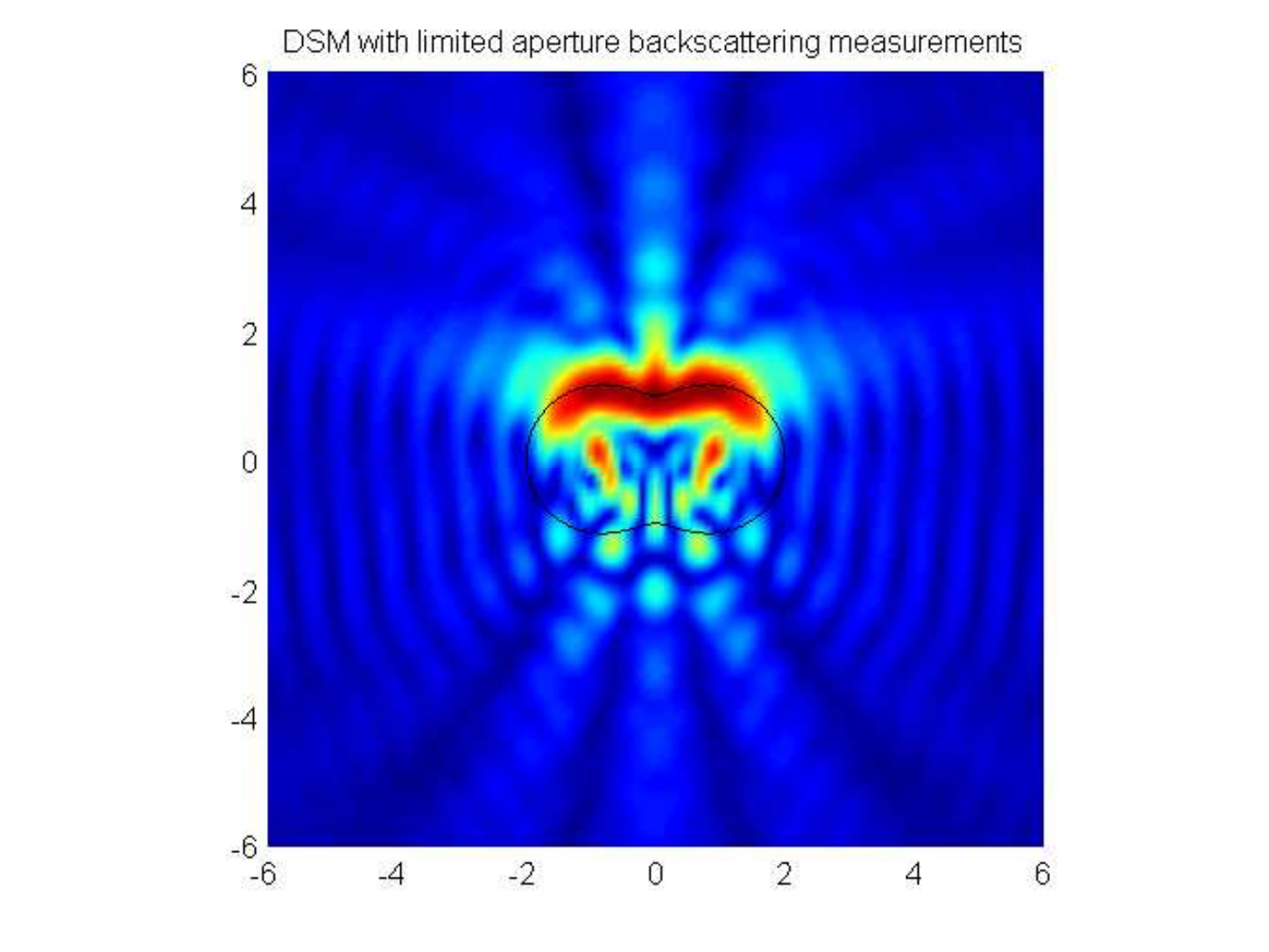}}
  \subfigure[\textbf{DSM with DC-FS}]{
    \includegraphics[width=2in]{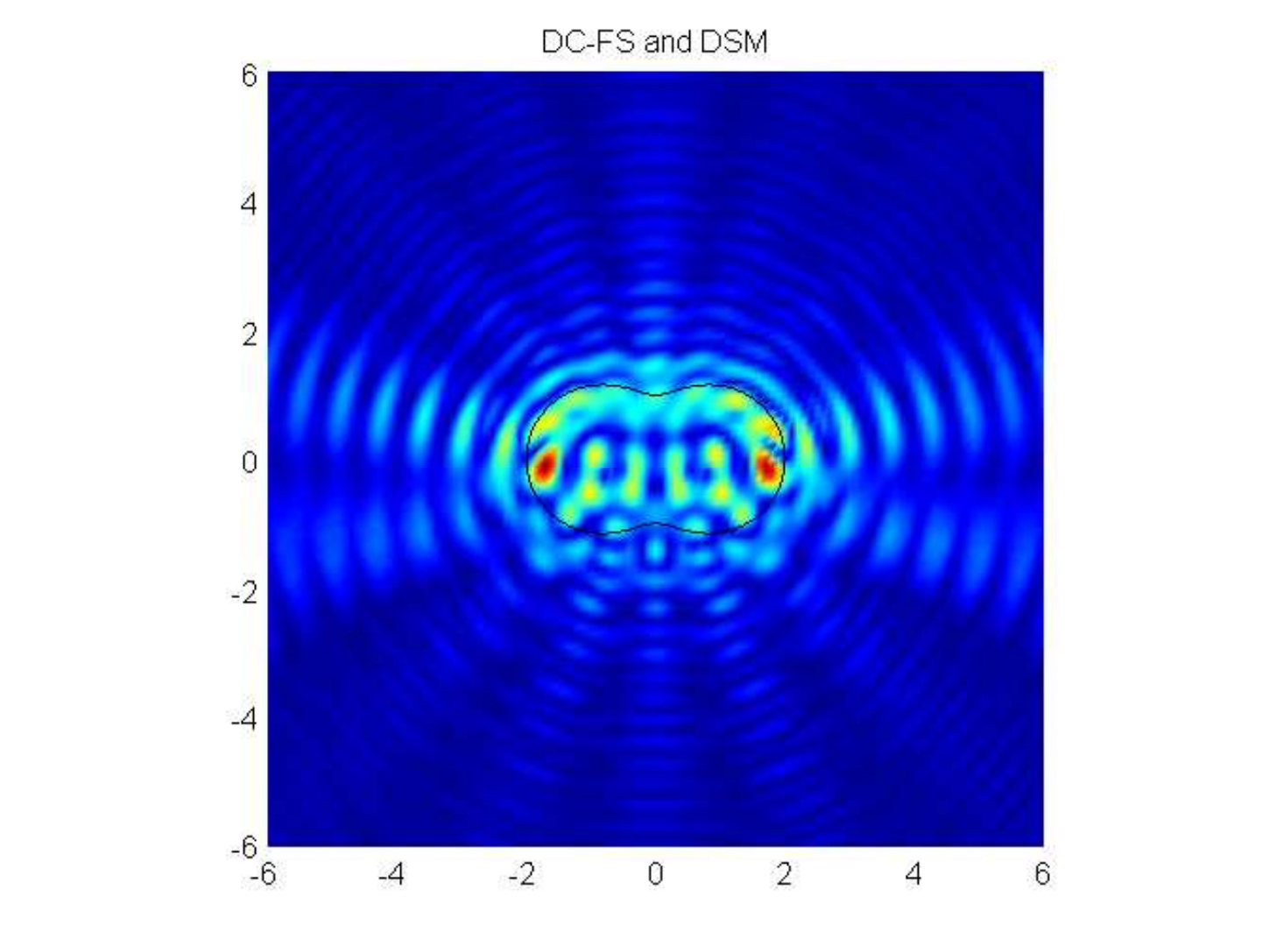}}\\
  \subfigure[\textbf{DSM with DC-IE}]{
    \includegraphics[width=2in]{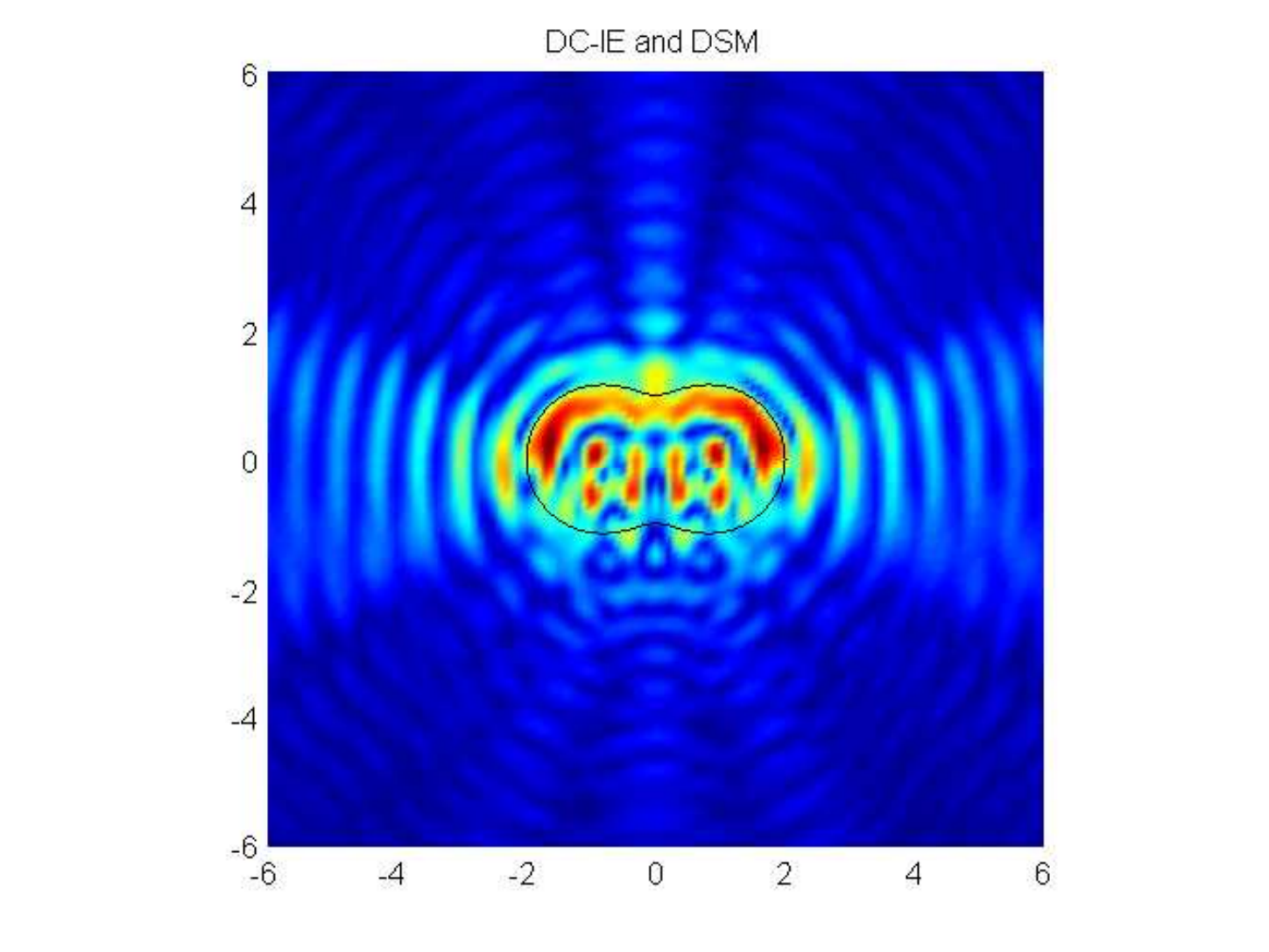}}
  \subfigure[\textbf{FM with DC-IE}]{
    \includegraphics[width=2in]{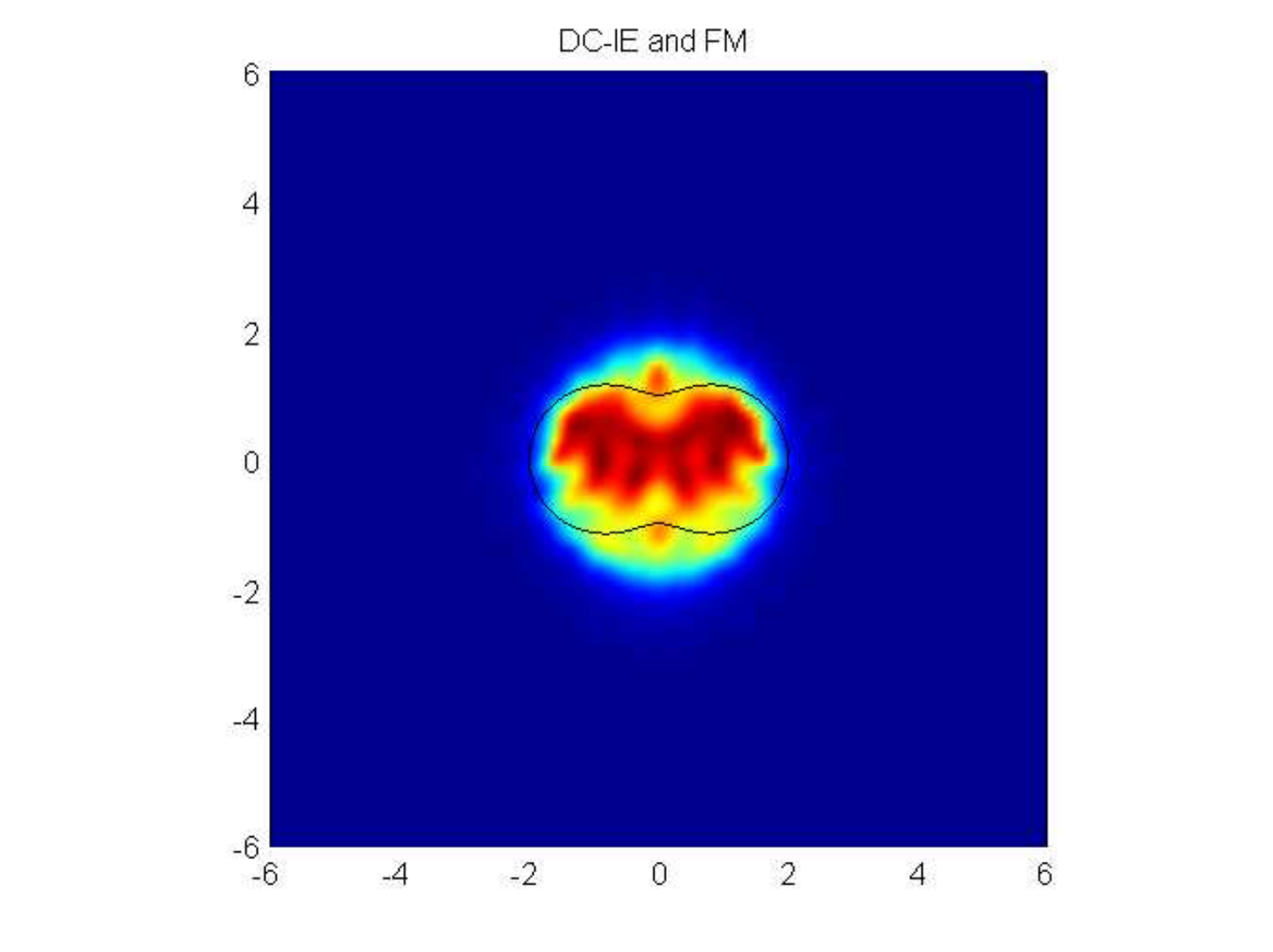}}
\caption{\bf Reconstructions of the sound-hard peanut.}
\label{Peanut-N}
\end{figure}

\begin{figure}[htbp]
  \centering
\subfigure[\textbf{DSM using partial data directly}]{
    \includegraphics[width=2in]{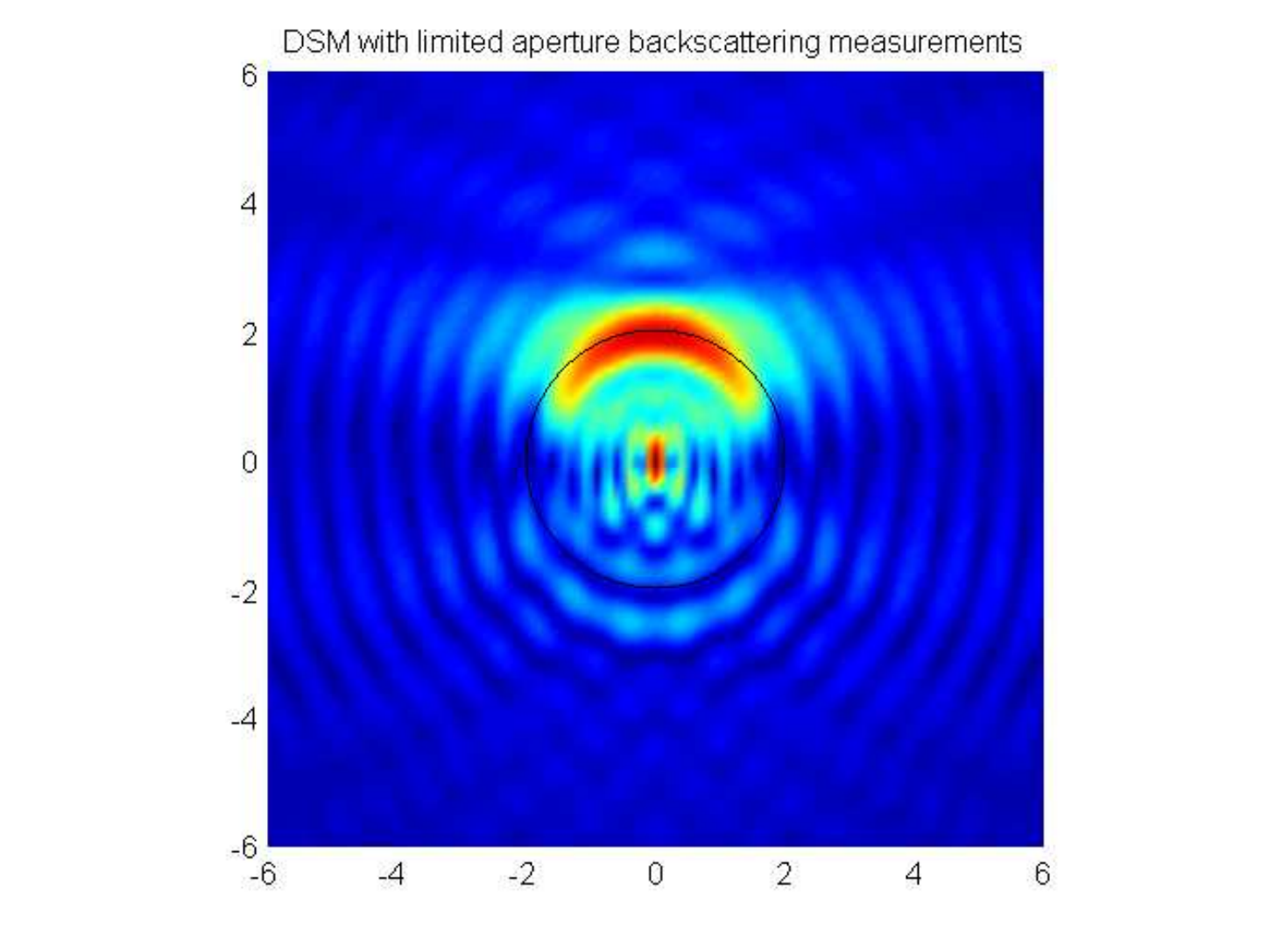}}
  \subfigure[\textbf{DSM with DC-FS}]{
    \includegraphics[width=2in]{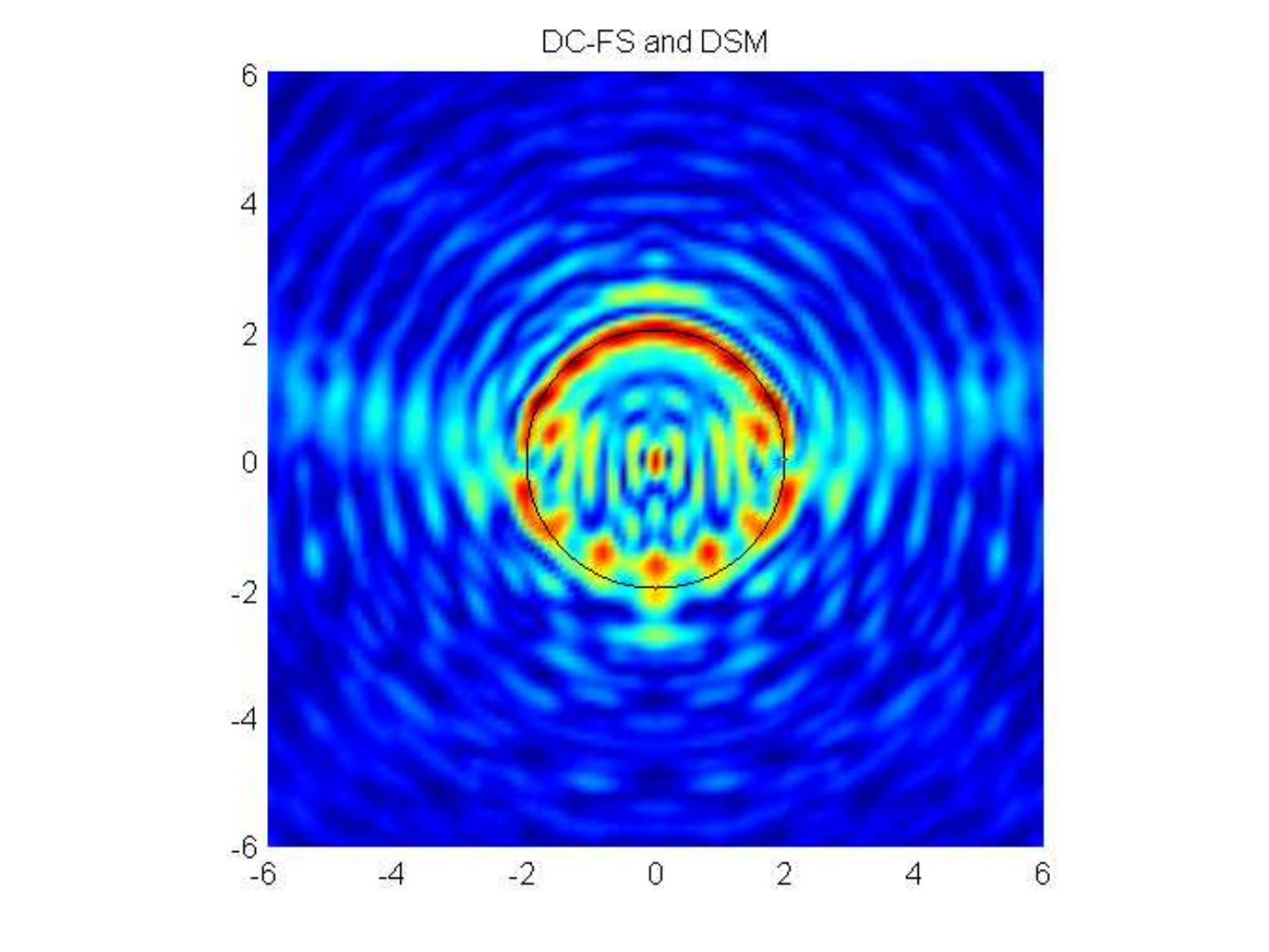}}\\
  \subfigure[\textbf{DSM with DC-IE}]{
    \includegraphics[width=2in]{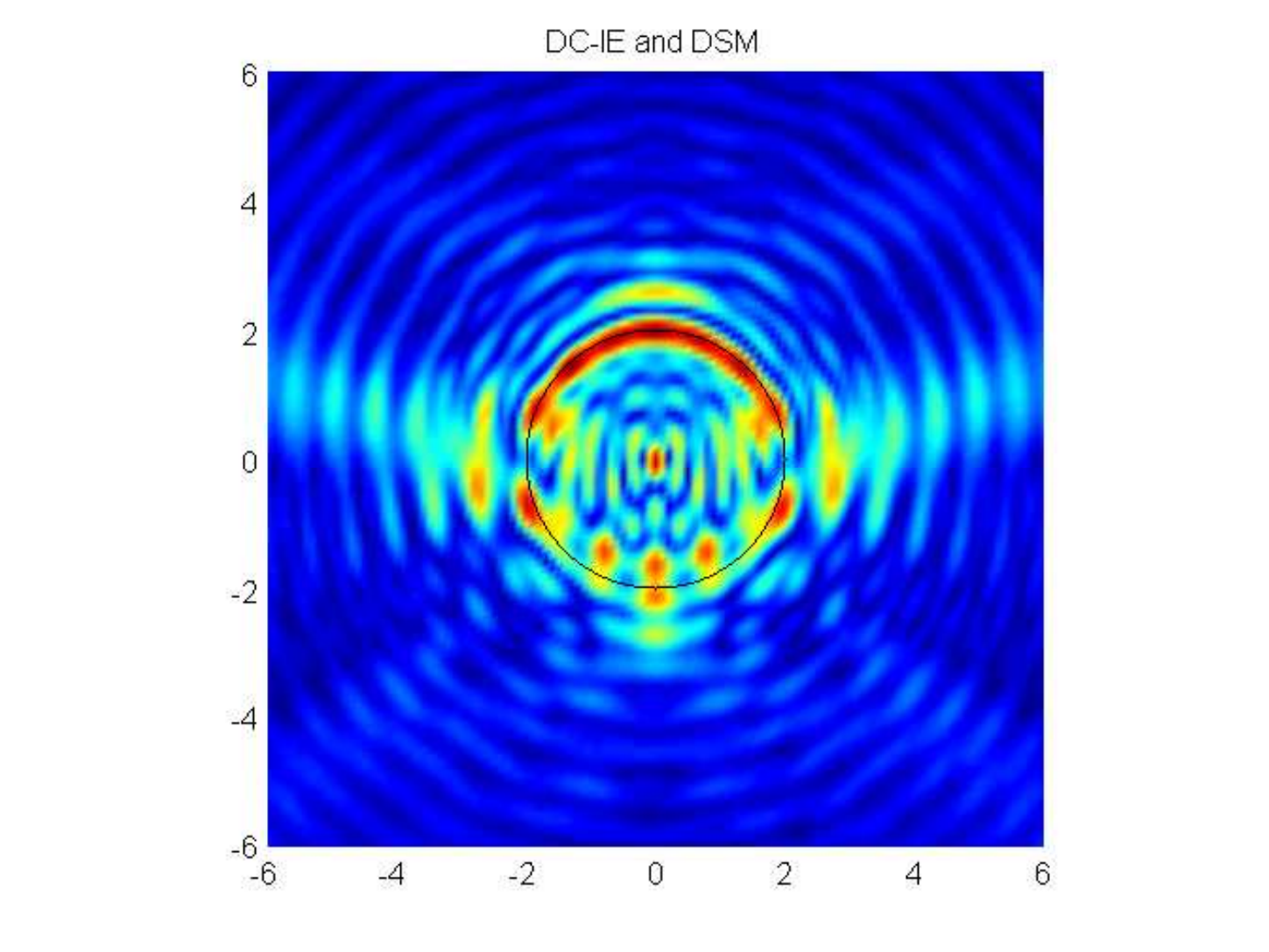}}
  \subfigure[\textbf{FM with DC-IE}]{
    \includegraphics[width=2in]{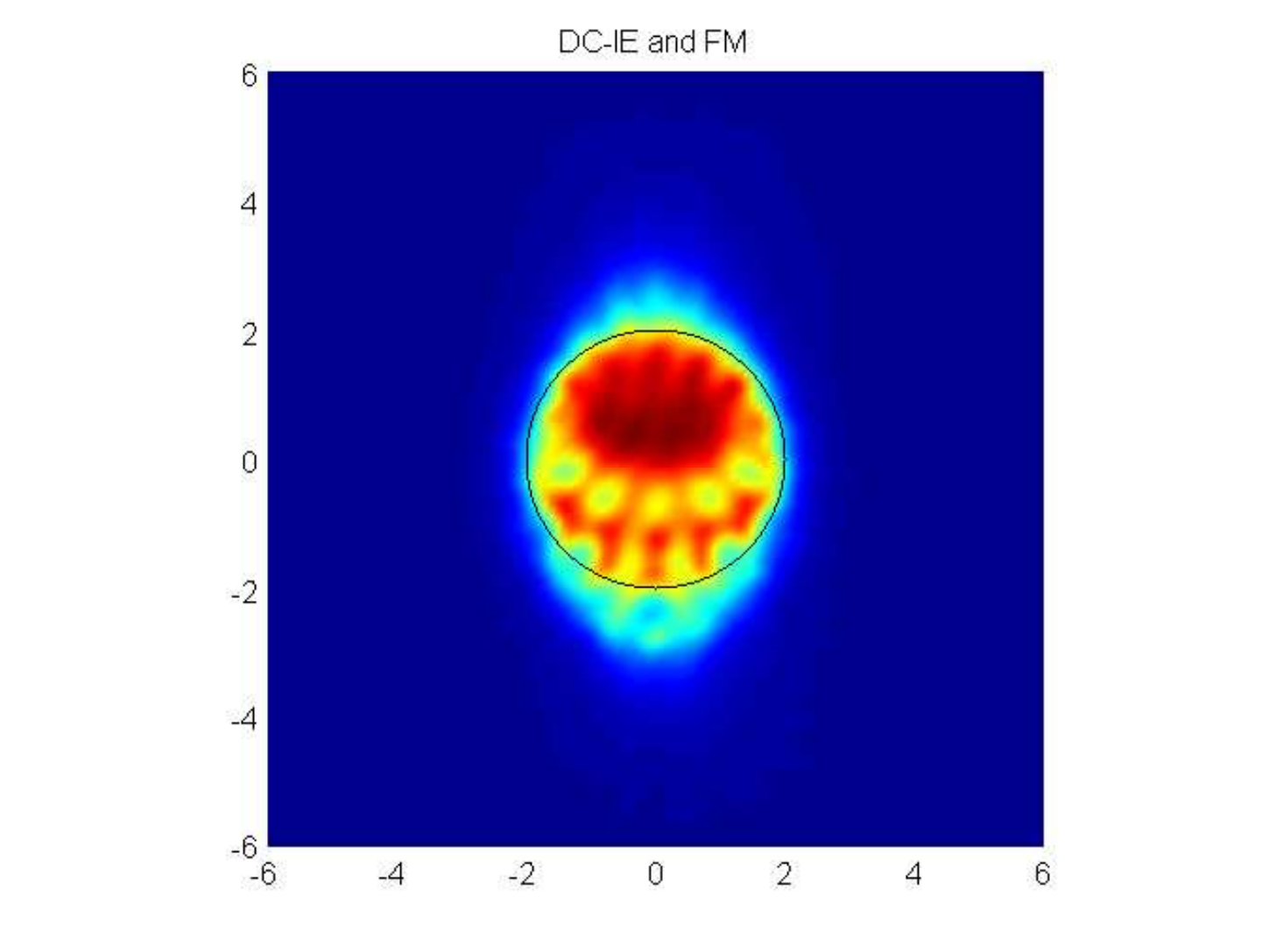}}
\caption{\bf Reconstructions of the sound-hard disk.}
\label{Disk-N}
\end{figure}

We conclude this section with some remarks. We have proposed two simple and fast data completion algorithms {\bf DC-FS} and {\bf DC-IE}. Both  algorithms involve the inversion of the highly ill-conditioned prolate matrix, which  indicate the ill-posedness nature of analytic continuation. The choice of $J$ and the regularization methods play an important role in the data completion algorithms. We have  illustrated the potential of the data completion algorithms and sampling methods.
Broadly speaking, the data completion algorithms are also applicable to many other imaging methods  with the reconstructed full-aperture data. Furthermore, it might still be possible to regularize the inverse of the prolate matrix in some other ways to potentially enhance the performance of the data completion algorithms and to  facilitate the corresponding resolution analysis.

\section*{Acknowledgement}
The research of F. Dou is supported by the NSFC (No. 12071061,11971093), the Applied Fundamental Research Program of Sichuan Province (No. 2020YJ0264), the Fundamental Research Funds for the Central Universities (No. ZYGX2019J094) and the Science  Strength Promotion Programme of UESTC. The research of X. Liu is supported by the NNSF of China grant 11971471 and the Youth Innovation Promotion Association, CAS. The research of B. Zhang is partially supported by the NNSF of China grant 91630309.

%

\bibliographystyle{SIAM}


\begin{thebibliography}{99}

\bibitem{AhnJeonMaPark} C.Y. Ahn, K. Jeon, Y.K. Ma and W.K. Park,
	A study on the topological derivative-based imaging of thin electromagnetic inhomogeneities in limited-aperture problems,
	{\em Inverse Problems \bf30}, (2014), 105004.

\bibitem{Atkinson} D. Atkinson,
	Analytic extrapolations and inverse problems,
	Applied Inverse Problems (Lecture Notes in Physics 85) ed P. C. Sabatier, (Berlin: Springer), (1978), 111-121.

\bibitem{BaoLiu} G. Bao and J. Liu,
	Numerical solution of inverse problems with multi-experimental limited aperture data,
	{\em SIAM J.Sci.Comput. \bf25}, (2003), 1102-1117.
	
	  \bibitem{BORCEA2019556}
{L.~Borcea, F.~Cakoni, and S.~Meng},  {A direct approach to imaging
  in a waveguide with perturbed geometry}, {\em J. Comput. Phys.},
  {\bf 392} (2019), 556--577.

	\bibitem{Borcea_2009}
{L.~Borcea, T.~Callaghan, J.~Garnier, and G.~Papanicolaou}, {A
  universal filter for enhanced imaging with small arrays}, {\em Inverse Problems},
 {\bf  26} (2009), 015006.

  \bibitem{CB2009} M. Cheney and B. Borden,
  Fundamentals of Radar Imaging, CBMS-NSF Regional Conf. Ser. Appl.
Math. 79, SIAM, Philadelphia, 2009.


\bibitem{CC2014} F. Cakoni and D. Colton,
	A Qualitative Approach in Inverse Scattering Theory,
	AMS Vol.188, Springer-Verlag, 2014.

\bibitem{ChengPengYamamoto2005IP} J. Cheng, L. Peng, and M. Yamamoto,
	The conditional stability in line unique continuation for a wave equation and an inverse wave source problem,
	{\em Inverse Problems \bf21}, (2005), 1993-2007.

\bibitem{ChengYamamoto1998IP} J. Cheng and M. Yamamoto,
	Unique continuation on a line for harmonic functions,
	{\em Inverse Problems \bf14}, (1998), 869-882.

\bibitem{CK} D. Colton and R. Kress,
	Inverse Acoustic and Electromagnetic Scattering Theory (Third Edition), Springer, 2013.

\bibitem{ColtonMonk06} D. Colton and P. Monk,
	Target identification of coated objects
 {\em IEEE Trans. Antennas Propagat., \bf54}, (2006), 1232-1242.

\bibitem{FuDouFengQian}
C. Fu, F. Dou, X. Feng and Z. Qian,
A simple regularization method for stable analytic continuation,
{\em Inverse Problems \bf24}, (2008), 065003.

\bibitem{FuDengFengDou}
C. Fu, Z. Deng, X. Feng and F. Dou,
A modified Tikhonov regularization for stable analytic continuation,
{\em SIAM J. Numer. Anal. \bf47(4)}, (2009), 2982-3000.

\bibitem{GrenanderSzego}
U. Grenander and G. Szeg\"{o},
Toeplitz Forms and Their Applications,
Univ. of California Press, Berkeley, 1958.

\bibitem{IkehataNiemiSiltanen} M. Ikehata, E. Niemi and S. Siltanen,
	Inverse obstacle scattering with limited-aperture data,
	{\em Inverse Probl. Imaging \bf1}, (2012), 77-94.

\bibitem{JiLiuXi}
{X. Ji, X. Liu and Y. Xi},
{Direct sampling methods for inverse elastic scattering problems},
{\em Inverse Problems} {\bf 34} (2018), 035008.

\bibitem{Kirsch98} A. Kirsch,
	Characterization of the shape of a scattering obstacle using the spectral data of the far field operator,
	{\em Inverse Problems \bf14}, (1998), 1489-1512.

\bibitem{KirschGrinberg} A. Kirsch and N. Grinberg,
	The Factorization Method for Inverse Problems, Oxford University Press, 2008.

\bibitem{L4} J. Li, P. Li, H. Liu and X. Liu,
	Recovering multiscale buried anomalies in a two-layered medium,
	{\em Inverse Problems \bf31}, (2015), 105006.

\bibitem{LiuIP17} X. Liu,
	A novel sampling method for multiple multiscale targets from scattering amplitudes at a fixed frequency,
	{\em Inverse Problems \bf33}, (2017), 085011.

\bibitem{LiuSun19}X. Liu and J. Sun,
    Data recovery in inverse scattering problems: from limited-aperture to full-aperture,
    {\em J. Comput. Phys. \bf386(1)}, (2019), 350-364.

\bibitem{LuXuXu2012AA} S. Lu, B. Xu and X. Xu,
	Unique continuation on a line for the Helmholtz equation.
	{\em Appl. Anal. \bf91(9)}, (2012), 1761-1771.

\bibitem{MagerBleistein} R.D. Mager and N. Bleistein,
	An approach to the limited aperture problem of physical optics far field inverse scattering,
	Tech. Report Ms-R-7704, University of Denver, Denver, CO, 1977.

\bibitem{MagerBleistein1978} R.D. Mager and N. Bleistein,
	An examination of the limited aperture problem of physical optics inverse scattering, {\em IEEE Trans.
	 Antennas Propag. \bf26}, (1978), 695-699.

\bibitem{Robert1987} R.L. Ochs, Jr.,
	The limited aperture problem of inverse acoustic scattering: Dirichlet boundary conditions,
	{\em SIAM J. Appl. Math. \bf47(6)}, (1987), 1320-1341.

\bibitem{Slepian78}
D. Slepian,
Prolate spheroidal wave functions, Fourier analysis, and uncertainty V: The discrete case,
{\em Bell System Tech. J. \bf57}, (1978), 1371-1430.


\bibitem{Varah1993}
J. M. Varah,
The prolate matrix,
{\em Linear Algebra Appl. \bf187}, (1993), 269-278.

\bibitem{Zinn1989}
A. Zinn,
On an optimisation method for the full- and limited-amperture problem in inverse acoustic scattering for a sound-soft obstacle,
{\em Inverse Problems \bf5}, (1989), 239-253.

\end{thebibliography}

\end{document}